\documentclass[a4paper]{article}

\author{Jan Bouwe van den Berg\thanks{Department of Mathematics, Vrije Universiteit Amsterdam, The Netherlands, \href{mailto:janbouwe@few.vu.nl}{janbouwe@few.vu.nl}; partially supported by NWO-VICI grant 639033109},  Wouter Hetebrij\thanks{Department of Mathematics, Vrije Universiteit Amsterdam, The Netherlands, \href{mailto:w.a.hetebrij@vu.nl}{w.a.hetebrij@vu.nl}.} \ and Bob Rink\thanks{Department of Mathematics, Vrije Universiteit Amsterdam, The Netherlands, \href{mailto:b.w.rink@vu.nl}{b.w.rink@vu.nl}.}}
\title{The parameterization method for center manifolds}
\date{}

\usepackage[english]{babel}

\usepackage{amsmath}
\usepackage{amssymb} 
\usepackage{amsthm}

\usepackage{enumitem}
\usepackage{mathtools}

\usepackage{hyperref}
\hypersetup{
	pdfborder={0 0 0},
	pdftoolbar = true,
	pdfmenubar = true,
}

\usepackage{amsrefs}
\usepackage[capitalize,noabbrev]{cleveref}
	\crefname{equation}{}{}
	\crefname{enumi}{}{}

\usepackage{multirow}
\usepackage{longtable}
\usepackage{arydshln}

\newcounter{Rest}

\newtheorem{theorem}[Rest]{Theorem}
\newtheorem*{theorem*}{Theorem}
\newtheorem{corollary}[Rest]{Corollary}
\newtheorem*{corollary*}{Corollary}
\newtheorem{lemma}[Rest]{Lemma}
\newtheorem*{lemma*}{Lemma}
\newtheorem{proposition}[Rest]{Proposition}
\theoremstyle{remark}
\newtheorem{remark}[Rest]{Remark}
\newtheorem*{remark*}{Remark}

\numberwithin{equation}{subsection}
\numberwithin{Rest}{section}

\newcommand{\supnorm}[1]{\| #1 \|_0}
\newcommand{\onenorm}[1]{\| #1 \|_{1}}
\newcommand{\operatornorm}[1]{\| #1 \|_{\textup{op}}}
\newcommand{\norm}[1]{\| #1 \|}
\newcommand{\bignorm}[1]{\left\| #1 \right\|}

\newcommand{\tmatrix}[3]{\left( \begin{smallmatrix} #1 \\ #2 \\ #3 \end{smallmatrix} \right)}

\newcommand{\isdef}{:=}
\newcommand{\isfed}{=:}

\usepackage{subcaption}
\usepackage{dcpic,pictex}

\begin{document}

\maketitle

\begin{abstract}

In this paper, we present a generalization of the parameterization method, introduced by Cabr\'{e}, Fontich and De la Llave, to center manifolds associated to non-hyperbolic fixed points of discrete dynamical systems. As a byproduct, we find a new proof for the existence and regularity of center manifolds. However, in contrast to the classical center manifold theorem, our parameterization method will simultaneously obtain the center manifold and its conjugate center dynamical system. Furthermore, we will provide bounds on the error between approximations of the center manifold and the actual center manifold, as well as  bounds for the error in the conjugate dynamical system.

\end{abstract}

\section{Introduction}

Cabr{\'e}, Fontich and De la Llave introduced the parameterization method for invariant manifolds of dynamical systems on Banach spaces in \cites{Cabre03,Cabre03-2,Cabre05}. The goal of the parameterization method is to find a parameterization of the (un)stable manifold associated with an equilibrium point of the dynamical system. This parameterization is defined as a conjugacy between a dynamical system on the (un)stable eigenspace and the original dynamical system on the Banach space. To find a unique conjugacy, the dynamical system on the (un)stable eigenspace is fixed (it is given  as a polynomial map). Besides providing a new proof of the classical (un)stable manifold theorem, the method is useful for computational existence proofs of for example homoclinic and heteroclinic orbits, see \cites{VandenBerg11,Lessard14,VandenBerg15-2}. The method has also been used for constructing (un)stable manifolds of periodic orbits, see \cite{Castelli18}. Recent computational advances include applications to delay differential equations, see \cite{Groothedde17}, and partial differential equations, see \cite{Reinhardt19}.

 The goal of this paper is to generalize the parameterization method to center manifolds.  As explained above, the proof of the classical parameterization method fixes the dynamics on the (un)stable eigenspace. Having explicit dynamics on the (un)stable eigenspace a priori ensures that the parameterization of the manifold is unique. For center manifolds, we generally cannot choose the dynamics on the center eigenspace ourselves. Thus besides solving for the conjugacy, we also need to solve for the dynamics on the center eigenspace. To ensure a unique parameterization, we fix part of the conjugacy instead of fixing the conjugate dynamics.

To state the main theorem of this paper, we introduce some notation. Let $F{}{}:X \to X$ be a $C^n$ (for $n \ge 2$) discrete time dynamical system on a Banach space $X$, with non-hyperbolic fixed point $0$. Denote $F{}{}(x) = Ax + g(x) $, with $A$ the Jacobian of $F{}{}$ at $0$. We assume that $A$ has a spectral gap around the unit circle, and write $X = X_c \oplus X_h$, where $X_c$ is the eigenspace associated with the eigenvalues on the unit circle, and $X_h$ the eigenspace associated with the eigenvalues away from the unit circle. We write $A = \tmatrix{ A_c & 0 }{0 & A_h}{}$ with respect to the splitting $X = X_c \oplus X_h$. Our goal is to find a conjugacy $K{}{} :X_c \to X$ and a dynamical system $R : X_c \to X_c$ such that the diagram

\[
\begin{aligned}
\begindc{\commdiag}[50]
\obj(20,15)[1]{$X = X_c \oplus X_h$}
\obj(60,15)[2]{$X = X_c \oplus X_h$}

\obj(20,0)[3]{$X_c$}
\obj(60,0)[4]{$X_c$}

\mor{1}{2}{$F{}{} = A + g$}
\mor{3}{1}{$K{}{} = \iota + \left( \substack{  k_c \\ k_h} \right) $}[-1,0]
\mor{4}{2}{$K{}{} = \iota +  \left( \substack{ k_c \\ k_h} \right) $}[1,0]
\mor{3}{4}{$R = A_c + r$}

\enddc
\end{aligned}
\]
commutes, where $\iota:X_c \to X_c \oplus X_h$ is the inclusion map. We impose that the conjugacy $K{}{}$ is tangent to $X_c$ at $0$, i.e. $DK{}{}(0) = \tmatrix{\operatorname{Id}}{0}{}$, which implies that the linearization of $R$ at $0$ is given by $A_c$. We also fix part of the conjugacy $K{}{}$ by an explicit choice of the non-linear map $k_c : X_c \to X_c$. Our result is the following.

\begin{theorem*}[Parameterization of the center manifold] If $g:X \to X$ and $k_c: X_c \to X_c$ are $C^n$, sufficiently small in the $C^1$ norm and bounded in $C^2$ norm, see \cref{MainTheoremSection} for the explicit bounds, then there exist a $C^n$ map $k_h :  X_c \to  X_h$ and a $C^n$ map $r : X_c \to X_c$ such that 
\begin{align}
(A + g) \circ K{}{} = K{}{} \circ (A_c +r) && \text{ where }  K{}{} =\iota + \begin{pmatrix}  k_c \\ k_h \end{pmatrix}. \label{FirstEquation}
\end{align}
Furthermore, we have explicit bounds on the $C^1$ norms of $k_h$ and $r$ in terms of the spectral gap of $A$ and the $C^1$ norms of $g$ and $k_c$. These bounds are made precise in \cref{MainTheoremSection}.
\end{theorem*}\noindent
Despite $k_h$ and $r$ being $C^n$, we only give explicit bounds on their $C^1$ norm. We note that as $g$ is at least $C^2$, we may multiply $g$ as usual with a cut-off function $\xi$ such that $x \mapsto Ax + g (x)\xi(x)$ satisfies the conditions of \cref{MainTheorem}. Then we find a local center manifold for $x \mapsto Ax + g(x)$ on the region where $\xi = 1$, as well as local conjugate dynamics on the center eigenspace.

Our theorem provides a new proof of the classical center manifold theorem, see \cites{Carr82,Vanderbauwhede89,Schinas92}. The classical proof of the center manifold theorem is based on the variation of constants formula. Our proof is more elementary. It simultaneously finds $k_h$ and $r$ as fixed point of an operator on a function space, which is a contraction on an appropriately chosen subset. The freedom to choose $k_c \neq 0$ moreover allows us to obtain the dynamics on the center eigenspace directly in a desired normal form.

As we mentioned in the beginning, our work is inspired by the parameterization method for invariant (un)stable manifolds in \cites{Cabre03,Cabre03-2,Cabre05}, which we shall refer to as the classical parameterization method. The main difference between the classical and the proposed parameterization method is the treatment of the conjugate dynamics $R$. The classical parameterization method fixes a polynomial map $R :X_{u/s} \to X_{u/s}$ which determines the dynamical system on the (un)stable eigenspace, whereas our method produces a $C^n$ map $R: X_c \to X_c$. 

It follows from \cite{Farres10} that one can find Taylor approximations $R_0$  and $K{}{}_0$ of the maps $R$ and $K{}{}$ up to any finite order. These $R_0$ and $K{}{}_0$, will then almost satisfy the conjugacy equation, i.e. $F{}{} \circ K{}{}_0 - K{}{}_0 \circ R_0$ will be of order $\norm{x}^{m}$. This will imply that the real solution of the conjugacy equation lies very close to $R_0$ and $K{}{}_0$. More precisely, we can find rigorous bounds for $R - R_0$ and $K{}{} - K{}{}_0$. Hence, from the Taylor approximations we can extract very detailed information of the true dynamical behavior on the center manifold.

Another generalization of the classical parameterization method concerning center manifolds can be found in \cites{Baldoma16,Baldoma16-2}. This generalization finds certain submanifolds of the center manifold at parabolic fixed points. A fixed point is called parabolic if the linearization of the dynamical system is the identity. As in the classical parameterization method, the conjugate vector field is polynomial in this case, and the proof in \cites{Baldoma16,Baldoma16-2} follows roughly the same analysis as the proof of the classical parameterization method. A drawback of this result is that it produces invariant manifolds which are only continuous at the origin. Furthermore, it imposes that the linearization of the dynamical system at the fixed point is the identity on the center eigenspace. Our theorem on the other hand produces invariant manifold which are everywhere $C^n$ and does not impose that the linearization restricted to the center eigenspace is the identity.
 
\subsubsection*{Outline of the paper}

Our paper consists of three parts. We start by introducing notation in the following section and by giving a more precise formulation of the main theorem in \cref{MainTheoremSection}. As a corollary of the main theorem, we find an upper bound between the error of an almost solution $\mathcal{M}$ of \cref{FirstEquation} and the actual solution $\Lambda$ of \cref{FirstEquation} in terms of how well $\mathcal{M}$ solves \cref{FirstEquation}. Furthermore, we introduce a fixed point problem in \cref{C1Section}. This fixed point problem will produce the conjugacy and conjugate dynamics of \cref{MainTheorem}. Finally, we will prove that the fixed point problem defines a contraction in $C^1$. 

In the second part of our paper, which consists of \cref{C2ContinuitySection,CmContinuitySection,UniquenessSection}, we will obtain the smoothness results of the main theorem by means of two bootstrapping arguments. The first bootstrapping argument, in \cref{C2ContinuitySection}, shows that once we have a $C^1$ conjugacy and conjugate dynamics, they are in fact both $C^2$. The second argument, in \cref{CmContinuitySection}, shows that we can inductively increase the smoothness from $C^2$ to $C^n$. Finally, we prove in \cref{UniquenessSection} that our center manifold is unique, and we derive the estimate on $\Lambda - \mathcal{M}$.

\subsubsection*{Future work}

Our proposed method is defined for general discrete time dynamical systems. Although it is not proven in this paper, we are convinced that the method can be generalized to continuous time dynamical systems. Furthermore, our method can be used to parameterize different manifolds than only the center manifold, for instance the center-(un)stable or (un)stable manifold. In the latter case, our method should even work if the (un)stable manifold contains infinitely many resonances.

\subsection{Notation and conventions}\label{NotationSection}

We use the following notation and conventions in this paper.
\begin{itemize}
\item For a Banach space $X = \bigoplus_{i=1}^n X_i$, we assume that the norm on $X$ satisfies
\begin{align}
\norm{x} = \max_{1 \le i \le n} \left\{ \norm{x_i}_i \right\},  && \text{ for } x = \oplus_{i=1}^n x_i\label{DesiredPropertyNorm}
\end{align}
where $x_i \in X_i$ and $\norm{\cdot}_i$ the norm on $X_i$. If $X$ is equipped with another norm, we can always define an equivalent norm on $X$ which satisfies \cref{DesiredPropertyNorm} and leaves the norm unchanged on $X_i$.

\item For functions $f: X \to Y$ between Banach spaces, we denote with
\begin{align*}
\norm{f}_n \isdef \max_{0 \le m \le n} \sup_{x \in X} \norm{D^mf(x)}
\end{align*}
the $C^n$ norm of $f$ for $n \ge 0$.

For $X$ and $Y$ Banach spaces, we denote with 
\begin{align*}
C^n_b(X,Y) &\isdef \left\{ f  : X \to Y \mid f \text{ is } C^n \text{ and } \norm{f}_n < \infty \right\}.
\end{align*}
the Banach space of all $C^n$ bounded functions between $X$ and $Y$.

\item For a linear operator $A: X \to Y$ between Banach spaces, we denote with
\begin{align*}
\operatornorm{A} \isdef \sup_{\norm{x} =1} \norm{Ax}
\end{align*}
the operator norm of $A$.

For $X$ and $Y$ Banach spaces, we denote with 
\begin{align*}
\mathcal{L}(X,Y) \isdef \left\{ A: X \to Y \mid A \text{ is a linear operator and } \operatornorm{A} < \infty \right\}
\end{align*}
the Banach space of all bounded linear operators between $X$ and $Y$.

\item For a $k$-linear operator $A: X^k \to Y$ between Banach spaces, we denote with
\begin{align*}
\operatornorm{A} \isdef \sup_{\substack{ \norm{x_i} \le 1  \\ 1 \le i \le k}} \norm{A(x_1, \dots, x_k)}.
\end{align*}
the operator norm of $A$.

For $X$ and $Y$ Banach spaces, we denote with 
\begin{align*}
\mathcal{L}^k(X,Y) \isdef \left\{A : X^k \to Y \middle\vert \  A \text{ is a $k$-linear operator and $\operatornorm{A} < \infty$} \right\}
\end{align*}
the Banach space of all bounded $k$-linear operators between $X^k$ and $Y$.

\item For Banach spaces $X$, $Y$ and $Z$ and $f \in C^n_b(Y,Z)$ and $g \in C^n_b(X,Y)$, the $k^{\text{th}}$ derivative of $f \circ g$ is given by
\begin{align}
D^k[f \circ g](x) = D^kf(g(x)) ( \underbrace{Dg(x) , \dots, Dg(x)}_{k \text{ times } Dg(x)} ) + \text{lower derivatives of } f. \label{ChainRuleExplained}
\end{align}
We can identify the $k^\text{th}$ derivative of $f$ at $g(x)$ with a symmetric $k$-linear operator between $Y^k$ and $Z$. This motivates us to define the shorthand notation
\begin{align*}
A^{\otimes k} \isdef \underbrace{\left( A, \dots , A\right)}_{k \text{ times } A} \in \mathcal{L}(X^k,Y^k) && \text{ for } A \in \mathcal{L}(X,Y).
\end{align*}
Hence we may rewrite \cref{ChainRuleExplained} as
\begin{align*}
D^k[f \circ g](x) = D^kf(g(x)) \left( Dg(x) \right)^{\otimes k} + \text{lower derivatives of } f,
\end{align*}
where $D^kf(g(x)) \in \mathcal{L}^k(Y,Z)$.

\end{itemize}

\section{A quantitative formulation of the parameterized center manifold theorem}\label{MainTheoremSection}

We will now formulate the quantitative version of the parameterization of the center manifold theorem.

\begin{theorem}[Parameterization of the center manifold]\label{MainTheorem} Let $X$ be a Banach space and $F{}{}: X \to X $ a  $C^n$, $n \ge 2$, discrete dynamical system on $X$ such that $0$ is a fixed point of $F{}{}$. Denote $F{}{} = A + g$ with $A \isdef Df(0)$ and let $k_c : X_c \to X_c$ be chosen. Assume that
\begin{enumerate}
\item\label{MainTheoremSubpaces } There exist closed $A$-invariant subspaces $X_c$, $X_u$ and $X_s$ such that $X = X_c \oplus X_u \oplus X_s$. We write $A = \tmatrix{A_c & 0 & 0}{0 & A_u & 0}{0 & 0 & A_s}$ where we define $A_c \isdef A \big|_{X_c}$, and similarly define $A_u$ and $A_s$.
\item\label{MainTheoremOperators } The linear operators $A_c$ and $A_u$ are invertible.
\item\label{MainTheoremNorm } The norm on $X$ is such that
\begin{align*}
\operatornorm{A_c^{-1} }^{\tilde{n}} \operatornorm{ A_s } < 1 && \text{ and } && \operatornorm{A_u^{-1} } \operatornorm{A_c}^{\tilde{n}} < 1 && \text{ for all }  1 \le \tilde{n} \le n.
\end{align*}
\item\label{MainTheoremKnownBounds } The non-linearities $g$ and $k_c$ satisfy
\begin{align*}
g &\in \left\{ h \in C^{n}_b(X,X) \mid h(0)=0, \   Dh(0) = 0 \text{ and }\supnorm{Dh} \le L_g \right\}, \\ 
k_c &\in \left\{ h \in C^{n}_b(X_c,X_c) \mid h(0) = 0 , \ Dh(0) = 0 \text{ and } \supnorm{Dh} \le L_c \right\},
\end{align*}
for $L_g$ and $L_c$ small enough; see \cref{SmallnessConditions} for the explicit inequalities that $L_g$ and $L_c$ should satisfy.
\end{enumerate}
Then there exist a $C^{n}$ conjugacy $K{}{} :  X_c \to X$ and $C^{n}$ discrete dynamical system $R = A_c + r : X_c \to X_c$ such that
\begin{align}
(A + g) \circ K{}{} = K{}{} \circ (A_c +r) \label{ConjugacyEquation}.
\end{align}
Furthermore, $K{}{}$ and $R$ have the following properties:
\begin{enumerate}[label=\Alph*)]
\item\label{PropertiesR} The dynamical system $R=A_c + r$ is globally invertible, with its inverse given by $T = A_c^{-1} + t$, where
\begin{align*}
r &\in \left\{ h \in C^{n}_b(X_c,X_c) \mid h(0)=0, \   Dh(0) = 0 \text{ and }\supnorm{Dh} \le L_r \right\}, \\
t &\in \left\{ h \in C^{n}_b(X_c,X_c) \mid h(0)=0, \   Dh(0) = 0 \text{ and }\supnorm{Dh} \le L_t \right\}.
\end{align*}
The constants $L_r$ and $L_t$ depend on $L_g$ and $L_c$. See \cref{SmallnessConditions} for their definition.
\item\label{PropertiesK} The conjugacy $K{}{}$ is given by
\begin{align*}
K{}{} = \iota +  \begin{pmatrix}  k_c \\ k_u \\ k_s \end{pmatrix}
\end{align*}
with $\iota: X_c \to X$ the inclusion map and
\begin{align*}
k_u &\in \left\{ h \in C^{n}_b(X_c,X_u) \mid h(0)=0, \   Dh(0) = 0 \text{ and }\supnorm{Dh} \le L_u \right\}, \\ 
k_s &\in \left\{ h \in C^{n}_b(X_c,X_s) \mid h(0)=0, \   Dh(0) = 0 \text{ and }\supnorm{Dh} \le L_s \right\}. 
\end{align*}
The constants $L_u$ and $L_s$ depend on $L_g$ and $L_c$. See \cref{SmallnessConditions} for their definition.
\end{enumerate}
\end{theorem}

\begin{remark}
If $A$ has a spectral gap around the unit circle, conditions \cref{MainTheoremSubpaces ,MainTheoremOperators } hold when we take $X_c$ to be the center eigenspace, $X_u$ to be the unstable eigenspace and $X_s$ to be the stable eigenspace. Furthermore, in this case there exists a norm on $X$ such that condition \cref{MainTheoremNorm } holds. 
 For the construction of this norm, see Proposition A.1 of \cite{Cabre03}.
\end{remark}

In many practical situations, using for example numerics or Taylor approximations, one may be able to find $K{}{}_0$ and $R_0$ which almost solve \cref{ConjugacyEquation}, i.e. for which $F{}{} \circ K{}{}_0 - K{}{}_0 \circ R_0$ is small.  In other instances, the dynamical system $F{}{}: X \to X$ may be close to another dynamical system $F{}{}_0 :X \to X$, for which we are able to find a $K{}{}_0$ and $R_0$ satisfying $F{}{}_0 \circ K{}{}_0 - K{}{}_0 \circ R_0 = 0$ exactly. In both cases, we are interested in estimating the error between $K{}{}_0$ and the true conjugacy $K{}{}$, as well as the error between $R_0$ and the true dynamical system $R$. The following corollary of \cref{MainTheorem} gives bounds for these errors.

\begin{corollary}\label{AlmostSolution}
Let $n \ge 2$ and let $F{}{} :X \to X$ and $k_c : X_c \to X_c$ be $C^n$ functions which satisfy the conditions of \cref{MainTheorem}. Let $\varepsilon > 0$, $m<n$ and $M > 0$. Then there exists a constant $\mathcal{C}(M,m)$, which will be introduced in \cref{ProofofMainCorollary}, such that if
\begin{flalign*}
\text{1. } & k_0 \in \left\{ h \in C_b^{m+1}(X_c,X_u \oplus X_s ) \ \middle| \ h(0) = 0,\   Dh(0) = 0  \text{ and } \norm{h}_{m+1} \le M \right\}, &\\
\text{2. } & r_0 \in \left\{ h \in C_b^{m+1}(X_c,X_c ) \ \middle| \ h(0) = 0, \  Dh(0) = 0, \supnorm{Dh} \le L_r \text{ and } \norm{h}_{m+1} \le M \right\}, &\\
\text{3. } & \bignorm{F{}{} \circ \left( \iota +  \begin{pmatrix} k_c \\ k_0 \end{pmatrix} \right) - \left( \iota +  \begin{pmatrix} k_c \\ k_0\end{pmatrix} \right)  \circ \left( A_c + r_0 \right)}_m \le \varepsilon, &
\end{flalign*}
then we have
\begin{align*}
\norm{ k - k_0}_m \le C(M,m) \varepsilon &&\text{ and }&& \norm{ r - r_0}_m \le C(M,m) \varepsilon
\end{align*}
for $k = k_u \oplus k_s :X_c \to X_u \oplus X_s$ and $ r: X_c \to X_c$ from \cref{MainTheorem}.
\end{corollary}

\begin{remark}\label{SmallnessConditions} To state  what it means in \cref{MainTheorem} for $L_g$ and $L_c$ to be small enough, we first introduce explicit formulas for $L_r$, $L_t$, $L_u$ and $L_s$:
\begin{subequations}
\begin{align}
L_r &\isdef \frac{L_g + L_c \left( 2 \operatornorm{A_c} + L_g\right)}{1 - L_c}, \label{FirstDerivativeUnknownsR} \\
L_t &\isdef \frac{ \operatornorm{A_c^{-1}}^2 L_r}{1 - \operatornorm{A_c^{-1} } L_r }, \label{FirstDerivativeUnknownsT} \\
L_u &\isdef  \frac{\operatornorm{A_u^{-1} } \left( 1 + L_c \right) L_g}{1 - L_r \operatornorm{A_u^{-1} } - \operatornorm{A_c} \operatornorm{A_u^{-1}}}, \label{FirstDerivativeUnknownsU} \\
L_s &\isdef  \frac{\operatornorm{A_c^{-1}} \left( 1 + L_c \right)  L_g}{1 - L_r \operatornorm{A_c^{-1}} - \operatornorm{A_s }\operatornorm{ A_c^{-1} }} . \label{FirstDerivativeUnknownsS}
\end{align}
\end{subequations}
For $L_r$, $L_t$, $L_u$, and $L_s$ to be positive, their denominators have to be positive. So we will first of all require $L_g$ and $L_c$ to be so small that
\begin{align*}
\begin{aligned}
 L_c &< 1 , \\
 L_r \operatornorm{A_c^{-1}} &< 1 ,
\end{aligned} &&
\begin{aligned} 
 L_r \operatornorm{A_u^{-1}} + \operatornorm{A_c}\operatornorm{A_u^{-1}}  &< 1 ,\\ 
 L_r \operatornorm{A_c^{-1}} + \operatornorm{A_s} \operatornorm{A_c^{-1}} &< 1.
\end{aligned}
\end{align*}
Besides these conditions on $L_g$ and $L_c$, the proof of \cref{MainTheorem} contains multiple fixed points arguments, with contraction constants depending on $L_g$ and $L_c$. The corresponding contraction constants are, for $0 \le \tilde{n} \le n$,
\begin{subequations}
\begin{align}
\theta_{\tilde{n},1} &\isdef L_g + L_c, \label{ContractionConstants1} \\
\theta_{\tilde{n},2} &\isdef \operatornorm{A_u^{-1}} \left( \left( \operatornorm{A_c } + L_r \right)^{\tilde{n}} + L_g + L_u \right), \label{ContractionConstants2} \\
\theta_{\tilde{n},3} &\isdef L_{-1}^{\tilde{n}} \left( \operatornorm{A_s} \left( 1 + L_{-1}L_s \right) + L_g \left( 1 + L_{-1} \left( 1 + L_c \right) \right) \right). \label{ContractionConstants3} 
\end{align}
\end{subequations}
We note that $\theta_{\tilde{n},1}$ does not depend on $\tilde{n}$, but for consistency we choose to keep the index $\tilde{n}$. Furthermore, we define
\begin{align*}
L_{-1} \isdef  \operatornorm{A_c^{-1}} + L_t,
\end{align*}
which will act as a bound on $\supnorm{D(A_c + r)^{-1}}$. Now, if $L_g$ and $L_c$ are small, then we have that $\theta_{\tilde{n},i} < 1$. In particular, we require $L_g$ and $L_c$ to be so small that indeed
\begin{align*}
\theta_{\tilde{n},1} &<1, &
\theta_{\tilde{n},2} &<1, &
\theta_{\tilde{n},3} &<1  &\text{ for all } \tilde{n} \le n.
\end{align*}
We will often refer to this remark by writing ``(\dots) small in the sense of \cref{SmallnessConditions} for $n=m$.''. By this, we mean  that $\theta_{\tilde{n},i} < 1$ holds for all $\tilde{n} \le m$ and $i=1,2,3$.
\end{remark}

\subsection{Proof scheme of \texorpdfstring{\cref{MainTheorem}}{the main theorem}}\label{ProofScheme}

To prove \cref{MainTheorem} we use four steps:
\begin{enumerate}
\item We define a fixed point operator and show that its fixed points are solutions to the conjugacy equation \cref{ConjugacyEquation}.
\item We show that the fixed point operator is a contraction with respect to the $C^1$ norm on an appropriate set of $C^2$ functions. Therefore, we find a pair of $C^1$ functions $(K{}{},r)$ such that equation \cref{ConjugacyEquation} holds.
\item We show that the $C^1$ solution $(K{}{},r)$ of equation \cref{ConjugacyEquation} is in fact $C^2$.
\item Finally, we use induction to prove that  $K{}{}$ and $r$ are $C^n$.
\end{enumerate}

We find it easiest to split the proof in the four steps outlined above. We note that in step 2 we find a contraction with respect to the $C^1$ norm in a space of $C^2$ functions, which is the reason that we assume $n \ge 2$ in \cref{MainTheorem}. Furthermore, going from a $C^1$ solution to a $C^2$ solution needs different estimates than when going from a $C^m$ solution to a $C^{m+1}$ solution for $m \ge 2$, see for instance \cref{DifferenceC2Norm} and \cref{DifferenceCmNorm}.

\section{A \texorpdfstring{$C^1$}{C1} center manifold}\label{C1Section}
\subsection{A fixed point operator}\label{FixedPointSection}

We start by setting up a fixed point operator. For this, we consider the conjugacy equation
\begin{align*}
(A +g ) \circ K{}{} - K{}{} \circ (A_c + r) = 0.
\end{align*}
When we write $g = \tmatrix{g_c}{g_u}{g_s}$ and $K{}{} = \iota + \tmatrix{ k_c}{k_u}{k_s}$, then we have component-wise
\begin{align*}
\begin{pmatrix} A_c & 0 & 0 \\ 0 & A_u & 0 \\ 0 & 0 & A_s \end{pmatrix}  \begin{pmatrix} \operatorname{Id} + k_c \\ k_u \\ k_s \end{pmatrix} + \begin{pmatrix} g_c \circ K{}{} \\ g_u \circ K{}{} \\ g_s \circ K{}{} \end{pmatrix} - \begin{pmatrix} A_c + r + k_c \circ (A_c + r) \\ k_u \circ (A_c + r) \\ k_s \circ (A_c +r) \end{pmatrix} &= 0 .
\end{align*}
In other words, we obtain the three equations
\begin{align}
\left\{  \begin{matrix*}[r]  A_c k_c +g_c \circ K{}{} - r - k_c \circ ( A_c +r)  = 0,			 \\
			A_u k_u + g_u \circ K{}{} - k_u \circ (A_c + r) = 0, \\
			A_s k_s + g_s \circ K{}{} - k_s \circ (A_c + r) = 0.\end{matrix*}	  \right. \label{FixedPointOperatorEq1}
\end{align}
The first equation is equivalent to
\begin{align*}
r = A_c k_c + g_c \circ K{}{} - k_c \circ (A_c + r).
\end{align*}
Secondly, as $A_u$ is invertible, the second equation is equivalent to
\begin{align*}
k_u = A_u^{-1} \left( A_u k_u \right)  = A_u^{-1} \left( k_u \circ (A_c + r) - g_u \circ K{}{} \right).
\end{align*}
Finally, if we assume that $A_c + r$ is invertible, the last equation is equivalent to
\begin{align*}
k_s = A_s k_s \circ (A_c +r)^{-1} + g_s \circ K{}{} \circ (A_c + r)^{-1}.
\end{align*}
We see that the system \cref{FixedPointOperatorEq1} is equivalent with
\begin{align}\left\{
\begin{aligned}	 r &= A_c k_c +g_c \circ K{}{} -k_c \circ (A_c +r),  \\
			k_u &=A_u^{-1}k_u \circ (A_c + r)  - A_u^{-1} g_u \circ K{}{}, \\
			k_s &= A_s k_s\circ (A_c + r)^{-1} + g_s \circ K{}{} \circ (A_c +r)^{-1}.  \end{aligned} \right.
			\label{FixedPointOperatorEquation}
\end{align}
Thus $\tmatrix{r}{k_u}{k_s}$ is a fixed point of the operator $\Theta$, defined by
\begin{align}
\Theta^{} : \begin{pmatrix} r \\ k_u \\ k_s \end{pmatrix} \mapsto \begin{pmatrix*}[l]
	A_c k_c +g_c \circ K{}{} -k_c \circ (A_c +r)  \\
	A_u^{-1}k_u \circ (A_c + r)  - A_u^{-1} g_u \circ K{}{} \\
	 A_s k_s\circ (A_c + r)^{-1} + g_s \circ K{}{} \circ (A_c +r)^{-1}
\end{pmatrix*}.
\label{FixedPointOperator}
\end{align}
The following proposition summarizes this derivation.

\begin{proposition}\label{FixedPointOperatorProposition}
Let $K{}{} = \iota + \tmatrix{ k_c}{k_u}{k_s}: X_c \to X$ and $A_c + r : X_c \to X_c$. If $A_c + r$ is globally invertible, then the following are equivalent:
\begin{enumerate}[label = \roman*)]
\item The dynamical system $F{}{}$ is conjugate to $A_c + r$ by $K{}{}$.
\item The function $\tmatrix{r}{k_u}{k_s}: X_c \to X$ satisfies the system \cref{FixedPointOperatorEquation}.
\item The function $\tmatrix{r}{k_u}{k_s} :X_c \to X$ is a fixed point of $\Theta$ defined in \cref{FixedPointOperator}.
\end{enumerate}
\end{proposition}
\begin{proof} This follows from the derivation above.
\end{proof}

\subsection{\texorpdfstring{$\Theta$}{Theta} is well-defined}\label{NormInverse}

One of the conditions of \cref{FixedPointOperatorProposition} is that $A_c + r$ is globally invertible. Under a mild condition on the derivative of $r$, we show that $A_c + r$ is a global diffeomorphism. We will state a more general result, as we need a more general statement in \cref{LuckManifoldIndifferent}. Furthermore, we note that the result also follows from theorem 2.1 of \cite{Plastock74}, which states that under some growth condition on $A_c + r$ and the inverse of its derivative, $A_c +r $ is a global diffeomorphism. However, we will provide an alternative proof, as this proof is exemplary for the structure of the proof of \cref{MainTheorem}.

\begin{lemma}\label{GlobalDiffeomorphism} Let $B \in \mathcal{L}(X_c,X_c)$ be invertible. 
\begin{enumerate}[label = {\roman*)}, ref={\cref{GlobalDiffeomorphism}\roman*)}]
\item\label{GlobalDiffeomorphism1} We have that $B + \psi :X_c \to X_c$ is a global diffeomorphism for all $\psi \in C_b^1(X_c,X_c)$ with  $\supnorm{D\psi} < \operatornorm{B^{-1}}^{-1}$.
\item\label{GlobalDiffeomorphism2}  If $B+ \psi :X_c \to X_c$ is a homeomorphism and $\psi \in C_b^0(X_c, X_c)$, then $(B+ \psi)^{-1} = B^{-1} + \varphi$ with $\varphi \in C_b^0(X_c, X_c)$.
\end{enumerate}
\end{lemma}

\begin{proof}\textit{i)} We will first show that $B + \psi$ is globally invertible, and use its inverse to prove that it is in fact a diffeomorphism, which is similar to how we show the smoothness result of \cref{MainTheorem}.

For $\varphi \in C_b^0(X_c,X_c)$, we have $(B + \psi) \circ  (B^{-1} + \varphi) = \operatorname{Id}$ if and only if 
\begin{align}
\varphi &= - B^{-1} \left(  \psi \circ \left( B^{-1} + \varphi \right) \right)  \isfed \Psi(\varphi). \label{GlobalDiffeoEq1} 
\end{align}
Our first goal is to show that $\Psi$ is a contraction on $C^0_b(X_c, X_c)$.

As $\psi$ is $C^1$ with bounded derivative, it is Lipschitz with Lipschitz constant equal to the norm of its derivative. Thus for $\varphi,\tilde{\varphi}\in C^0_b(X_c,X_c)$ we have 
\begin{align*}
\supnorm{\Psi(\varphi) - \Psi(\tilde{\varphi})} \le \operatornorm{B^{-1}} \supnorm{D\psi} \supnorm{\varphi - \tilde{\varphi}}.
\end{align*}
Thus with $\supnorm{D\psi} < \operatornorm{B^{-1}}^{-1}$ we have that $\Psi$ is contraction operator. Let $\varphi$ denote the fixed point of $\Psi$.

If $\varphi$ is $C^1$, we can differentiate \cref{GlobalDiffeoEq1} to $x$, and the derivative of $\varphi$ satisfies
\begin{align}
D\varphi = -B^{-1} D\psi(B^{-1} + \varphi)  (B^{-1} + D\varphi). \label{GlobalDiffeoEq3}
\end{align}
In other words, if $\varphi$ is $C^1$, then $D\varphi$ is a fixed point of the operator
\begin{align*}
\Phi:&\  C^0_b(X_c, \mathcal{L}(X_c,X_c)) \to C^0_b(X_c,\mathcal{L}(X_c,X_c)),  \\
	&\ A \mapsto - B^{-1} D\psi(B^{-1} + \varphi)  B^{-1} - B^{-1}D\psi(B^{-1} + \varphi)  A. 
\end{align*}
We will show that $\Phi$ is a contraction, and that its fixed point is indeed the derivative of $\varphi$.

As $\supnorm{B^{-1}D\psi(B^{-1} + \varphi)} \le \operatornorm{B} \supnorm{D\psi} < 1$, we see that $\Phi$ is contraction, and we denote the fixed point of $\Phi$ by $\mathcal{A}$. All that is left to show is that $\mathcal{A}$ is the derivative of $\varphi$. We have
\begin{align*}
\psi (B^{-1}(x+y) + \varphi(x+y)) &-\psi(B^{-1}x-\varphi(x)) \nonumber \\
				&= \int_0^1 D\psi(z(s,x,y)) \text{d}s (B^{-1}y + \varphi(x+y) - \varphi(y)),
\end{align*}
where we define $z(s,x,y)  \isdef  B^{-1} x + sB^{-1}y + s \varphi(x+y) + (1-s)\varphi(x)$. We find, as $\phi$ is the fixed point of $\Psi$,
\begin{align}
\norm{\varphi(x+y) &- \varphi(y) - \mathcal{A}(x)y}  \nonumber \\
			&\le\operatornorm{B^{-1}} \int_0^1 \operatornorm{ D\psi(z(s,x,y)) }\text{d}s \norm{(\varphi(x+y) - \varphi(x) -\mathcal{A}(x)y)  }  \nonumber \\
			&\quad + \operatornorm{B^{-1}} \int_0^1 \operatornorm{D\psi(z(s,x,y)) - D\psi(z(s,x,0))} \text{d}s  \norm{B^{-1}y}  \nonumber \\
			&\quad + \operatornorm{B^{-1}} \int_0^1 \operatornorm{D\psi(z(s,x,y)) - D\psi(z(s,x,0))} \text{d}s \norm{ \mathcal{A}(x)y }.  \label{GlobalDiffeoEq2}
\end{align}
Define $\theta = \operatornorm{B^{-1}} \int_0^1 \norm{ D\psi(z(s,x,y)) }\text{d}s$, then $\theta \le \operatornorm{B^{-1}} \supnorm{D\psi} < 1$. Hence \cref{GlobalDiffeoEq2} implies
\begin{align*}
&\frac{\norm{\varphi(x+y) - \varphi(y) - \mathcal{A}(x)y}}{\norm{y}} \\
&\quad \quad \le \frac{\operatornorm{B^{-1}}}{1-\theta} \int_0^1 \operatornorm{D\psi(z(s,x,y)) - D\psi(z(s,x,0))} \text{d}s  ( \operatornorm{B^{-1}} + \operatornorm{\mathcal{A}(x)}).
\end{align*}
Furthermore, $\supnorm{D\psi}$ is finite, hence by using the Lebesgue Dominated Convergence Theorem, we can interchange limits and integrals in
\begin{align*}
\lim_{\norm{y} \to 0} \int_0^1 \operatornorm{D\psi(z(s,x,y)) - D\psi(z(s,x,0))} \text{d}s.
\end{align*}
Using continuity of $\operatornorm{D\psi(z(s,x,y))}$ in $y$, the limit is $0$. Therefore, 
\begin{align*}
\lim_{\norm{y} \to 0} \frac{\norm{\varphi(x+y) - \varphi(y) - \mathcal{A}(x)y}}{\norm{y}} = 0
\end{align*}
and we conclude that the right inverse $B^{-1} + \varphi$ of $B+\psi$ is $C^1$.

We still need to show that $B^{-1} + \varphi$ is the inverse of $B + \psi$, or equivalently show that $B+ \psi$ is injective. Let $x,y \in X_c$ be such that $Bx + \psi(x) = By + \psi(y)$, then
\begin{align*}
0 = \norm{Bx + \psi(x) - By - \psi(y)} \ge \operatornorm{B^{-1}}^{-1} \norm{x-y} - \supnorm{D\psi} \norm{x-y}.
\end{align*}
As $\operatornorm{B^{-1}}^{-1} - \supnorm{D\psi} > 0$, we must have $\norm{x-y} = 0$, and thus $x=y$. With this, we have shown that $B + \psi$ is injective, and thus invertible.

\textit{ii)} We want to show that for all bounded continuous $\psi : X_c \to X_c$ such that $B+ \psi$ is a homeomorphism, $\varphi \isdef B^{-1} - (B+ \psi)^{-1}$ is uniformly bounded. We have
\begin{align*}
\sup_{x \in X_c} \norm{ \varphi(x)} &= \sup_{x \in X_c} \norm{ \left( B^{-1} - (B+ \psi)^{-1}\right)(x)} \\
						&= \sup_{x \in X_c} \norm{ \left(B^{-1} -(B+\psi)^{-1} \right) \left( \left( B+ \psi \right)(x) \right)} \\
						&= \sup_{x \in X_c} \norm{ x + B^{-1} \psi(x) - x} \\
						&\le \sup_{x \in X_c} \operatornorm{B^{-1}} \norm{\psi(x)} \\
						&=\operatornorm{B^{-1}} \supnorm{\psi}. \qedhere
\end{align*}
\end{proof}
From the above lemma it follows that if $r : X_c \to X_c$ is $C^1$ and its derivative $Dr$ is uniformly bounded by $\operatornorm{A_c^{-1}}^{-1}$, then there exists a $C^1$  bounded function $t :X_c \to X_c$ such that $(A_c + r)^{-1} = A_c^{-1} + t$. Furthermore, from equation \cref{GlobalDiffeoEq3} it follows that $\supnorm{Dt}$ satisfies
\begin{align*}
\supnorm{Dt} \le \supnorm{ A_c^{-1}} \supnorm{Dr} \supnorm{A_c^{-1}} + \supnorm{ A_c^{-1}} \supnorm{Dr} \supnorm{Dt}.
\end{align*}
In particular, we claim in \cref{MainTheorem} that the dynamical system $A_c + r$ will be invertible with inverse $A_c^{-1} + t$. Furthermore, we claim that we have the bounds $\supnorm{Dr} \le L_r$ and $\supnorm{Dt} \le L_t$, with $L_r$ and $L_t$ defined in \cref{FirstDerivativeUnknownsR} and \cref{FirstDerivativeUnknownsT} respectively. Since $L_r \le \operatornorm{A_c^{-1}}^{-1}$, it follows from the lemma that if $\supnorm{Dr} \le L_r$, we have that $A_c + r$ is invertible, and from the definition of $L_t$ that $\supnorm{Dt} \le L_t$. In particular, $\Theta$ is well-defined if we assume $\supnorm{Dr} \le L_r$ and the desired properties of $R = A_c + r$ follow from the above discussion.

\subsection{A first invariant set for \texorpdfstring{$\Theta$}{theta}}

We want to find an invariant set for $\Theta^{}$ by putting additional bounds on $\supnorm{Dr}$, $\supnorm{Dk_u}$ and $\supnorm{Dk_s}$. To this end, we define 
\begin{align*}
\Gamma_0 \isdef \left\{ \Lambda =  \begin{pmatrix} r \\ k_u \\ k_s \end{pmatrix} \in C^1_b(X_c, X)  \middle\vert \  \begin{matrix*}[l] 
\Lambda(0) = 0, \ D\Lambda(0) = 0 \vspace{3 pt} \\
\qquad \text{ and }\left\{ \begin{matrix*}[l]
\ \supnorm{Dr} \le L_r  \\
\supnorm{Dk_u} \le  L_u \\
\supnorm{Dk_s} \le  L_s   \end{matrix*}\right. 
\end{matrix*}\right\}.
\end{align*}

\begin{theorem}\label{LemmaBanachSpace}
Assume that $L_g$ and $L_c$ are small in the sense of \cref{SmallnessConditions} for $n=2$. Then $\Theta^{}$ is well-defined on $\Gamma_0$. Furthermore, $\Gamma_0$ is non-empty, invariant under $\Theta^{}$, and closed.
\end{theorem}
\begin{proof}
We first note that by \cref{SmallnessConditions} it holds that $0 \in \Gamma_0$ as $L_r$, $L_u$ and $L_s$ are all positive. Furthermore, from \cref{SmallnessConditions} it follows that $L_r < \operatornorm{A_c^{-1}}^{-1}$. By \cref{NormInverse}, we thus have that $A_c +r$ is a global diffeomorphism for $\tmatrix{r}{k_u}{k_s} \in \Gamma_0$, hence $\Theta^{}$ is well defined on $\Gamma_0$. Finally, it follows directly from  the definition of $\Gamma_0$ that $\Gamma_0$ is closed. All that remains to prove is that $\Gamma_0$ is invariant under $\Theta^{}$.

Let $\Lambda = \tmatrix{r}{k_u}{k_s} \in \Gamma_0$. Then we must prove that $\Theta^{}(\Lambda) = \tmatrix{\Theta^{}_1(\Lambda)}{\Theta^{}_2(\Lambda)}{\Theta^{}_3(\Lambda)} \in \Gamma_0$. We will prove the resulting conditions on $\Theta^{}(\Lambda)$ component-wise. We start by showing $\Theta^{}(\Lambda)(0)=0$. We have
\begin{align*}
\Theta^{}_1 (\Lambda)(0) &=	A_c k_c(0) +g_c(K{}{}(0)) -k_c((A_c +r)(0)).
\end{align*}
With $\Lambda \in \Gamma_0$, we have $r(0) = 0$, thus $(A_c + r)(0) = 0$. So we get
\begin{align*}
\Theta^{}_1 (\Lambda)(0) &=	A_c k_c(0) +g_c(K{}{}(0)) -k_c(0).
\end{align*}
By assumption \cref{MainTheoremKnownBounds } of \cref{MainTheorem} we have $k_c(0) = 0$. Together with $k_u(0) = k_s(0) =0$ we find $K{}{}(0) = \iota(0) +  \tmatrix{ k_c}{k_u}{k_s}(0) = 0$. Therefore
\begin{align*}
\Theta^{}_1 (\Lambda)(0) &=g_c(0).
\end{align*}
Finally, again by assumption \cref{MainTheoremKnownBounds } of \cref{MainTheorem}, we have $g_c(0) = 0$ and we conclude that
\begin{align*}
\Theta^{}_1(\Lambda)(0) = 0.
\end{align*}
 Similarly, $\Theta^{}_2(\Lambda)(0) = 0$ and $\Theta^{}_3(\Lambda)(0) = 0$. For the latter equality, we note that $A_c + r$ is a diffeomorphism and $(A_c + r)(0) = 0$, thus $(A_c + r)^{-1}(0) = 0$. So we find that $\Theta^{}(\Lambda)(0) = 0$.
 
 Likewise, we show that $D[\Theta^{}(\Lambda)](0) = 0$ component-wise.  For example, we have
\begin{align*}
D[\Theta^{}_1(\Lambda)](0) &= A_c D k_c(0) + Dg_c(K{}{}(0))DK{}{}(0) -Dk_c((A_c + r)(0))D(A_c + r)(0) \\
				&=A_c Dk_c(0) + Dg_c(0)DK{}{}(0) - Dk_c(0)(A_c + Dr)(0).
\end{align*}
By assumption \cref{MainTheoremKnownBounds } of \cref{MainTheorem} we have $Dg_c(0) = 0$ and $Dk_c(0) = 0$. So we find $D[\Theta^{}_1(\Lambda)](0)  = 0$. Similarly, we find $D[\Theta^{}_2(\Lambda)](0)= 0$ and $D[\Theta^{}_3(\Lambda)](0) = 0$. So we see that $D[\Theta^{}(\Lambda)](0) = 0$. 

For the bounds on the derivatives of the components of $\Theta^{}(\Lambda)$, we use the shorthand notation $R = A_c + r$. For the first component we have
\begin{align}
\supnorm{D[ \Theta^{}_1(\Lambda)]} &\le \supnorm{D[A_c k_c]} + \supnorm{D[g_c \circ K{}{}]} + \supnorm{D[k_c \circ R]}. \label{ FPOInvarianceEq7}
\end{align}
We have by assumption \cref{MainTheoremKnownBounds } of \cref{MainTheorem}
\begin{align}
\supnorm{D[A_c k_c]} &\le \operatornorm{A_c} \supnorm{Dk_c} \le \operatornorm{A_c} L_c. \label{ FPOInvarianceEq8}
\end{align}
Furthermore, we may estimate
\begin{align*}
\supnorm{DK{}{}} = \bignorm{ \tmatrix{ D \left( \operatorname{Id} + k_c \right)}{Dk_u}{Dk_s}}_0 \le \max \left\{ 1 + \supnorm{Dk_c} , \supnorm{Dk_u} , \supnorm{Dk_s} \right\}.
\end{align*}
Now note that $1 + \supnorm{Dk_c} \le 1 + L_c$, $\supnorm{Dk_u} \le L_u$ and $\supnorm{Dk_s} \le L_s$. From the definition of $L_u$ in \cref{FirstDerivativeUnknownsU}, we see that $L_u \le 1 + L_c$ if 
\begin{align*}
\operatornorm{A_u^{-1}}L_g \le 1 - L_r \operatornorm{A_u^{-1}} - \operatornorm{A_c} \operatornorm{A_u^{-1}}.
\end{align*}
The latter inequality holds since $L_u \ge 0$ and it is assumed that $\theta_{1,2} <1$.
This holds as $0 \le L_u$ and $\theta_{1,2} < 1$. Likewise, we can bound $L_s \le 1 + L_c$ as $\theta_{1,3} < 1$. Thus we find
\begin{align}
\supnorm{DK{}{}} \le 1 + L_c. \label{EstimateDk}
\end{align}
With this estimate, we get
\begin{align}
\supnorm{D[g_c \circ K{}{}]}  &\le  \supnorm{Dg_c}\supnorm{DK{}{}} \le L_g (1 + L_c). \label{ FPOInvarianceEq9}
\end{align}
Finally, we have
\begin{align}
\supnorm{D[k_c \circ R]} &\le \supnorm{Dk_c} \supnorm{DR} \le L_c \left( \operatornorm{A_c} + L_r \right). \label{ FPOInvarianceEq10}
\end{align}
By combining \cref{ FPOInvarianceEq8,, FPOInvarianceEq9,, FPOInvarianceEq10} we obtain from \cref{ FPOInvarianceEq7} that
\begin{align*}
\supnorm{D[ \Theta^{}_1(\Lambda)]} \le \left( 1 - L_c \right) L_r + L_c L_r = L_r.
\end{align*}
The inequalities $\supnorm{D[\Theta^{}_2(\Lambda)]} \le L_u$ and $\supnorm{D[\Theta^{}_3(\Lambda)]} \le L_s$ follow from similar calculations.

Thus we see that $\Theta^{}$ leaves $\Gamma_0$ invariant.
\end{proof}

\subsection{A second invariant set for \texorpdfstring{$\Theta$}{theta}}

We would now like to prove that $\Theta^{} : \Gamma_0 \to \Gamma_0$ is a contraction with respect to the $C^1$ norm, i.e. we must show that $\onenorm{\Theta^{}(\Lambda) - \Theta^{}(\tilde{\Lambda})} \le c \onenorm{\Lambda - \tilde{\Lambda}}$ for all $\Lambda, \tilde{\Lambda} \in \Gamma_0$ and some $c<1$. We will need to restrict to a subset of $\Gamma_0$ to obtain such a bound. To motivate our choice of the subset, let us look more carefully at the second component of $\Theta^{}$ and focus on the bound on its derivative. For instance, one obtains an estimate of the form
\begin{align*}
\supnorm{D\Theta^{}_2(\Lambda) - D\Theta^{}_2(\tilde{\Lambda})} &\le  \operatornorm{A_u^{-1}} \supnorm{D[k_u \circ (A_c +r )] - D[\tilde{k}_u \circ (A_c + \tilde{r}) ]} \\
			&\quad + \operatornorm{A_u^{-1}} \supnorm{D[g_u \circ K{}{}] - D[g_u \circ \tilde{K{}{}}]}.
\end{align*}
If we now estimate the second factor of the first term, we obtain
\begin{align*}
\supnorm{D[k_u \circ (A_c +r )] - D[\tilde{k}_u \circ (A_c + \tilde{r}) ]} &\le \supnorm{D[k_u \circ (A_c + r)] - D[k_u \circ (A_c + \tilde{r})]} \\
					&\quad + \supnorm{D[k_u \circ (A_c + \tilde{r})] - D[\tilde{k}_u \circ (A_c + \tilde{r})]}.
\end{align*}
We can bound the first term in the right hand side by
\begin{align*}
\supnorm{D[k_u \circ (A_c +r )] - D[k_u \circ (A_c + \tilde{r}) ]} &\le \supnorm{Dk_u(A_c + r) - Dk_u  (A_c + \tilde{r})} \operatornorm{A_c} \\
					&\quad + \supnorm{Dk_u(A_c + r)Dr - Dk_u  (A_c + \tilde{r})D\tilde{r}}.
\end{align*}
But, to estimate $\supnorm{Dk_u(A_c + r) - Dk_u  (A_c + \tilde{r})}$ in terms of $\onenorm{r - \tilde{r}}$ we would need a uniform bound on the Lipschitz constant of $Dk_u$ for $\tmatrix{r}{k_u}{k_s} \in \Gamma_0$.

Therefore, we will restrict $\Theta^{}$ to a subset of $\Gamma_0$ consisting of once differentiable functions with uniform bound on the Lipschitz constant of the derivative. In fact, we will restrict $\Theta^{}$ to a subset of twice differentiable functions with uniform bound on the supremum norm of the second derivative. We make this choice rather than working with $C^1$ function with a Lipschitz bound on the derivative, because we  later want to show that $r$, $k_u$ and $k_s$ are $C^n$. We thus  define for $\delta >0$ the space of $C^2$ functions
\begin{align}
\Gamma_1( \delta) \isdef \Gamma_0 \cap \left\{ \begin{pmatrix} r \\ k_u \\ k_s \end{pmatrix} \in C^{2}_b(X_c, X)   \middle\vert \  \begin{matrix*}[l] 
\supnorm{ D^2r } \le \delta \\
\supnorm{ D^2k_u } \le \delta \\ 
\supnorm{ D^2k_s } \le \delta \end{matrix*}  \right\}. \label{EquationUniformSecondDerivativeBound}
\end{align}
In \cref{MainTheorem}, we assume that $k_c$ and $g$ have bounded second derivative. If we take $\supnorm{D^2k_c} \le \varepsilon$ and $\supnorm{D^2 g} \le \varepsilon$ for some positive $\varepsilon$, then we will be able to construct an explicit $\delta(\varepsilon)$ such that $\Gamma_1(\delta(\varepsilon))$ is invariant under $\Theta^{}$.

\begin{proposition}\label{UniformSecondDerivativeBound}
Let $\varepsilon > 0$ and assume that $\supnorm{D^2 g },\supnorm{D^2 k_c } \le \varepsilon$. Furthermore, assume that $L_g$ and $L_c$ are small in the sense of \cref{SmallnessConditions} for $n=2$. Then there exists a $\delta(\varepsilon) > 0$, which we explicitly define in \cref{DeltaDefinition}, such that $\Gamma_1(\delta(\varepsilon))$ is invariant under $\Theta^{}$. Furthermore, $\delta(\varepsilon) \downarrow 0$ as $\varepsilon \downarrow 0$.
\end{proposition}
\begin{proof}
Let $\varepsilon > 0$. We want to find $\delta(\varepsilon) > 0$ such that $\Gamma_1(\delta(\varepsilon))$ is invariant under $\Theta^{}$. Let $\Lambda = \tmatrix{r}{k_u}{k_s} \in \Gamma_1(\delta)$, and introduce the shorthand notation $R= (A_c + r)$ and $T= (A_c + r)^{-1}$. We have to show that $\supnorm{D^2[\Theta^{}_i(\Lambda)]} \le \delta$ for $i=1,2,3$.

We estimate the first component by
\begin{align}
\supnorm{D^2 \Theta^{}_1( \Lambda ) } &\le \supnorm{ D^2 [A_c k_c]} + \supnorm{D^2[g_c \circ K{}{}]} + \supnorm{D^2[k_c \circ R]}. \label{ BanachSpaceC2Eq1}
\end{align}
We estimate the terms separately. First,
\begin{align}
\supnorm{ D^2 [A_c k_c]} &= \supnorm{ A_c D^2k_c} \le \operatornorm{A_c} \supnorm{D^2k_c} \le \operatornorm{A_c} \varepsilon. \label{ BanachSpaceC2Eq2}
\end{align}
Second, as $\supnorm{D^2K{}{}} \le \max\{ \supnorm{D^2k_c} , \supnorm{D^2k_u}, \supnorm{D^2k_s} \} \le \max\{ \varepsilon , \delta \} \le \varepsilon + \delta$, where we choose the rough estimate of $\varepsilon + \delta$ in order to write $\delta(\varepsilon)$ explicitly in \cref{DeltaDefinition}, we find
\begin{align}
\supnorm{D^2[g_c \circ K{}{}]} &\le \supnorm{D^2g_c(K{}{}) ( DK{}{} , DK{}{})} + \supnorm{Dg_c(K{}{}) D^2K{}{}} \nonumber \\
					&\le \supnorm{D^2g_c}\supnorm{DK{}{}}^2 + \supnorm{Dg_c} \supnorm{D^2K{}{}} \nonumber \\
					&\le (1 + L_c)^2 \varepsilon + L_g \varepsilon + L_g \delta. \label{ BanachSpaceC2Eq3}
\end{align}
Finally, we have
\begin{align}
\supnorm{D^2[k_c \circ R]} &\le \supnorm{D^2k_c(R) (DR , DR)} + \supnorm{Dk_c(R) D^2 r} \nonumber \\
					&\le (\operatornorm{A_c} + L_r)^2 \varepsilon + L_c \delta. \label{ BanachSpaceC2Eq4}
\end{align}
So inequality \cref{ BanachSpaceC2Eq1} together with estimates \cref{ BanachSpaceC2Eq2,, BanachSpaceC2Eq3,, BanachSpaceC2Eq4} gives
\begin{align}
\supnorm{D^2 \Theta^{}_1( \Lambda ) }	&\le \left( L_g + L_c \right) \delta + \left( \operatornorm{A_c} + \left( 1 + L_c \right)^2 + L_g  +  ( \operatornorm{A_c} + L_r)^2  \right) \varepsilon \nonumber \\
							&= \theta_{2,1} \delta + C_1(\varepsilon), \label{ rComponentBanachSpaceC2}
\end{align}
where $C_1(\varepsilon) \isdef  \left(\operatornorm{A_c} + \left( 1 + L_c \right)^2 + L_g  + ( \operatornorm{A_c} + L_r)^2  \right)\varepsilon$. 

Likewise, we estimate the second component by
\begin{align}
\supnorm{D^2[ \Theta^{}_2( \Lambda ) ]} &\le \operatornorm{ A_u^{-1}} ( \supnorm{D^2[g_u \circ K{}{}]} + \supnorm{D^2[k_u \circ R]} ). \label{ BanachSpaceC2Eq5}
\end{align}
As before, we estimate the terms separately to find
\begin{align*}
\supnorm{D^2[g_u \circ K{}{}]} &\le \supnorm{D^2g_u} \supnorm{DK{}{}}^2 + \supnorm{Dg_u} \supnorm{D^2 K{}{}}  \\
						&\le (1 + L_c)^2 \varepsilon + L_g \varepsilon + L_g \delta,  \\ 
\supnorm{D^2[k_u \circ R]} &\le \supnorm{D^2 k_u} \supnorm{R}^2 + \supnorm{Dk_u} \supnorm{D^2r}  \\
							&\le ( \operatornorm{A_c} + L_r )^2 \delta  + L_u \delta. 
\end{align*}
Hence, inequality \cref{ BanachSpaceC2Eq5} becomes
\begin{align}
\supnorm{D^2 [\Theta^{}_2( \Lambda )] } &\le  \operatornorm{A_u^{-1}} ( (\operatornorm{A_c} + L_r)^2 + L_u + L_g ) \delta \nonumber \\
							&\quad + \operatornorm{A_u^{-1}}( (1+ L_c)^2  + L_g ) \varepsilon  \nonumber \\
							&= \theta_{2,2} \delta + C_2(\varepsilon),  \label{ uComponentBanachSpaceC2}
\end{align}
where $C_2(\varepsilon) \isdef \operatornorm{A_u^{-1}}( (1+ L_c)^2 + L_g) \varepsilon$. 

Finally, we estimate the third component by
\begin{align}
\supnorm{D^2 [\Theta^{}_3( \Lambda ) ]} &\le  \operatornorm{A_s} \supnorm{D^2[k_s \circ T]} + \supnorm{D^2[g_s \circ K{}{} \circ T]}. \label{ BanachSpaceC2Eq8}
\end{align}
Before we estimate the different parts of the right hand side of \cref{ BanachSpaceC2Eq8}, we estimate $\supnorm{D^2 T}$ by applying the chain rule twice to the right hand side of $0 = D^2 \operatorname{Id} = D^2[R \circ T]$.
\begin{align*}
0 = D^2 [R \circ T] = D^2 R(T) (DT, DT) + DR(T) D^2T.
\end{align*}
We let $DR(T)^{-1} = DT$ act on the left to obtain the upper bound
\begin{align}
\supnorm{D^2 T} = \supnorm{ - DT D^2R(T) (DT,DT)} \le L_{-1}^3 \delta. \label{EstimateDT}
\end{align}
We now find the following estimates for the terms in \cref{ BanachSpaceC2Eq8}:
\begin{align*}
\supnorm{D^2[k_s \circ T]} &\le \supnorm{D^2 k_s}\supnorm{DT}^2 + \supnorm{Dk_s}\supnorm{D^2T} \\
				 &\le L_{-1}^2 \delta + L_s L_{-1}^3\delta, \\
\supnorm{D^2[g_s \circ K{}{} \circ T]} &\le \supnorm{D^2 g_s} (\supnorm{DK{}{}} \supnorm{DT})^2 + \supnorm{Dg_s} \supnorm{D^2K{}{}} \supnorm{DT}^2  \\
					&\quad + \supnorm{Dg_s} \supnorm{DK{}{}}\supnorm{D^2T}  \\
					&\le L_{-1}^2(1 + L_c)^2 \varepsilon + L_{-1}^2L_g \delta + L_{-1}^2L_g \varepsilon + L_{-1}^3 L_g(1+L_c) \delta.
\end{align*}
Inequality \cref{ BanachSpaceC2Eq8} thus becomes
\begin{align}
\supnorm{D^3 [\Theta^{}_2( \Lambda )] } &\le L_{-1}^2 ( \operatornorm{A_s} ( 1 + L_{-1}L_s) + L_g (1 + L_{-1}(1+L_c)))  \delta  \nonumber \\
							&\quad + L_{-1}^2 ( (1+L_c)^2 + L_g) \varepsilon \nonumber \\
							&= \theta_{2,3} \delta + C_3(\varepsilon),  \label{ sComponentBanachSpaceC2}
\end{align}
where $C_3(\varepsilon) \isdef  L_{-1}^2 ( (1+L_c)^2 + L_g) \varepsilon$. 

If we show that $\delta(\varepsilon) > 0$ can be chosen such that
\begin{align*}
\theta_{2,i} \delta(\varepsilon) + C_i(\varepsilon) \le \delta(\varepsilon)
\end{align*}
then we can estimate inequalities \cref{ rComponentBanachSpaceC2,, uComponentBanachSpaceC2,, sComponentBanachSpaceC2} by $\delta(\varepsilon)$ which shows that $\Gamma_1(\delta(\varepsilon))$ is invariant under $\Theta^{}$. By assumption, we have $\theta_{2,i} < 1$ for $i=1,2,3$. Therefore we can define
\begin{align}
\delta(\varepsilon) \isdef  \max_{i=1,2,3} \left\{ \frac{C_i(\varepsilon)}{1- \theta_{2,i} }\right\} > 0. \label{DeltaDefinition}
\end{align}
This gives for $i=1,2,3$
\begin{align*}
\theta_{2,i} \delta(\varepsilon) + C_i(\varepsilon) = \theta_{2,i} \delta(\varepsilon)+ (1 - \theta_{2,i}) \frac{C_i(\varepsilon)}{1 - \theta_{2,1}} \le \theta_{2,i} \delta(\varepsilon) + (1 - \theta_{2,i}) \delta(\varepsilon) = \delta(\varepsilon).
\end{align*}
Furthermore, we have by construction that $\delta(\varepsilon) \downarrow 0$ when $\varepsilon \downarrow 0$, since $C_i(\varepsilon) \downarrow 0$ as $\varepsilon \downarrow 0$ for $i=1,2,3$.
\end{proof}

\subsection{Estimates for compositions}\label{C1ContinuitySection}

We already mentioned before \cref{UniformSecondDerivativeBound} that we have to estimate expressions such as $\supnorm{D[k_u\circ(A_c + r)] - D[\tilde{k}_u \circ( A_c + \tilde{r})]}$ in terms of $\onenorm{k_u - \tilde{k}_u}$ and $\onenorm{r - \tilde{r}}$ to show that $\Theta^{}$ is a contraction with respect to the $C^1$ norm. Furthermore, we briefly showed part of the steps we would take to achieve the desired estimate. However, to show that $\Theta^{}$ is a contraction with respect to the $C^1$ norm, we must also estimate expressions such as $k_u\circ(A_c + r) - \tilde{k}_u \circ( A_c + \tilde{r})$ in terms of $\onenorm{k_u - \tilde{k}_u}$ and $\onenorm{r - \tilde{r}}$. We will therefore provide a general result which allows use to bound both $\supnorm{k_u\circ(A_c + r) - \tilde{k}_u \circ( A_c + \tilde{r})}$ and $\supnorm{D[k_u\circ(A_c + r)] - D[\tilde{k}_u \circ( A_c + \tilde{r})]}$, and the corresponding expressions in the first and third component of $\Theta^{}(\Lambda) - \Theta^{}(\tilde{\Lambda})$. As we did in the definition of $\Gamma_1$, we prefer to work with twice differentiable functions instead of differentiable functions with Lipschitz first derivative.

\begin{lemma}\label{DifferenceC0Norm} Let $X$, $Y$ and $Z$ be Banach spaces. 
\begin{enumerate}[label = {\roman*)}, ref={\cref{DifferenceC0Norm}\roman*)}]
\item \label{DifferenceC0Norm1} For $f_1 \in C_b^0(Y,Z)$, $g_1 \in C_b^1(Y,Z)$, $ f_2,g_2 \in C_b^0(X,Y)$ we have the $C^0$-estimate
\begin{align*}
\supnorm{f_1 \circ f_2 - g_1 \circ g_2 } &\le \supnorm{f_1 - g_1 } + \supnorm{Dg_1}\supnorm{f_2 - g_2}. 
\end{align*}
\item\label{DifferenceC0Norm2}   For $f_1, g_1 \in C_b^2(Y,Z)$, $ f_2,g_2 \in C_b^1(X,Y)$ we have the $C^1$-estimate
\begin{align*}
\supnorm{D [ f_1 \circ f_2] - D[  g_1 \circ g_2 ] } &\le  \supnorm{ Dg_2 } \left( \supnorm{D^2g_1}   \supnorm{f_2 - g_2} + \supnorm{ Df_1 - Dg_1 } \right)  \nonumber \\
				&\quad + \supnorm{Df_1}  \supnorm{Df_2 - Dg_2}.
\end{align*}
\end{enumerate}
\end{lemma}

\begin{proof}
\textit{i)} We find by the Mean Value Theorem
\begin{align*}
\supnorm{ f_1 \circ f_2 &- g_1 \circ g_2 } \\
				&\le \supnorm{ f_1\circ f_2 - g_1 \circ f_2 }   + \supnorm{ g_1 \circ f_2 - g_1 \circ g_2 }   \\
				&\le \supnorm{ f_1 \circ f_2 - g_1 \circ f_2 } + \bignorm{\int_0^1 Dg_1(tf_2 +(1- t) g_2) \text{d}t \left( f_2 - g_2 \right) }_0\\
				&\le \supnorm{f_1 - g_1 } + \supnorm{ Dg_1} \supnorm{  f_2 - g_2}.
\end{align*}

\textit{ii)} We use the chain rule and triangle inequality to find
\begin{align*}
\supnorm{ D [ f_1 \circ f_2] - D[ g_1 \circ g_2] } &= \supnorm{Df_1(f_2)Df_2 - Dg_1(g_2)Dg_2}  \\
				&\le \supnorm{Df_1(f_2) \left(Df_2 -Dg_2 \right)}  \\
				&\quad + \supnorm{ \left( Df_1(f_2) - Dg_1(g_2) \right) Dg_2}.
\end{align*}
We estimate the first term using the submultiplicativity of the supremum norm:
\begin{align*}
\supnorm{Df_1(f_2) \left(Df_2 -Dg_2 \right)} &\le \supnorm{Df_1} \supnorm{Df_2 - Dg_2}.
\end{align*}
For the second term, we again use submultiplicativity of the norm and estimate the factor $Df_1(f_2) - Dg_1(g_2)$ with part i) of this lemma:
\begin{align*}
\supnorm{( Df_1(f_2) &- Dg_1(g_2) ) Dg_2} \nonumber \\
			&\quad \le  \left( \supnorm{Df_1 - Dg_1} + \supnorm{D^2g_1} \supnorm{f_2 - g_2} \right) \supnorm{Dg_2}. \qedhere
\end{align*}
\end{proof}

When we estimate the third component of  $\Theta^{}(\Lambda) - \Theta^{}(\tilde{\Lambda})$ with the previous lemma, we get an estimate involving $\norm{(A_c + r)^{-1}  - (A_c - \tilde{r})^{-1}}_i$ instead of $\norm{r - \tilde{r}}_i$ for $i=0,1$. So if we want to use the previous lemma to estimate $\Theta^{}(\Lambda) - \Theta^{}(\tilde{\Lambda})$ by $\onenorm{\Lambda - \tilde{\Lambda}}$, then we must find an estimate for $\supnorm{(A_c +r )^{-1} - (A_c + \tilde{r})^{-1}}$ and $\supnorm{D(A_c +r )^{-1} - D(A_c + \tilde{r})^{-1}}$ in terms of $\supnorm{r - \tilde{r}}$ and $\supnorm{Dr - D \tilde{r}}$.

\begin{lemma}\label{DifferenceInverse} Let $r_1, r_2 \in C_b^2(X_c,X_c)$ be such that $A_c + r_i$ is a diffeomorphism with $D(A_c + r_i)^{-1} \in C_b^1(X_c, \mathcal{L}(X_c,X_c))$ for $i=1,2$. 
\begin{enumerate}[label = {\roman*)}, ref={\cref{DifferenceInverse}\roman*)}]
\item\label{DifferenceInverse1} We have the $C^0$-estimate
\begin{align*}
\supnorm{(A_c + r_1)^{-1} - (A_c + r_2)^{-1}} \le \supnorm{D(A_c + r_1)^{-1}} \supnorm{r_1 - r_2}. 
\end{align*}
\item \label{DifferenceInverse2} We have the $C^1$-estimate
\begin{align*}
\supnorm{D(A_c + r_1)^{-1} &- D(A_c + r_2)^{-1}} \nonumber \\
			&\le \supnorm{D(A_c + r_1)^{-1}} \supnorm{D(A_c + r_2)^{-1}} \Big( \supnorm{Dr_1 - Dr_2} \nonumber \\
			&\quad + \supnorm{D^2r_2} \supnorm{D(A_c + r_1)^{-1}} \supnorm{r_1 - r_2} \Big).\end{align*}
\end{enumerate}
\end{lemma}
\begin{proof}
\textit{i)} We denote $R_i = A_c +r_i$ for $i=1,2$. \Cref{DifferenceC0Norm1} implies that
\begin{align*}
\supnorm{R_1^{-1}- R_2^{-1}} &= \supnorm{R_1^{-1} \circ R_2 \circ R_2^{-1} - R_1^{-1} \circ R_1 \circ R_2^{-1} }  \\
					&\le \supnorm{DR_1^{-1}}   \supnorm{ R_2 \circ R_2^{-1} - R_1 \circ R_2^{-1} }\\
					&= \supnorm{DR_1^{-1} } \supnorm{r_1 - r_2}.
\end{align*}

\textit{ii)} For the $C^1$-estimate we denote $T_i = DR_i^{-1} \in C_b^1(X_c, \mathcal{L}(X_c,X_c))$ for $i=1,2$. Let $x \in X_c$, then we know that $T_i(x) \in \mathcal{L}(X_c, X_c)$ is invertible for $i=1,2$ by the Inverse Function Theorem. We find
\begin{align}
\operatornorm{T_1(x) - T_2(x)} &= \operatornorm{T_2(x) \left(T_2(x)^{-1} - T_1(x)^{-1} \right) T_1(x)}  \nonumber \\
						&\le \operatornorm{T_1(x)} \operatornorm{T_2(x)} \operatornorm{ T_2(x)^{-1} - T_1(x)^{-1}}. \label{DifferenceC0Norm2Eq}
\end{align}
Furthermore, the Inverse Function Theorem allows us to rewrite
\begin{align*}
T_i(x)^{-1} = \left( DR_i^{-1}(x)  \right)^{-1} = DR_i(R_i^{-1}(x)) = A_c + Dr_i(R_i^{-1}(x)).
\end{align*}
Taking the supremum over $x \in X_c$ gives
\begin{align*}
\supnorm{T_1^{-1} - T_2^{-1}} \le \supnorm{Dr_1\left( R_1^{-1} \right) - Dr_2\left( R_2^{-1}\right)}.
\end{align*}
By using \cref{DifferenceC0Norm1} we find
\begin{align*}		
\supnorm{T_1^{-1} - T_2^{-1}} &\le  \supnorm{Dr_1 - Dr_2} + \supnorm{D^2 r_2} \supnorm{R_1^{-1} - R_2^{-1}}.
\end{align*}
Then we use part i) of this lemma applied to $\supnorm{R_1^{-1} - R_2^{-1}}$ to obtain
\begin{align*}
\supnorm{T_1^{-1} - T_2^{-1}}  &\le  \supnorm{Dr_1 - Dr_2} + \supnorm{D^2r_2} \supnorm{D(A_c + r_1)^{-1}} \supnorm{r_1 - r_2}.
\end{align*}
The result now follows from taking the supremum of $x \in X_c$ in inequality \cref{DifferenceC0Norm2Eq} and using the above estimate to bound $\sup_{x \in X_c} \operatornorm{ T_2(x)^{-1} - T_1(x)^{-1}}$.
\end{proof}

To bound $\Theta^{}(\Lambda) - \Theta^{}(\tilde{\Lambda})$ by $\onenorm{\Lambda - \tilde{\Lambda}}$, we can use \cref{DifferenceC0Norm} for the first two components. For the third component, we can use the same lemma together with \cref{DifferenceInverse} for the desired bound. However, while the first and second component consists of a single composition, the third component contains a double composition.

\begin{lemma}\label{DifferenceComponentThree} Let $r,\tilde{r} \in C_b^2(X_c,X_c)$  such that $\supnorm{Dr},\supnorm{D\tilde{r}} \le L_r$ and assume that $L_r \le \operatornorm{A_c^{-1}}^{-1}$. Furthermore, let $X$ and $Y$ be Banach spaces with functions $h \in C^2_b(Y,X)$ and $f_1,f_2 \in C^2_b(X_c,Y)$.
\begin{enumerate}[label = {\roman*)}, ref={\cref{DifferenceComponentThree}\roman*)}]
\item \label{DifferenceComponentThree1} We have the $C^0$-estimate
\begin{align*}
\supnorm{h \circ f_1 \circ (A_c + r)^{-1}  &- h \circ f_2 \circ (A_c + \tilde{r})^{-1}} \nonumber \\
		&\le \supnorm{Dh} \left(  \supnorm{f_1 - f_2} + L_{-1} \supnorm{Df_2} \supnorm{ r - \tilde{r}} \right).
\end{align*}
\item \label{DifferenceComponentThree2}  If $\supnorm{D^2\tilde{r}} \le \delta(\varepsilon)$ for some $\varepsilon > 0$, then we have the $C^1$-estimate
\begin{align*}
\supnorm{D [ h \circ f_1 &\circ (A_c + r)^{-1}] - D[ h \circ f_2 \circ (A_c + \tilde{r})^{-1} ] } \nonumber \\
				&\quad \le L_{-1} \left( \supnorm{Dh} + \supnorm{D^2 h }  \supnorm{Df_2 } \right) \onenorm{ f_1 - f_2 }   \nonumber \\
				&\quad \quad +  L_{-1}^2 \left( \supnorm{D^2h} \supnorm{Df_2}^2 + \supnorm{Dh} \supnorm{D^2f_2} \right) \onenorm{r - \tilde{r}}	\nonumber \\
				&\quad \quad +   L_{-1}^2 \supnorm{Dh} \supnorm{Df_1}\left( 1 + L_{-1} \delta(\varepsilon) \right) \onenorm{r - \tilde{r}}. 
\end{align*}
\end{enumerate}
\end{lemma}
\begin{proof}
\textit{i)} For the $C^0$-estimate, we first use \cref{DifferenceC0Norm1} twice:
\begin{align*}
\supnorm{h &\circ \left( f_1 \circ (A_c + r)^{-1} \right)  - h \circ \left( f_2 \circ (A_c + \tilde{r})^{-1} \right)} \\
		&\qquad \qquad \le \supnorm{h -h } + \supnorm{Dh}\supnorm{f_1 \circ (A_c + r)^{-1}  - f_2 \circ (A_c + \tilde{r})^{-1}} \\
		&\qquad \qquad \le \supnorm{Dh} \left(  \supnorm{f_1 - f_2} + \supnorm{Df_2} \supnorm{(A_c + r)^{-1} - (A_c + \tilde{r})^{-1}} \right).
\end{align*}
Then we use \cref{DifferenceInverse1} to estimate $(A_c + r)^{-1} - (A_c + \tilde{r})^{-1}$:
\begin{align*}
\supnorm{h \circ &\left( f_1 \circ (A_c + r)^{-1} \right) - h \circ \left( f_2 \circ (A_c + \tilde{r})^{-1} \right)} \\
		&\qquad \qquad  \le \supnorm{Dh} \left(  \supnorm{f_1 - f_2} + \supnorm{Df_2} \supnorm{D(A_c + r)^{-1}} \supnorm{r - \tilde{r}} \right).
\end{align*}
Recall from \cref{SmallnessConditions} that $L_r < \operatornorm{A_c^{-1}}^{-1}$ implies that $\supnorm{D(A_c + r)^{-1}} \le L_{-1}$, which proves the first estimate. 

\textit{ii)} To prove the $C^1$-estimate, we follow the same steps using \cref{DifferenceC0Norm2,DifferenceInverse2} instead of \cref{DifferenceC0Norm1,DifferenceInverse1} respectively.  We then use the estimates $\supnorm{D^2 \tilde{r}} \le \delta(\varepsilon)$, $\supnorm{D^i f_1 - D^if_2} \le \onenorm{f_1 - f_2}$ and $\supnorm{D^i r - D^i \tilde{r}} \le \onenorm{r - \tilde{r}}$ for $i =1,2$. From this, the desired estimate follows.
\end{proof}
\begin{remark}\label{UniquenessInverseR} The assumptions on $r$ and $f_1$ can be weakened in part i). We only need the assumptions that $A_c + r$ is a homeomorphism and $f_1$ is continuous and bounded.
\end{remark}

\subsection{A contraction}

Following our proof scheme for \cref{MainTheorem}, which we described in \cref{ProofScheme}, we want to show  that our fixed point operator $\Theta^{} : \Gamma_1(\delta(\varepsilon)) \to \Gamma_1(\delta(\varepsilon))$, which is defined in \cref{FixedPointOperator}, has a $C^1$ fixed point. We note that in \cref{C1Continuity} we will impose an upper bound on the second derivatives of the nonlinearities $k_c : X_c \to X_c$ and $g: X \to X$, whereas we only assume boundedness of the second derivatives in \cref{MainTheorem}. However, we will see in \cref{SecondDerivativeBounded} that we can always find a scaling such that the second derivative is sufficiently small.

\begin{theorem}\label{C1Continuity} Assume that $L_g$ and $L_c$ are small in the sense of \cref{SmallnessConditions} for $n=2$. There exists an $\varepsilon_0 >0 $ such for all $\varepsilon < \varepsilon_0$ it holds that  if $\supnorm{D^2 g}, \supnorm{D^2k_c} \le  \varepsilon$, then  $\Theta^{} :\Gamma_1(\delta(\varepsilon)) \to \Gamma_1(\delta(\varepsilon))$ is a contraction with respect to the $C^1$ norm. 
\end{theorem}
\begin{proof}
Let $\varepsilon > 0$ and $\supnorm{D^2 g}, \supnorm{D^2 k_c} \le \varepsilon$. Let $\Lambda = \tmatrix{r}{k_u}{k_s}, \tilde{\Lambda}= \tmatrix{\tilde{r}}{\tilde{k}_u}{\tilde{k}_s} \in \Gamma_1(\delta(\varepsilon))$. We denote $R = A_c + r$ and $\tilde{R} = A_c + \tilde{r}$.

Our proof that $\Theta^{}$ is a $C^1$-contraction is divided in three steps.
\begin{enumerate}[label=\Alph*)]
\item\label{C1Continuity1} We prove that $\Theta^{}$ is a contraction with respect to the $C^0$ norm, independent of $\varepsilon$.

\item\label{C1Continuity2} We show the existence of a constant $\theta_1(\varepsilon)$ such that 
\begin{align*}
\supnorm{D[\Theta^{}(\Lambda)] - D[\Theta^{}(\tilde{\Lambda})]} \le \theta_1(\varepsilon) \onenorm{\Lambda - \tilde{\Lambda}}.
\end{align*}

\item\label{C1Continuity3} We show that $\varepsilon > 0$ can be chosen so that $\theta_1(\varepsilon) <1$,  thus proving that $\Theta^{}$ is a contraction with respect to the $C^1$ norm. 
\end{enumerate}

\textbf{Step \cref{C1Continuity1}} We want to find $\theta_0 <1$ such that 
\begin{align*}
\supnorm{\Theta^{}(\Lambda) - \Theta^{}(\tilde{\Lambda})} \le \theta_0 \supnorm{\Lambda - \tilde{\Lambda}}.
\end{align*} 
Recall from equation \cref{FixedPointOperator} that
\begin{align*}
\Theta^{}( \Lambda) =  \Theta^{} \begin{pmatrix} r \\ k_u \\ k_s \end{pmatrix} = \begin{pmatrix*}[l]
			A_c k_c +g_c \circ K{}{} -k_c \circ (A_c + r)   \\
			A_u^{-1}k_u \circ (A_c + r)  - A_u^{-1} g_u \circ K{}{} \\
			A_s k_s\circ (A_c + r)^{-1} + g_s \circ K{}{} \circ (A_c +r)^{-1} 
			\end{pmatrix*}.
\end{align*}
We will find the contraction constant component-wise, i.e we will show that
\begin{align*}
\supnorm{\Theta^{}_i(\Lambda) - \Theta^{}_i(\tilde{\Lambda})} \le \theta_{0,i} \supnorm{\Lambda - \tilde{\Lambda}}
\end{align*} 
with $\theta_{0,i}$ given explicitly in equation \cref{ContractionConstants1,ContractionConstants2,ContractionConstants3} for $i=1,2,3$.

\underline{\textit{$r$-component:}} We start with
\begin{align}
\supnorm{ \Theta^{}_1 (\Lambda) - \Theta^{}_1(\tilde{\Lambda} ) } &\le \supnorm{g_c \circ K{}{} - g_c \circ \tilde{K{}{}}} + \supnorm{k_c \circ R - k_c \circ \tilde{R}} . \label{rComponentEqC0}
\end{align} 
By using \cref{DifferenceC0Norm1} we find the estimate
\begin{align}
\supnorm{g_c \circ K{}{}- g_c \circ \tilde{K{}{}}} &\le \supnorm{Dg_c} \supnorm{K{}{} - \tilde{K{}{}}} 
							\le L_g \supnorm{ K{}{} - \tilde{K{}{}}} 
							\le L_g \supnorm{\Lambda - \tilde{\Lambda}}. \label{rComponentEqC01}
\end{align}
Here we recall that $\supnorm{Dg_c} \le L_g$, which follows from assumption \cref{MainTheoremKnownBounds } of \cref{MainTheorem}. 
Likewise, we estimate
\begin{align}
\supnorm{k_c \circ R- k_c \circ \tilde{R}}	&\le \supnorm{Dk_c} \supnorm{R - \tilde{R}} = \supnorm{Dk_c} \supnorm{r - \tilde{r}} \le L_c \supnorm{\Lambda - \tilde{\Lambda}}. \label{rComponentEqC02}
\end{align}
Thus inequality \cref{rComponentEqC0} together with estimates  \cref{rComponentEqC01,rComponentEqC02} gives
\begin{align}
\supnorm{ \Theta^{}_1 (\Lambda) - \Theta^{}_1(\tilde{\Lambda})} &\le \left( L_g + L_c \right) \supnorm{\Lambda - \tilde{\Lambda}} =\theta_{0,1} \supnorm{\Lambda - \tilde{\Lambda}}. \label{ C1ContractionEquation1}
\end{align}

\underline{\textit{$k_u$-component:}} Similarly, we have
\begin{align}
\supnorm{ \Theta^{}_2 (\Lambda) - \Theta^{}_2 (\tilde{\Lambda}) } &\le \supnorm{A_u^{-1} \left( k_u \circ R -  \tilde{k}_u \circ \tilde{R} \right)} + \supnorm{A_u^{-1} \left( g_u \circ K{}{} - g_u \circ \tilde{K{}{}} \right)}  \nonumber \\
									&\le \operatornorm{A_u^{-1}} \left(  \supnorm{ k_u \circ R -  \tilde{k}_u \circ \tilde{R}} + \supnorm{g_u \circ K{}{} - g_u \circ \tilde{K{}{}}} \right). \label{uComponentEqC0}
\end{align}
We again use  \cref{DifferenceC0Norm1}, which gives
\begin{align*}
\supnorm{g_u \circ K{}{}- g_u \circ \tilde{K{}{}}} &\le \supnorm{Dg_u} \supnorm{K{}{} - \tilde{K{}{}}}  \le L_g \supnorm{ \Lambda - \tilde{\Lambda}}, \\
\supnorm{k_u \circ R - \tilde{k}_u \circ \tilde{R}} &\le  \supnorm{k_u - \tilde{k}_u}  + \supnorm{D\tilde{k}_u} \supnorm{R - \tilde{R}} \le (1 +  L_u ) \supnorm{\Lambda - \tilde{\Lambda}}. 
\end{align*}
Here we used that $\supnorm{D\tilde{k}_u} \le L_u$, which follows from the fact that $\Gamma_1(\delta(\varepsilon)) \subset \Gamma_0$, with the latter space defined before \cref{LemmaBanachSpace}. Thus inequality \cref{uComponentEqC0} becomes
\begin{align}
\supnorm{ \Theta^{}_2 (\Lambda) - \Theta^{}_2 (\tilde{\Lambda}) } & \le \operatornorm{A_u^{-1}} \left( 1 + L_{u} + L_g \right)\supnorm{\Lambda - \tilde{\Lambda}} =\theta_{0,2} \supnorm{\Lambda - \tilde{\Lambda}}. \label{ C1ContractionEquation2}
\end{align}

\underline{\textit{$k_s$-component:}} Let $T = \left(A_c + r \right)^{-1}$ and $\tilde{T}= (A_c + \tilde{r})^{-1}$, then we have
\begin{align}
\supnorm{ \Theta^{}_3 (\Lambda) - \Theta^{}_3( \tilde{\Lambda}) }  &\le \supnorm{A_s  \circ  k_s \circ T - A_c \circ \tilde{k}_s \circ \tilde{T}} + \supnorm{g_s \circ K{}{} \circ T - g_s \circ \tilde{K{}{}} \circ \tilde{T}}. \label{sComponentEqC0}
\end{align}
We use \cref{DifferenceComponentThree1}, where the condition $L_r < \operatornorm{A_c^{-1}}^{-1}$ is satisfied by \cref{SmallnessConditions}, to obtain
\begin{align*}
\supnorm{A_s  \circ  k_s \circ T - A_c \circ \tilde{k}_s \circ \tilde{T}} &\le \operatornorm{A_c} \left( \supnorm{k_s - \tilde{k}_s} + L_{-1} \supnorm{D\tilde{k}_s} \supnorm{r- \tilde{r}} \right) \\
				&\le \operatornorm{A_c} \left( 1 + L_{-1} L_s \right) \supnorm{\Lambda - \tilde{\Lambda}},\\
\supnorm{g_s \circ K{}{} \circ T - g_s \circ \tilde{K{}{}} \circ \tilde{T}} &\le \supnorm{Dg_s} \left( \supnorm{K{}{} - \tilde{K{}{}}} + L_{-1} \supnorm{D\tilde{K{}{}}} \supnorm{r- \tilde{r}} \right) \\
				&\le L_g \left( 1 + L_{-1} (1 + L_c ) \right) \supnorm{\Lambda - \tilde{\Lambda}}.
\end{align*}
We used $\supnorm{ D \tilde{K{}{}}} \le 1 + L_c$, see \cref{EstimateDk}, in the last estimate.
Thus inequality \cref{sComponentEqC0} becomes
\begin{align}
\supnorm{ \Theta^{}_3 (\Lambda) - \Theta^{}_3( \tilde{\Lambda}) } &\le  \left( \operatornorm{A_s} \left( 1 + L_s L_{-1} \right) +  L_g \left( 1 + L_{-1} \left( 1 + L_c \right) \right) \right) \supnorm{\Lambda - \tilde{\Lambda}} \nonumber \\
									&=  \theta_{0,3} \supnorm{ \Lambda - \tilde{\Lambda}}. \label{ C1ContractionEquation3}
\end{align}

\underline{\textit{Contraction constant:}} We can now estimate $\supnorm{\Theta^{}(\Lambda) - \Theta^{}(\tilde{\Lambda})}$ with inequalities \cref{ C1ContractionEquation1,, C1ContractionEquation3,, C1ContractionEquation2}. We obtain
\begin{align}
\supnorm{\Theta^{}(\Lambda) - \Theta^{}(\tilde{\Lambda})} &= \max_{i=1,2,3} \left\{ \supnorm{\Theta^{}_i(\Lambda) - \Theta^{}_i(\tilde{\Lambda})} \right\} \nonumber \\
										&\le \max_{i=1,2,3} \left\{ \theta_{0,i} \supnorm{\Lambda - \tilde{\Lambda}} \right\} \nonumber \\
										&= \theta_0 \supnorm{\Lambda - \tilde{\Lambda}}. \label{C1ContractionEq1}
\end{align}
Here we define $\theta_0 \isdef \max\left\{ \theta_{0,1}, \theta_{0,2}, \theta_{0,3} \right\}$.  Since it is assumed that \cref{SmallnessConditions} holds for $n=2$, we have $\theta_{0,i} < 1$ and thus $\theta_0 <1$. This implies that $\Theta^{}$ is a contraction with respect to the $C^0$ norm. \vspace{1 \baselineskip}

\textbf{Step \cref{C1Continuity2}} Analogous to step \cref{C1Continuity1}, we want to prove the component-wise inequality
\begin{align*}
\supnorm{ D[\Theta^{}_i(\Lambda)] - D[\Theta^{}_i(\tilde{\Lambda})]} \le \left( \theta_{1,i} + C_{1,i}(\varepsilon) \right) \onenorm{\Lambda - \tilde{\Lambda}},
\end{align*} 
with $\theta_{1,i}$ defined in \cref{ContractionConstants1,ContractionConstants2,ContractionConstants3} and $C_{1,i}$ defined below in the proof. We note that $\Lambda, \tilde{\Lambda} \in \Gamma_1(\delta(\varepsilon))$, so we have $\supnorm{D^2r}, \supnorm{D^2k_u}, \supnorm{D^2k_s} \le \delta(\varepsilon)$.

\underline{\textit{$r$-component:}} We start with
\begin{align}
\supnorm{D[\Theta^{}_1 (\Lambda)] &- D[\Theta^{}_1(\tilde{\Lambda})]} \nonumber \\
									&\le \supnorm{D[ g_c \circ K{}{}] - D[ g_c  \circ  \tilde{K{}{}}]} + \supnorm{D[ k_c \circ R ] - D[ k_c \circ \tilde{R}]}. \label{rComponentEqC1}
\end{align}
We infer from \cref{DifferenceC0Norm2} that
\begin{align}
\supnorm{D[ g_c \circ K{}{} ] - D[ g_c \circ \tilde{K{}{}}]}	&\le \supnorm{D\tilde{K{}{}}} \supnorm{D^2 g_c} \supnorm{K{}{} - \tilde{K{}{}}} + \supnorm{Dg_c} \supnorm{DK{}{} - D \tilde{K{}{}}} \nonumber \\
								&\le \left( \supnorm{D\tilde{K{}{}}} \supnorm{D^2g_c} + \supnorm{Dg_c}\right)  \onenorm{\Lambda - \tilde{\Lambda}}  \nonumber \\
								&\le \left( \left( 1 + L_c \right) \varepsilon + L_g \right) \onenorm{\Lambda - \tilde{\Lambda}},  \label{rComponentEqC11}
\end{align}
where we have used \cref{EstimateDk}. Likewise, we find the estimate
\begin{align}
\supnorm{D[ k_c \circ R ] - D[ k_c \circ \tilde{R} ]}	&\le \supnorm{D\tilde{R}} \supnorm{D^2k_c} \supnorm{r - \tilde{r}}+ \supnorm{Dk_c} \supnorm{Dr  - D\tilde{r}}  \nonumber \\
								&\le \left(  \left( \operatornorm{A_c} + L_{r} \right) \varepsilon + L_c\right) \onenorm{\Lambda - \tilde{\Lambda}}. \label{rComponentEqC12}
\end{align}
Thus inequality \cref{rComponentEqC1} together with estimates  \cref{rComponentEqC11,rComponentEqC12} gives
\begin{align}
\supnorm{D[\Theta^{}_1 (\Lambda)] - D[\Theta^{}_1(\tilde{\Lambda})]} &\le  \left( L_g + L_c +\left(  1 + L_c  +  \operatornorm{A_c} + L_{r}  \right)  \varepsilon \right) \onenorm{\Lambda - \tilde{\Lambda}} \nonumber \\
								&= \left( \theta_{1,1} + C_{1,1}(\varepsilon) \right) \onenorm{\Lambda - \tilde{\Lambda}}, \label{ C1ContractionEquation4}
\end{align}
where we define $C_{1,1}(\varepsilon) \isdef \left(  1 + L_c  +  \operatornorm{A_c} + L_{r}  \right)  \varepsilon$.

\underline{\textit{$k_u$-component:}} Similarly, we have
\begin{align}
\supnorm{D[\Theta^{}_2 (\Lambda)] - D[\Theta^{}_2(\tilde{\Lambda})]} 
									&\le \operatornorm{A_u^{-1}} \supnorm{D [ g_u \circ K{}{} ] - D[g_u \circ \tilde{K{}{}}]}  \nonumber \\
									&\quad + \operatornorm{A_u^{-1}}\supnorm{D[ k_u \circ R ] - D[  \tilde{k}_u \circ \tilde{R}]}. \label{uComponentEqC1}
\end{align}
Using \cref{DifferenceC0Norm2} we get 
\begin{align*}
\supnorm{D [ g_u \circ K{}{} ] - D[g_u \circ \tilde{K{}{}}]} 	&\le  \supnorm{D\tilde{K{}{}}} \supnorm{D^2 g_u} \supnorm{K{}{} - \tilde{K{}{}}}  + \supnorm{Dg_u} \supnorm{DK{}{} - D \tilde{K{}{}}}    \\
									&\le  \left( \left( 1 + L_c \right) \varepsilon + L_g \right) \onenorm{\Lambda - \tilde{\Lambda}}, \\
\supnorm{D[ k_u \circ R ] - D[  \tilde{k}_u \circ \tilde{R}]} &\le  \supnorm{D\tilde{R}} \left( \supnorm{D^2 \tilde{k}_u}  \supnorm{r - \tilde{r}}  + \supnorm{Dk_u - D\tilde{k}_u} \right)  \\
									&\quad +  \supnorm{D k_u}  \supnorm{Dr - D \tilde{r}}    \\
									&\le  \left( \left( \operatornorm{A_c} + L_r \right) ( 1 + \delta(\varepsilon) ) + L_u \right) \onenorm{\Lambda - \tilde{\Lambda}}. 
\end{align*} 
Thus inequality \cref{uComponentEqC1} becomes
\begin{align}
\supnorm{D\Theta^{}_2 (\Lambda) - D \Theta^{}_2(\tilde{\Lambda}) }	&\le \left( \theta_{1,2} + C_{1,2}(\varepsilon) \right) \onenorm{\Lambda - \tilde{\Lambda}}, \label{ C1ContractionEquation5}
\end{align}
where we define $C_{1,2}(\varepsilon) \isdef \operatornorm{A_u^{-1}} \left( L_g( 1 + L_c) \varepsilon + (\operatornorm{A_c} + L_r ) \delta(\varepsilon) \right)$.

\underline{\textit{$k_s$-component:}} Recall that $T = (A_c + r)^{-1}$ and $\tilde{T} = (A_c + \tilde{r})^{-1}$, then we have
\begin{align}
\supnorm{D[\Theta^{}_3 (\Lambda)] - D[\Theta^{}_3(\tilde{\Lambda})]} &\le\supnorm{D[A_s \circ k_s \circ T] - D[A_s \circ \tilde{k}_s \circ \tilde{T}]} \nonumber \\
									&\quad +  \supnorm{D[g_s \circ K{}{} \circ T] - D[g_s \circ \tilde{K{}{}} \circ \tilde{T}]}. \label{sComponentEqC1}
\end{align}
We will estimate both terms with \cref{DifferenceComponentThree2}. For the first term, we note that $\supnorm{DA_s} = \operatornorm{A_s}$ and $\supnorm{D^2A_s} = 0$, which gives us
\begin{align}
\supnorm{D[A_s \circ k_s \circ T] &- D[A_s \circ \tilde{k}_s \circ \tilde{T}]} 	\nonumber \\
									&\qquad \le L_{-1} \operatornorm{A_s}  \onenorm{k_s - \tilde{k}_s} + L_{-1}^2  \operatornorm{A_s} \supnorm{D^2 \tilde{k}_s}  \onenorm{r - \tilde{r}}  \nonumber \\
									&\qquad \quad +L_{-1}^2 \operatornorm{A_s} \supnorm{Dk_s} ( 1 + L_{-1} \delta(\varepsilon)) \onenorm{r - \tilde{r}} \nonumber \\
									&\qquad \le \operatornorm{A_s} L_{-1} \left(1 + L_{-1}L_s \right) \onenorm{\Lambda - \tilde{\Lambda}} \label{sComponentEqTheta1} \\
									&\qquad \quad + \operatornorm{A_s} L_{-1}^2 \delta(\varepsilon) \left(1  +L_{-1} L_s  \right) \onenorm{\Lambda - \tilde{\Lambda}}, \label{sComponentEqThetaGamma1}
\end{align}
where we grouped the terms with and without a factor $\delta(\varepsilon)$. The second term in \cref{sComponentEqC1} involves the first and second derivative of $\tilde{K{}{}}$. We estimate the first derivate again with $1 + L_c$ and we estimate the second derivative with  $\supnorm{D^2 \tilde{K{}{}}} = \max\{ \supnorm{D^2 k_c}, \supnorm{D^2\tilde{k}_u}, \supnorm{D^2\tilde{k}_s} \} \le \max\{ \varepsilon , \delta(\varepsilon) \} \isfed \gamma(\varepsilon)$. Hence we obtain
\begin{align}
\supnorm{D[g_s \circ K{}{} \circ T] &- D[g_s \circ \tilde{K{}{}} \circ \tilde{T}]}  \nonumber \\
									&\qquad \le L_{-1} \left( \supnorm{Dg_s} + \supnorm{D^2 g_s} \supnorm{D \tilde{K{}{}}} \right) \onenorm{K{}{} - \tilde{K{}{}}} \nonumber \\
									&\qquad \quad + L_{-1}^2 \left( \supnorm{D^2 g_s} \supnorm{D \tilde{K{}{}}}^2 + \supnorm{Dg_s} \supnorm{D^2 \tilde{K{}{}}} \right) \onenorm{r - \tilde{r}}  \nonumber \\
									&\qquad \quad +L_{-1}^2 \supnorm{Dg_s} \supnorm{DK{}{}} ( 1 + L_{-1} \delta(\varepsilon)) \onenorm{r - \tilde{r}} \nonumber \\
									&\qquad \le L_{-1}\left( L_g + L_{-1} L_g (1 + L_c )  \right) \onenorm{\Lambda - \tilde{\Lambda}} \label{sComponentEqTheta2} \\
									&\qquad \quad + L_{-1} \left( (1+ L_c) \varepsilon + L_{-1}(1 + L_c)^2 \varepsilon  \right)  \onenorm{\Lambda - \tilde{\Lambda}} \nonumber  \\
									&\qquad \quad + L_{-1}^2  \left( L_g \gamma(\varepsilon) + L_{-1} L_g (1 + L_c) \delta(\varepsilon) \right)  \onenorm{\Lambda - \tilde{\Lambda}}. \label{sComponentEqThetaGamma2}
\end{align}
Here we again grouped the terms with and without $\varepsilon$. We see that \cref{sComponentEqTheta1,sComponentEqTheta2} together are $\theta_{1,3} \onenorm{\Lambda - \tilde{\Lambda}}$. Likewise, we can estimate \cref{sComponentEqThetaGamma1,sComponentEqThetaGamma2} together by $\theta_{2,3} \gamma(\varepsilon) \onenorm{\Lambda - \tilde{\Lambda}}$ as $\delta(\varepsilon) \le \gamma(\varepsilon)$. Then inequality \cref{sComponentEqC1} reduces to
\begin{align}
\supnorm{D\Theta^{}_3 (\Lambda) - D \Theta^{}_3(\tilde{\Lambda}) }	&\le \left( \theta_{1,3} + C_{1,3}(\varepsilon) \right) \onenorm{\Lambda - \tilde{\Lambda}}, \label{ C1ContractionEquation6}
\end{align}
where we define $C_{1,3}(\varepsilon) \isdef \theta_{2,3}\gamma(\varepsilon)+ L_{-1}(1+ L_c)\varepsilon + L_{-1}^2(1 + L_c )^2 \varepsilon$.

\underline{\textit{Lipschitz constant:}} Inequalities \cref{ C1ContractionEquation4,, C1ContractionEquation5,, C1ContractionEquation6} give
\begin{align}
\supnorm{D[\Theta^{}(\Lambda)] - D[\Theta^{}(\tilde{\Lambda})]} &= \max_{i=1,2,3} \left\{ \supnorm{D[\Theta^{}_i(\Lambda)] - D[\Theta^{}_i(\tilde{\Lambda})]} \right\} \nonumber \\
										&\le \max_{i=1,2,3} \left\{ \left( \theta_{1,i} + C_{1,i}(\varepsilon) \right) \onenorm{\Lambda - \tilde{\Lambda}} \right\} \nonumber \\
										&\le \theta_1(\varepsilon) \onenorm{\Lambda - \tilde{\Lambda}}. \label{C1ContractionEq2}
\end{align}
Here we define 
\begin{align}
\theta_1(\varepsilon) \isdef \max\left\{ \theta_{1,1}, \theta_{1,2}, \theta_{1,3} \right\} + \max\{C_{1,1}(\varepsilon), C_{1,2}(\varepsilon),C_{1,3}(\varepsilon)\}. \label{DefinitionTheta1}
\end{align}

\textbf{Step \cref{C1Continuity3}} From \cref{SmallnessConditions} it follows that 
\begin{align*}
\max\left\{ \theta_{1,1}, \theta_{1,2}, \theta_{1,3} \right\} < 1.
\end{align*}
As $\delta(\varepsilon) \downarrow 0$ and thus also $\gamma(\varepsilon) \downarrow 0$ when $\varepsilon \downarrow 0$, see \cref{UniformSecondDerivativeBound}, we have
\begin{align*}
\lim_{\varepsilon \to 0} \max\{C_{1,1}(\varepsilon), C_{1,2}(\varepsilon),C_{1,3}(\varepsilon)\} = 0.
\end{align*}
We infer that
\begin{align*}
\lim_{\varepsilon \to 0} \theta_1(\varepsilon) = \max\left\{ \theta_{1,1}, \theta_{1,2}, \theta_{1,3} \right\} < 1.
\end{align*}
Hence, we can find an $\varepsilon_0 > 0 $ such that $\theta_1(\varepsilon) < 1$ for all $\varepsilon < \varepsilon_0$. Then estimates  \cref{C1ContractionEq1,C1ContractionEq2} imply that
\begin{align*}
\onenorm{\Theta^{}(\Lambda) - \Theta^{}(\tilde{\Lambda})} &= \max\left\{ \supnorm{\Theta^{}(\Lambda) - \Theta^{}(\tilde{\Lambda})}, \supnorm{D[\Theta^{}(\Lambda)] - D[\Theta^{}(\tilde{\Lambda})]} \right\} \\
											&\le \max \left\{ \theta_0 \supnorm{\Lambda - \tilde{\Lambda}}, \theta_1(\varepsilon) \onenorm{\Lambda - \tilde{\Lambda}} \right\} \\
											&\le \lambda_1 \onenorm{\Lambda - \tilde{\Lambda}}.
\end{align*}
We define the contraction constant $\lambda_1 \isdef \max\{\theta_0 , \theta_1(\varepsilon) \}$, which is smaller than $1$ for $\varepsilon < \varepsilon_0$ by our previous discussion. We conclude that $\Theta^{} : \Gamma_1(\delta(\varepsilon)) \to \Gamma_1(\delta(\varepsilon))$ is a contraction with respect to the $C^1$ norm for $\varepsilon < \varepsilon_0$.
\end{proof}

We can now prove the existence of a $C^1$ center manifold under the assumption that the second derivative of $k_c$ and $g$ are small enough. As we will see in \cref{SecondDerivativeBounded}, we can always find a scaling such that these second derivatives will be sufficiently small.

\begin{corollary}\label{C1Existence}
Let $\varepsilon >0$ be such that $F{}{} :X \to X$ satisfies the assumptions of \cref{MainTheorem,C1Continuity}. Then the conclusion of \cref{MainTheorem} holds for $K{}{} \in C^1(X_c,X)$ and $r \in C^1(X_c,X_c)$. In particular, the image of $K{}{}$ is a $C^1$ center manifold for $F{}{}$.
\end{corollary}
\begin{proof}
By assumption, $\varepsilon > 0$ is such that $\Theta^{}$ is a contraction. In \cref{C1Continuity,FixedPointOperatorProposition} we proved the existence of a conjugacy $K{}{}$ and conjugate dynamics $A_c + r$. Furthermore, from the definition of $\Gamma_0$, it follows that $K{}{}$ and $r$ satisfy the properties \cref{PropertiesK} and \cref{PropertiesR} respectively. In particular, it follows that image of $K{}{}$ is invariant under $F{}{}$ and tangent to $X_c$ at $0$, hence the image of $K{}{}$ is a $C^1$ center manifold for $F{}{}$.
\end{proof}

\section{A \texorpdfstring{$C^2$}{C2} center manifold}\label{C2ContinuitySection}

Now that we have a $C^1$ conjugacy, the third step in our proof scheme in \cref{ProofScheme} is showing that the conjugacy is $C^2$. We will prove the equivalent statement that the derivative of the $C^1$ conjugacy is also $C^1$. For this, we define another fixed point operator acting on $C^1$ functions, and show that its fixed point is the derivative of the conjugacy from \cref{C1Existence}.

\subsection{A new fixed point operator}

We first note that $\Theta^{}$ is a contraction with respect to the $C^1$ norm on $\Gamma_1(\delta(\varepsilon))$, a set that is not closed with respect to the $C^1$ norm. That means that the fixed point of $\Theta^{}$ lies in the $C^1$ closure of $\Gamma_1(\delta(\varepsilon))$, which is enclosed by $\Gamma_0$. 

Let $\Lambda = \tmatrix{r}{k_u}{k_s} \in \Gamma_0$ denote any fixed point of $\Theta^{}$, i.e. $\Lambda$ consists of three $C^1$ functions and we have
\begin{align*}
\begin{pmatrix}r \\ k_s \\ k_s \end{pmatrix} = \Lambda = \Theta^{}(\Lambda) = \begin{pmatrix*}[l]
	A_c k_c +g_c \circ K{}{} -k_c \circ (A_c +r)  \\
	A_u^{-1}k_u \circ (A_c + r)  - A_u^{-1} g_u \circ K{}{} \\
	 A_s k_s\circ (A_c + r)^{-1} + g_s \circ K{}{} \circ (A_c +r)^{-1}
\end{pmatrix*}.
\end{align*}
We can therefore take the derivative at both sides of the equation, which gives
\begin{align}
\begin{pmatrix} Dr \\ Dk_u \\ Dk_s \end{pmatrix}
 &= \begin{pmatrix*}[l] A_c  Dk_c + Dg_c(K{}{})  DK{}{} - Dk_c\left(R\right)  DR \\
-A_u^{-1}  Dg_u(K{}{})  DK{}{} + A_u^{-1}  Dk_u(R)  DR \\
 A_s  Dk_s(T)DT  + Dg_s(K{}{} \circ T)  DK{}{}(T)  DT
 \end{pmatrix*}, \label{FixedPointOperatorIntroduction}
\end{align}
where we define $R \isdef A_c + r$ and $T \isdef (A_c + r)^{-1}$, notation that we will use throughout the rest of the paper. To express $DT = D(A_c + R)^{-1}$ in terms of $r$ and $Dr$, we use the Inverse Function Theorem and write 
\begin{align*}
DT(x) = D(A_c + r)^{-1}(x) = \left( D R(R^{-1}(x)) \right)^{-1} = \left( DR(T(x)) \right)^{-1}.
\end{align*}
This motivates us to introduce for $\rho : X_c \to \mathcal{L}(X_c,X_c)$ the functions
\begin{align}
\begin{aligned}
P_\rho :\ &X_c \to \mathcal{L}(X_c , X_c)  \\  
		&x \mapsto A_c + \rho(x)
\end{aligned} && \text{and} &&
\begin{aligned}
Q_\rho : \  & X_c \mapsto \mathcal{L}(X_c,X_c) \\ 
&  x \mapsto (P_\rho (T(x)))^{-1}
\end{aligned} \label{FunctionsFixedPointOperatorC2}
\end{align}
so that we can write $DT(x) = Q_{Dr}(x)$ and $DR(x) = P_{Dr}(x)$. In view of \cref{FixedPointOperatorIntroduction} we use these functions to introduce the fixed point operator
\begin{align}
\Theta^{[2]} : \begin{pmatrix} \rho \\ \kappa_u \\ \kappa_s \end{pmatrix} \mapsto
\begin{pmatrix*}[l]
A_c  Dk_c + Dg_c(K{}{})  \kappa - Dk_c\left(R\right)  P_\rho \\
-A_u^{-1}  Dg_u(K{}{})  \kappa + A_u^{-1}  \kappa_u(R)  P_\rho \\
A_s  \kappa_s(T)  Q_\rho  + Dg_s(K{}{} \circ T)  \kappa(T)  Q_\rho
\end{pmatrix*} \label{FixedPointOperatorC2Equation}
\end{align}
where $\kappa = \tmatrix{\operatorname{Id} + Dk_c}{\kappa_u}{\kappa_s}$ and $\tmatrix{r}{k_u}{k_s}\in \Gamma_0$ is a fixed point of $\Theta^{}$. To summarize, we have the following proposition:

\begin{proposition}\label{FixedPointOperatorC2} Let $\Lambda \in \Gamma_0$ be any fixed point of the operator $\Theta^{}$. Then $D \Lambda$ is a fixed point of the operator $\Theta^{[2]}$ defined in \cref{FixedPointOperatorC2Equation}.
\end{proposition}
\begin{proof}
This follows immediately from the above discussion.
\end{proof}

We want to use $\Theta^{[2]}$ to show that $D \Lambda$ is $C^1$ instead of only $C^0$. To this end, we want to show that $\Theta^{[2]}$ is a contraction in $C^1$ on a suitable set of $C^1$ functions, and show that its fixed point in this set is $D\Lambda$. We therefore want to restrict $\Theta^{[2]}$ to a space similar to $\Gamma_1(\delta(\varepsilon))$. In particular, we want to reflect that $\Theta^{}$ is a fixed point operator acting on functions and $\Theta^{[2]}$ is a fixed point operator acting on derivatives. So where functions in $\Gamma_1(\delta(\varepsilon))$ have restrictions on the first and second derivative, we want the same restrictions on the function and its derivative in our new space respectively. Therefore, let $\delta >0$, and define the set
\begin{align}
 \Gamma_2(\delta) \isdef \left\{ \mathcal{M} = \begin{pmatrix} \rho \\ \kappa_u \\ \kappa_s \end{pmatrix} \in C^1_b \left(X_c, \mathcal{L}(X_c , X)  \right) \ \middle| \   \begin{matrix*}[l] 
\mathcal{M}(0) = 0, \\
\supnorm{ \rho } \le L_{r} \\ 
\supnorm{\kappa_u}\le  L_{u}  \\ 
\supnorm{  \kappa_s } \le  L_{s} \\
 \supnorm{ D\mathcal{M} } \le \delta   \\
\end{matrix*}  \right\}. \label{BanachSpaceC2}
\end{align}

\begin{proposition}
Let $\varepsilon > 0$ and assume that $\supnorm{D^2g}, \supnorm{D^2k_c} \le \varepsilon$. Furthermore, assume that $L_g$ and $L_c$ are small in the sense of \cref{SmallnessConditions} for $n=2$. Then, for $\delta(\varepsilon) > 0$ from \cref{UniformSecondDerivativeBound}, the set $\Gamma_2(\delta(\varepsilon))$ is invariant under $\Theta^{[2]}$.
\end{proposition}
\begin{proof}
The proof follows from similar estimates as in \cref{LemmaBanachSpace} for the bounds on $\rho$, $\kappa_u$ and $\kappa_s$ as well as for $\mathcal{M}(0) = 0$. The bound on the second derivative follows from the same estimates as in \cref{UniformSecondDerivativeBound}. We will illustrate this for the derivative of the second component, i.e. we will show that $D\Theta^{[2]}_1(\mathcal{M})$ is bounded by $\delta(\varepsilon)$ for $\mathcal{M} \in \Gamma_2(\delta(\varepsilon))$.

We start as we did in \cref{UniformSecondDerivativeBound} with
\begin{align*}
\supnorm{D \Theta^{[2]}_1(\mathcal{M})} &\le \supnorm{D[A_c D k_c]} + \supnorm{D[ Dg_c(K{}{}) \kappa]} + \supnorm{D[Dk_c(R) P_\rho]}.
\end{align*}
We again estimate the terms separately:
\begin{align*}
\supnorm{A_c D^2k_c}  &= \operatornorm{A_c} \supnorm{Dk_c} \le \operatornorm{A_c} \varepsilon, \\
\supnorm{D[Dg_c(K{}{}) \kappa]}&\le (1 + L_c)^2 \varepsilon + L_g (\varepsilon + \delta(\varepsilon)), \\
\supnorm{D[Dk_c(R)P]} &\le (\operatornorm{A_c} + L_r )^2 \varepsilon + L_c \delta(\varepsilon),
\end{align*}
where we used $D\kappa \le \varepsilon + \delta(\varepsilon)$ as we did in \cref{ BanachSpaceC2Eq3}. All together, we find
\begin{align*}
\supnorm{D\Theta^{[2]}_1(\mathcal{M})} \le \theta_{2,1} \delta(\varepsilon) + C_1(\varepsilon) \le \delta(\varepsilon).
\end{align*}
Here we used the definition of $C_1(\varepsilon)$ just below \cref{ uComponentBanachSpaceC2}, and the last inequality follows from the definition of $\delta(\varepsilon)$ in \cref{DeltaDefinition}. The other estimates are similar.
\end{proof}

\subsection{Estimates for products}

In \cref{C1ContinuitySection} we gave some preliminary results for \cref{C1Continuity} in \cref{DifferenceC0Norm,DifferenceComponentThree}. We want to derive similar results for derivatives instead of functions in \cref{DifferenceC2Norm,DifferenceC2ComponentThree} respectively. The results below will be framed in a slightly more general setting, so that we can use them in the next section as well.

\begin{lemma}\label{DifferenceC2Norm} Let $X$, $Y$ and $Z$ be Banach spaces, $m \in \mathbb{N}$ and $h \in C^1_b(X,Y)$.
\begin{enumerate}[label = {\roman*)}, ref={\cref{DifferenceC2Norm}\roman*)}]
\item\label{DifferenceC2Norm1} For $f_1,g_1 \in C^0_b(Y, \mathcal{L}(Y,Z))$, $f_2,g_2 \in C^0_b(X, \mathcal{L}^m(X,Y))$ we have the $C^0$-estimate
\begin{align*}
\supnorm{(f_1 \circ h) f_2 - ( g_1 \circ h) g_2 } &\le \supnorm{f_1} \supnorm{f_2 - g_2 } + \supnorm{f_1 - g_1}\supnorm{g_2}. 
\end{align*}
\item\label{DifferenceC2Norm2} For $f_1,g_1 \in C^1_b(Y, \mathcal{L}(Y,Z))$, $f_2,g_2 \in C^1_b(X, \mathcal{L}^m(X,Y))$ we have the $C^1$-estimate
\begin{align*}
\supnorm{D [ (f_1 \circ h) f_2] &- D[ ( g_1 \circ h) g_2 ] }  \nonumber \\
			&\qquad \le \supnorm{Df_1} \supnorm{Dh} \supnorm{f_2 - g_2} + \supnorm{Df_1 - Dg_1} \supnorm{Dh} \supnorm{g_2}  \nonumber \\
			&\qquad \quad + \supnorm{f_1} \supnorm{Df_2 - Dg_2} + \supnorm{f_1 - g_1} \supnorm{Dg_2}.
\end{align*}
\end{enumerate}
\end{lemma}
\begin{proof}  \textit{i)} The $C^0$-estimate follows  from the triangle inequality and submultiplicativity of the norm.
\begin{align*}
\supnorm{(f_1 \circ h) f_2 - (g_1 \circ h)g_2} &\le \supnorm{(f_1 \circ h)(f_2 - g_2)} + \supnorm{(f_1 \circ h - g_1 \circ h)g_2} \\
					&\le\supnorm{f_1} \supnorm{f_2 - g_2 } + \supnorm{g_2}\supnorm{f_2 - g_2}.
\end{align*}

\textit{ii)} For the $C^1$-estimate we use the product rule and triangle inequality to find
\begin{align*}
\supnorm{D [ (f_1 \circ h) f_2] - D[ ( g_1 \circ h) g_2 ] }  & = \supnorm{DF_1f_2 + F_1Df_2- DG_1g_2 - G_1Dg_2} \\
			& \le \supnorm{DF_1f_2 - DG_1g_2 } + \supnorm{ F_1Df_2 -  G_1Dg_2},
\end{align*}
where we introduce $F_1 = f_1 \circ h$ and $G_1 = g_1 \circ h$. We then estimate
\begin{align*}
\supnorm{DF_1f_2 - DG_1g_2 } &\le \supnorm{DF_1f_2 - DF_1g_2 } + \supnorm{DF_1g_2 - DG_1g_2 }\\
			&\le \supnorm{Df_1}\supnorm{Dh} \supnorm{f_2 - g_2} + \supnorm{Df_1 - Dg_1}\supnorm{Dh}\supnorm{g_2}, \\
\supnorm{ F_1Df_2 -  G_1Dg_2} &\le \supnorm{F_1Df_2 - F_1Dg_2} + \supnorm{F_1 Dg_2 - G_2 Dg_2} \\
			&\le \supnorm{f_1} \supnorm{Df_1 - Dg_2} + \supnorm{f_1 - g_1} \supnorm{Dg_2}.
\end{align*}
For those estimates we have used that $DF_1 = Df_1(h) Dh$, and thus $DF_1$ is bounded by $\supnorm{Df_1}\supnorm{Dh}$ and likewise we have bounded $DF_1 - DG_1$ by $\supnorm{Df_1 - Dg_2} \supnorm{Dh}$. For the last estimate, we have used that $F_1$ is bounded by $\supnorm{f_1}$ and $F_1 - G_1$ is bounded by $\supnorm{f_1 - g_1}$. We obtain the desired estimate by adding the two estimates together
\end{proof}

\begin{lemma}\label{DifferenceC2ComponentThree} Let $\rho,\tilde{\rho} \in C_b^1(X_c,\mathcal{L}(X_c,X_c))$ be such that $\supnorm{\rho},\supnorm{\tilde{\rho}} \le L_r$ and $\supnorm{D\rho} , \supnorm{D\tilde{\rho}} \le \delta(\varepsilon)$ for some $\varepsilon > 0$. Furthermore, let $X$ and $Y$ be Banach spaces. Let $h \in C^2_b(Y,\mathcal{L}(Y,X))$ and $f_1,f_2 \in C^2_b(X_c,\mathcal{L}(X_c,Y))$. Furthermore, assume that $L_r < \operatornorm{A_c^{-1}}^{-1}$ and recall the definition of $Q_\rho$ in \cref{FunctionsFixedPointOperatorC2}.
\begin{enumerate}[label = {\roman*)}, ref={\cref{DifferenceC2ComponentThree}\roman*)}]
\item\label{DifferenceC2ComponentThree1} We have the $C^0$-estimate
\begin{align*}
\supnorm{h (f_1 \circ T)Q_\rho  - h (f_2 \circ T)Q_{\tilde{\rho}}} &\le \supnorm{h}\supnorm{f_1} L_{-1}^2 \supnorm{\rho - \tilde{\rho}} + L_{-1} \supnorm{h}\supnorm{f_1 - f_2}.
\end{align*}
\item\label{DifferenceC2ComponentThree2} We have the $C^1$-estimate
\begin{align*}
\supnorm{D [ h(f_1 \circ T)&Q_\rho] - D[ h(f_2 \circ T)Q_{\tilde{\rho}}] } \nonumber \\
				&\quad \le \supnorm{Dh}\supnorm{f_1} L_{-1}^2 \supnorm{\rho - \tilde{\rho}} + L_{-1} \supnorm{Dh}\supnorm{f_1 - f_2} \nonumber \\
				&\quad \quad + \supnorm{h}\supnorm{Df_1} L_{-1}^3 \supnorm{\rho - \tilde{\rho}} + L_{-1}^2 \supnorm{h}\supnorm{Df_1 - Df_2} \nonumber \\
				&\quad \quad + 2 \supnorm{h}\supnorm{f_1}  L_{-1}^4 \delta(\varepsilon)  \supnorm{\rho - \tilde{\rho}}   + \supnorm{h}\supnorm{f_1} L_{-1}^3 \supnorm{D\rho -D\tilde{\rho}}  \nonumber \\
				&\quad \quad + L_{-1}^3 \delta(\varepsilon) \supnorm{h}\supnorm{f_1 - f_2}.
\end{align*}
\end{enumerate}
\end{lemma}
\begin{proof}
\textit{i)} For the $C^0$-estimate, we note that $h(y)$ is a linear operator for all $y \in Y$ and we use submultiplicativity of the norm and \cref{DifferenceC2Norm1}
\begin{align*}
\supnorm{h (f_1 \circ T)Q_\rho  - h (f_2 \circ T)Q_{\tilde{\rho}}} &\le \supnorm{h} \left( \supnorm{f_1} \supnorm{Q_\rho - Q_{\tilde{\rho}}} + \supnorm{f_1 - f_2} \supnorm{Q_{\tilde{\rho}}} \right). 
\end{align*}
All that is left to do is to show that $Q_{\rho} - Q_{\tilde{\rho}}$ is bounded by $L_{-1}^2 \supnorm{\rho - \tilde{\rho}}$ and that $Q_{\tilde{\rho}}$ is bounded $L_{-1}$. For the latter bound, we use similar calculations as performed at the end of the proof of \cref{GlobalDiffeomorphism1}. Namely, fix $x \in X_c$, denote $y = T(x)$ and $\tau =Q_{\tilde{\rho}}(x) -  A_c^{-1}$, then we have $(A_c^{-1} + \tau)(A_c + \tilde{\rho}(y)) = Q_{\tilde{\rho}}(x) P_{\tilde{\rho}}(y) = \operatorname{Id}$. We can rewrite this as $\tau = - A_c^{-1} \tilde{\rho}(y)A_c^{-1} - \tau \tilde{\rho}(y)A_c^{-1}$. This implies that the norm of $\tau$ is bounded by $\operatornorm{A_c^{-1}}^2\operatornorm{\tilde{\rho}(y)} /\left(1 - \operatornorm{A_c^{-1}} \operatornorm{\tilde{\rho(y)}}\right) \le L_t$, as $\operatornorm{\tilde{\rho}(y)} \le L_r$, where $L_r$ and $L_t$ are defined in \cref{FirstDerivativeUnknownsR,FirstDerivativeUnknownsT} respectively. Therefore, we have the desired bound
\begin{align*}
\supnorm{Q_{\tilde{\rho}}} \le \sup_{x \in X_c} \operatornorm{Q_{\tilde{\rho}}(T(x))} = \sup_{x \in X_c} \operatornorm{A_c^{-1} + \tau(x)} \le \sup_{x \in X_c}\operatornorm{A_c^{-1}} + L_t = L_{-1}.
\end{align*}
The bound on $Q_\rho - Q_{\tilde{\rho}}$ now follows from submultiplicativity and 
\begin{align*}
Q_\rho - Q_{\tilde{\rho}} = Q_\rho\left( P_{\tilde{\rho}} \circ T - P_\rho \circ T \right)Q_{\tilde{\rho}} = Q_\rho\left( \tilde{\rho} \circ T - \rho \circ T \right)Q_{\tilde{\rho}}.
\end{align*}

\textit{ii)} For the $C^1$-estimate, we start by with the product rule and triangle inequality to find
\begin{align*}
\supnorm{D [ h(f_1 \circ T)&Q_\rho] - D[ h(f_2 \circ T)Q_{\tilde{\rho}}] } \nonumber \\
				&\quad \le \supnorm{Dh\left( \operatorname{Id}, (f_1 \circ T)(Q_\rho -Q_{\tilde{\rho}}) \right) } + \supnorm{Dh \left( \operatorname{Id},(f_1 \circ T - f_2 \circ T)Q_{\tilde{\rho}}\right)} \\
				& \quad \quad +\supnorm{h}\supnorm{D[(f_1 \circ T)Q_\rho] - D[f_2 \circ T)Q_{\tilde{\rho}}]}.
\end{align*}
The first two terms of the right hand side are estimated using similar arguments as those used in part i), that is
\begin{align*}
\supnorm{Dh\left( \operatorname{Id}, (f_1 \circ T)(Q_\rho -Q_{\tilde{\rho}})\right)} &\le \supnorm{Dh}\supnorm{f_1} L_{-1}^2 \supnorm{\rho - \tilde{\rho}}, \\
\supnorm{Dh \left( \operatorname{Id}, (f_1 \circ T - f_2 \circ T)Q_{\tilde{\rho}} \right(} &\le   L_{-1} \supnorm{Dh}\supnorm{f_1 - f_2},
\end{align*}
which are precisely the first two terms of the right hand side of our desired $C^1$-estimate. The last term can be estimated using \cref{DifferenceC2Norm2}:
\begin{align}
\supnorm{D[(f_1 \circ T)Q_\rho] -& D[f_2 \circ T)Q_{\tilde{\rho}}]} \nonumber \\
		& \le \supnorm{Df_1} \supnorm{DT} \supnorm{Q_\rho - Q_{\tilde{\rho}}} + \supnorm{Df_1 - Df_2}\supnorm{DT}\supnorm{Q_{\tilde{\rho}}} \nonumber \\
& \quad + \supnorm{f_1} \supnorm{DQ_{\rho} - DQ_{\tilde{\rho}}} + \supnorm{f_1 - f_2} \supnorm{DQ_{\tilde{\rho}}}. \label{DifferenceC2ComponentThreeEq2}
\end{align}
We will estimate the four terms separately. With the estimates of $Q_{\tilde{\rho}}$ and $Q_\rho - Q_{\tilde{\rho}}$ from the proof of part i), and given that $\supnorm{DT} \le L_{-1}$, we find 
\begin{align*}
\supnorm{Df_1} \supnorm{DT} \supnorm{Q_\rho - Q_{\tilde{\rho}}}  &\le  \supnorm{Df_1} L_{-1}^3 \supnorm{\rho - \tilde{\rho}}, \\
\supnorm{Df_1 - Df_2}\supnorm{DT}\supnorm{Q_{\tilde{\rho}}} &\le  L_{-1}^2 \supnorm{Df_1 - Df_2},
\end{align*}
which are, up to the factor $\supnorm{h}$, the third and fourth term of the right hand side of our desired $C^1$-estimate.
Finally, we have to find an upper bound for $DQ_{\tilde{\rho}}$ and $DQ_\rho - DQ_{\tilde{\rho}}$ in terms of $L_{-1}$ and $\supnorm{\rho - \tilde{\rho}}$ to estimate the final two terms in \cref{DifferenceC2ComponentThreeEq2}. The product rule gives us, since $Q_{\tilde{\rho}} = \left( P_{\tilde{\rho}} \circ T \right)^{-1}$,
\begin{align*}
0 = D[Q_{\tilde{\rho}} (P_{\tilde{\rho}} \circ T)] =  D Q_{\tilde{\rho}} (P_{\tilde{\rho}} \circ T) + Q_{\tilde{\rho}} DP_{\tilde{\rho}}(T)DT.
\end{align*}
We isolate $D Q_{\tilde{\rho}} (P_{\tilde{\rho}} \circ T)$ and multiply from the right with $(P_{\tilde{\rho}} \circ T)^{-1} = Q_{\tilde{\rho}}$:
\begin{align}
DQ_{\tilde{\rho}} = - Q_{\tilde{\rho}}  DP_{\tilde{\rho}}(T)\left(  DT,Q_{\tilde{\rho}} \right). \label{DifferenceC2ComponentThreeEq1}
\end{align}
Furthermore, we note that $P_{\tilde{\rho}} = A_c + \tilde{\rho}$, hence $DP_{\tilde{\rho}} = D \tilde{\rho}$, which is bounded by $\delta(\varepsilon)$. We also saw that $\supnorm{Q_{\tilde{\rho}}} \le L_{-1}$ in the proof of part i). Hence we find with the triangle inequality
\begin{align*}
\supnorm{DQ_{\rho} - DQ_{\tilde{\rho}}} &\le \supnorm{  \left( Q_{\rho} - Q_{\tilde{\rho}} \right) DP_{\rho}(T)\left(  DT,Q_{\rho} \right) } \\
						&\quad + \supnorm{ Q_{\tilde{\rho}}  \left(DP_{\rho}(T) - DP_{\tilde{\rho}}(T) \right) \left(  DT,Q_{\rho} \right)} \\
						&\quad + \supnorm{Q_{\tilde{\rho}}  DP_{\tilde{\rho}}(T) \left(  DT,Q_{\rho} - Q_{\tilde{\rho}} \right)} \\
						&\le 2 L_{-1}^4 \delta(\varepsilon) \supnorm{\rho - \tilde{\rho}} + L_{-1}^3 \supnorm{D\rho - D \tilde{\rho}}.
\end{align*}
For the last inequality we used that $P_{\tilde{\rho}} = A_c + \tilde{\rho}$, and thus $DP_{\tilde{\rho}} = D \tilde{\rho}$, which is bounded by $\delta(\varepsilon)$. Furthermore, we used that $Q_{\tilde{\rho}}$ is bounded by $L_{-1}$ and $Q_{\rho} - Q_{\tilde{\rho}}$ is bounded by $L_{-1}^2 \supnorm{\rho - \tilde{\rho}}$, as shown in the proof of part i). Hence the third factor of \cref{DifferenceC2ComponentThreeEq2} is bounded by the fifth and sixth term appearing in the right hand side of our desired $C^1$-estimate. Finally, we estimate the last term of \cref{DifferenceC2ComponentThreeEq2}, where we use \cref{DifferenceC2ComponentThreeEq1} to bound $DQ_{\tilde{\rho}}$:
\begin{align*}
\supnorm{f_1 - f_2} \supnorm{DQ_{\tilde{\rho}}} \le \supnorm{Q_{\tilde{\rho}}}^2 \supnorm{DT} \supnorm{D \tilde{\rho}} \supnorm{f_1 - f_2} \le L_{-1}^3 \delta(\varepsilon) \supnorm{f_1 -f_2}.
\end{align*}
This is precisely the final term appearing in the asserted estimate.
\end{proof}

\subsection{A new contraction}

With the previous two lemmas, we will show that $\Theta^{[2]}$ is a contraction on $\Gamma_2(\delta(\varepsilon))$ for $\varepsilon >0$ small enough. We note again that we will later show that we can always scale our functions to satisfy the bound on the second derivative.

\begin{proposition}\label{C2Continuity} Assume that $L_g$ and $L_c$ are small in the sense of \cref{SmallnessConditions} for $n=2$.  There exists an $\varepsilon_0 >0 $ such for all $\varepsilon < \varepsilon_0$ it holds that  if $\supnorm{D^2 g}, \supnorm{D^2k_c} \le  \varepsilon$, then $\Theta^{[2]} : \Gamma_2(\delta(\varepsilon)) \to \Gamma_2(\delta(\varepsilon))$ is a contraction with respect to the $C^1$ norm.
\end{proposition}

\begin{proof}
Let $\varepsilon > 0$ and $\supnorm{D^2 g}, \supnorm{D^2k_c} \le \varepsilon$. Let $\mathcal{M} = \tmatrix{\rho}{\kappa_u}{\kappa_s},  \tilde{\mathcal{M}}= \tmatrix{\tilde{\rho}}{\tilde{\kappa}_u}{ \tilde{\kappa}_s} \in \Gamma_{2}(\delta(\varepsilon))$. To show that $\Theta^{[2]}$ is a $C^1$ contraction, we will use the same steps as we used in the proof of  \cref{C1Continuity}.

\begin{enumerate}[label=\Alph*)]
\item\label{C2Continuity1} We prove that $\Theta^{[2]}$ is a contraction with respect to the $C^0$ norm, independent of $\varepsilon$,

\item\label{C2Continuity2} We show the existence of a constant $\theta_2(\varepsilon)$ such that 
\begin{align*}
\supnorm{D[\Theta^{[2]}(\mathcal{M})] - D[\Theta^{[2]}(\tilde{\mathcal{M}})]} \le \theta_2(\varepsilon) \onenorm{\mathcal{M} - \tilde{\mathcal{M}}},
\end{align*}

\item\label{C2Continuity3} We show that $\varepsilon > 0$ can be chosen so that $\theta_2(\varepsilon) <1$,  thus proving that $\Theta^{[2]}$ is a contraction with respect to the $C^1$ norm. 
\end{enumerate}

\textbf{Step \cref{C2Continuity1}} We recall that $\theta_1(\varepsilon)$, defined in \cref{DefinitionTheta1}, has the property $\theta_1(0) <1$. We want to show that 
\begin{align*}
\supnorm{\Theta^{[2]}(\mathcal{M}) - \Theta^{[2]}(\tilde{\mathcal{M}})} \le \theta_1(0) \supnorm{\mathcal{M} - \tilde{\mathcal{M}}}.
\end{align*}
Recall from equation \cref{FixedPointOperatorC2Equation} that
\begin{align*}
\Theta^{[2]}(\mathcal{M}) = \Theta^{[2]} \begin{pmatrix} \rho \\ \kappa_u \\ \kappa_s \end{pmatrix} =
\begin{pmatrix*}[l]
A_c  Dk_c + Dg_c(K{}{})  \kappa - Dk_c\left(R\right)  P_\rho \\
-A_u^{-1}  Dg_u(K{}{})  \kappa + A_u^{-1}  \kappa_u(R)  P_\rho \\
A_s  \kappa_s(T)  Q_\rho  + Dg_s(K{}{} \circ T)  \kappa(T)  Q_\rho
\end{pmatrix*},
\end{align*}
which was derived by taking the derivative of $\Theta^{}\tmatrix{r}{k_u}{k_s}$. Therefore, we will use similar arguments as in step \cref{C1Continuity2} of the proof of \cref{C1Continuity} to show that
\begin{align*}
\supnorm{ \Theta^{[2]}_i(\mathcal{M}) - \Theta^{[2]}_i(\tilde{\mathcal{M}})} &\le \theta_{1,i} \supnorm{\mathcal{M} - \tilde{\mathcal{M}}}
\end{align*}
for $\theta_{1,i}$ given explicitly in equation \cref{ContractionConstants1,ContractionConstants2,ContractionConstants3} for $i=1,2,3$.

\underline{\textit{$\rho$-component:}} We have
\begin{align}
\supnorm{\Theta^{[2]}_1(\mathcal{M})& - \Theta^{[2]}_1(\tilde{\mathcal{M}})} \nonumber \\
	 &\le \supnorm{ (Dg_c \circ K{}{}) \kappa  -  (Dg_c \circ K{}{}) \tilde{\kappa}} + \supnorm{ (Dk_c \circ R)P_{\rho} - (Dk_c \circ R) P_{\tilde{\rho}}}. \label{rComponentEqC20}
\end{align}
The first term is estimated by \cref{DifferenceC2Norm1}:
\begin{align}
\supnorm{ (Dg_c \circ K{}{}) \kappa  -  (Dg_c \circ K{}{}) \tilde{\kappa}}  &\le \supnorm{Dg_c} \supnorm{\kappa - \tilde{\kappa}} \le L_g \supnorm{\mathcal{M} - \tilde{\mathcal{M}}}. \label{rComponentEqC201}
\end{align}
Here we recall that $\supnorm{Dg_c} \le L_g$, which follows from assumption \cref{MainTheoremKnownBounds } of \cref{MainTheorem}. Likewise, we estimate
\begin{align}
\supnorm{ (Dk_c \circ R)P_{\rho} - (Dk_c \circ R) P_{\tilde{\rho}}} &\le \supnorm{Dk_c} \supnorm{ \rho - \tilde{\rho}} \le L_c \supnorm{\mathcal{M}  - \tilde{\mathcal{M}}}. \label{rComponentEqC202}
\end{align}
Thus inequality \cref{rComponentEqC20} together with estimates  \cref{rComponentEqC201,rComponentEqC202} gives
\begin{align}
\supnorm{\Theta^{[2]}_1(\mathcal{M}) - \Theta^{[2]}_1(\tilde{\mathcal{M}})} &\le \left( L_g + L_c  \right) \supnorm{\mathcal{M} - \tilde{\mathcal{M}}} = \theta_{1,1} \supnorm{\mathcal{M} - \tilde{\mathcal{M}}}. \label{ C2ContractionEquation1}
\end{align}

\underline{\textit{$\kappa_u$-component:}}
Similarly, we have
\begin{align}
\supnorm{\Theta^{[2]}_2(\mathcal{M}) - \Theta^{[2]}_2(\tilde{\mathcal{M}})}  &\le \operatornorm{ A_u^{-1}} \supnorm{ (Dg_u \circ K{}{}) \kappa - (Dg_u \circ K{}{}) \tilde{\kappa} } \nonumber \\
							&\quad + \operatornorm{ A_u^{-1}} \supnorm{(\kappa_u \circ R) P_{\rho} - ( \tilde{\kappa}_u\circ R) P_{\tilde{\rho}}} . \label{uComponentEqC20}
\end{align}
We again use \cref{DifferenceC2Norm1}, which gives
\begin{align*}
\supnorm{ (Dg_u \circ K{}{}) \kappa - (Dg_u \circ K{}{}) \tilde{\kappa} } &\le \supnorm{Dg_u} \supnorm{\kappa - \tilde{\kappa}} \le  L_g \supnorm{\mathcal{M} - \tilde{\mathcal{M}}}, \\
\supnorm{(\kappa_u \circ R) P_{\rho} - ( \tilde{\kappa}_u\circ R) P_{\tilde{\rho}}} &\le \supnorm{\kappa_u} \supnorm{P_\rho - P_{\tilde{\rho}}} + \supnorm{P_{\tilde{\rho}}} \supnorm{\kappa_u - \tilde{\kappa}_u} \\
				&\le \left( L_u + \operatornorm{A_c} + L_r \right) \supnorm{\mathcal{M} - \tilde{\mathcal{M}}}. 
\end{align*}
Here we used that $\kappa_u$ is bounded by $L_u$ and $\tilde{\rho}$ by $L_r$. Thus inequality \cref{uComponentEqC20} becomes
\begin{align}
\supnorm{\Theta^{[2]}_2(\mathcal{M}) - \Theta^{[2]}_2(\tilde{\mathcal{M}})} &\le  \theta_{1,2} \supnorm{\mathcal{M} - \tilde{\mathcal{M}}}. \label{ C2ContractionEquation2}
\end{align}

\underline{\textit{$\kappa_s$-component:}} Let us denote $G_s =  Dg_s \circ K{}{} \circ T$. Then
\begin{align}
\supnorm{\Theta^{[2]}_3(\mathcal{M}) - \Theta^{[2]}_3(\tilde{\mathcal{M}})}   &\le  \supnorm{A_s (\kappa_s\circ T)Q_{\rho} - A_s (\tilde{\kappa}_s \circ T) Q_{\tilde{\rho}}} \nonumber \\
			&\quad   + \supnorm{G_s( \kappa \circ T) Q_\rho - G_s (\tilde{\kappa} \circ T) Q_{\tilde{\rho}}}. \label{sComponentEqC20}
\end{align}
We will estimate both terms with \cref{DifferenceC2ComponentThree1}. We note that $\kappa$ is bounded by $1+L_c$, and hence we obtain
\begin{align*}
\supnorm{A_s (\kappa_s\circ T)Q_{\rho} - A_s (\tilde{\kappa}_s \circ T) Q_{\tilde{\rho}}}  &\le \left( \operatornorm{A_s} L_s L_{-1}^2 + L_{-1} \operatornorm{A_s} \right) \supnorm{\mathcal{M} - \tilde{\mathcal{M}}}, \\
\supnorm{G_s( \kappa \circ T) Q_\rho - G_s (\tilde{\kappa} \circ T) Q_{\tilde{\rho}}} &\le  \left( L_g(1 + L_c) L_{-1}^2 + L_{-1}L_g \right) \supnorm{\mathcal{M} - \tilde{\mathcal{M}}}. 
\end{align*}
Thus inequality \cref{sComponentEqC20} becomes
\begin{align}
\supnorm{\Theta^{[2]}_3(\mathcal{M}) - \Theta^{[2]}_3(\tilde{\mathcal{M}})} 
									 &\le \theta_{1,3}  \supnorm{\mathcal{M} - \tilde{\mathcal{M}}}. \label{ C2ContractionEquation3}
\end{align}

\underline{\textit{Contraction constant:}} We can now estimate $\supnorm{\Theta^{[2]}(\mathcal{M}) - \Theta^{[2]}(\tilde{\mathcal{M}})}$ with inequalities \cref{ C2ContractionEquation1, C2ContractionEquation2, C2ContractionEquation3}. We have
\begin{align}
\supnorm{\Theta^{[2]}(\mathcal{M}) - \Theta^{[2]}(\tilde{\mathcal{M}})} &= \max_{i=1,2,3} \left\{ \supnorm{\Theta^{[2]}_i(\mathcal{M}) - \Theta^{[2]}_i(\tilde{\mathcal{M}})} \right\} \nonumber \\
										&\le \max_{i=1,2,3} \left\{ \theta_{1,i} \supnorm{\mathcal{M} - \tilde{\mathcal{M}}} \right\} \nonumber \\
										&= \theta_1(0) \supnorm{\mathcal{M} - \tilde{\mathcal{M}}}, \label{C2ContractionEq1}
\end{align}
where the last equality follows from the definition of $\theta_1$ in \cref{DefinitionTheta1}. Since $\theta_1(0)<1$, we conclude that $\Theta^{[2]}$ is a contraction with respect to the $C^0$ norm.  \vspace{1 \baselineskip}

\textbf{Step \cref{C2Continuity2}} Analogous to step \cref{C1Continuity2} of the proof of \cref{C1Continuity}, we want to prove the component-wise inequality
\begin{align*}
\supnorm{D[\Theta^{[2]}_i(\mathcal{M})] - D[\Theta^{[2]}_i(\tilde{\mathcal{M}})]} &\le \left(\theta_{2,i} + C_{2,i}(\varepsilon) \right) \onenorm{\mathcal{M} - \tilde{\mathcal{M}}},
\end{align*} 
with $\theta_{2,i}$ defined in \cref{ContractionConstants1,ContractionConstants2,ContractionConstants3} and $C_{2,i}$ will be defined during the proof. We note that $\mathcal{M},\tilde{\mathcal{M}} \in \Gamma_2(\delta(\varepsilon))$, hence $D\rho$, $D\kappa_u$ and $D\kappa_s$ are bounded by $\delta(\varepsilon)$.

\underline{\textit{$\rho$-component:}} We start with
\begin{align}
\supnorm{D[\Theta^{[2]}_1(\mathcal{M})] - D[\Theta^{[2]}_1(\tilde{\mathcal{M}})]} &\le \supnorm{D[(Dg_c \circ K{}{})\kappa] - D[(Dg_c \circ K{}{}) \tilde{\kappa}]} \nonumber \\
		&\quad + \supnorm{D[(Dk_c \circ R) P_{\rho}] - D[(Dk_c \circ R)P_{\tilde{\rho}}]}.  \label{rComponentEqC21}
\end{align}
By applying \cref{DifferenceC2Norm2} we find that
\begin{align*}
\supnorm{D[(Dg_c \circ K{}{})\kappa] &- D[(Dg_c \circ K{}{}) \tilde{\kappa}]} \\
						&\le  \supnorm{D^2 g_c} \supnorm{D K{}{}}  \supnorm{\kappa - \tilde{\kappa}} + \supnorm{Dg_c} \supnorm{D\kappa - D \tilde{\kappa}} \\
						&\le \left(  (1+ L_c) \varepsilon + L_g \right) \onenorm{\mathcal{M} - \tilde{\mathcal{M}}},
\end{align*}
where we recall that $DK{}{}$ is bounded by $1+L_c$, see \cref{EstimateDk}. Likewise, we find the estimate
\begin{align*}
\supnorm{D[(Dk_c \circ R) P_{\rho}] &- D[(Dk_c \circ R)P_{\tilde{\rho}}]}  \\
								&\le \supnorm{D^2k_c} \supnorm{DR} \supnorm{\rho- \tilde{\rho}}  + \supnorm{Dk_c} \supnorm{D\rho- D\tilde{\rho}}   \\ 
								&\le  \left(\left( \operatornorm{A_c} + L_r \right) \varepsilon + 1 + L_c  \right) \onenorm{\mathcal{M} - \tilde{\mathcal{M}}}. 
\end{align*}
Thus inequality \cref{rComponentEqC21} together with the above estimates gives
\begin{align}
\supnorm{D[\Theta^{[2]}_1(\mathcal{M})] - D[\Theta^{[2]}_1(\tilde{\mathcal{M}})]}	&\le \left( L_g + L_c + \left( 1 + L_c  +  \operatornorm{A_c}+ L_r \right) \varepsilon \right) \onenorm{\mathcal{M} - \tilde{\mathcal{M}}} \nonumber \\
							&= \left( \theta_{2,1} + C_{2,1}(\varepsilon)\right) \onenorm{\mathcal{M} - \tilde{\mathcal{M}}}, \label{ C2ContractionEquation4}
\end{align}
where we define $C_{2,1}(\varepsilon) \isdef \left( 1 + L_c + \operatornorm{A_c} +  L_r \right)\varepsilon$.

\underline{\textit{$\kappa_u$-component:}} We have
\begin{align}
\supnorm{D[\Theta^{[2]}_2(\mathcal{M})] - D[\Theta^{[2]}_2(\tilde{\mathcal{M}})]} &\le  \operatornorm{A_u^{-1}}  \supnorm{ D[ (Dg_u \circ K{}{}) \kappa] - D[(Dg_u \circ K{}{}) \tilde{\kappa}]} \nonumber \\
									&\quad + \operatornorm{A_u^{-1}} \supnorm{D[( \kappa_u \circ R) P_{\rho}] - D[( \tilde{\kappa}_u \circ R) P_{\tilde{\rho}}]}. \label{uComponentEqC21}
\end{align}
By \cref{DifferenceC2Norm2} we get
\begin{align*}
\supnorm{ D[ (Dg_u \circ K{}{}) \kappa] - D[(Dg_u \circ K{}{}) \tilde{\kappa}]} &\le  \left( (1 + L_c) \varepsilon + L_g \right) \onenorm{\mathcal{M} -\tilde{\mathcal{M}}}, \\
\supnorm{D[( \kappa_u \circ R) P_{\rho}] - D[( \tilde{\kappa}_u \circ R) P_{\tilde{\rho}}]}  
&\le \left(  (\operatornorm{A_c} + L_r ) \delta( \varepsilon) + L_u \right) \onenorm{\rho - \tilde{\rho}} \\
								&\quad + \left(  (\operatornorm{A_c} + L_r )^2  +  \delta( \varepsilon) \right) \onenorm{\kappa_u - \tilde{\kappa}_u}.
\end{align*}
Thus inequality \cref{uComponentEqC21} becomes, as $\rho - \tilde{\rho}$ and $\kappa - \tilde{\kappa}$ are bounded by $\onenorm{\mathcal{M} - \tilde{\mathcal{M}}}$,
\begin{align}
\supnorm{D[\Theta^{[2]}_2(\mathcal{M})] - D[\Theta^{[2]}_2(\tilde{\mathcal{M}})]} &\le \operatornorm{A_u^{-1}} \Big(  (\operatornorm{A_c} + L_r)^2 + L_g + L_u  \nonumber \\
							&\quad + (2 + L_c) \varepsilon + (\operatornorm{A_c} + L_r ) \delta \Big) \onenorm{\mathcal{M} - \tilde{\mathcal{M}}} \nonumber \\
							&\le \left( \theta_{2,2} + C_{2,2}(\varepsilon)\right) \onenorm{\mathcal{M} - \tilde{\mathcal{M}}}, \label{ C2ContractionEquation5}
\end{align}
where we define $C_{2,2}(\varepsilon) \isdef \operatornorm{A_u^{-1}}( (2 + L_c) \varepsilon + (\operatornorm{A_c} + L_r ) \delta(\varepsilon))$.

\underline{\textit{$\kappa_s$-component:}} Let us again denote $G_s = Dg_s \circ K{}{} \circ T$, then we have
\begin{align}
\supnorm{D[\Theta^{[2]}_3(\mathcal{M})] - D[\Theta^{[2]}_3(\tilde{\mathcal{M}})]}	&\le \supnorm{D[A_s(\kappa_s \circ T)Q_\rho] -D[A_s (\tilde{\kappa}_s \circ T)Q_{\tilde{\rho}}]} \nonumber \\
									&\quad + \supnorm{D[G_s ( \kappa \circ T)Q_\rho] - D[G_s(\tilde{\kappa} \circ T)Q_{\tilde{\rho}}]}. \label{sComponentEqC21}
\end{align}
We will estimate both terms with \cref{DifferenceC2ComponentThree2}. For the first term, we note that $A_s$ is the constant operator $x \mapsto A_s$, hence $\supnorm{A_s} = \operatornorm{A_s}$ and $DA_s = 0$, which gives us
\begin{align}
\supnorm{D[A_s(\kappa_s \circ T)Q_\rho] &-D[A_s (\tilde{\kappa}_s \circ T)Q_{\tilde{\rho}}]} \nonumber \\
				&\le \operatornorm{A_s} L_{-1}^3  \supnorm{D\kappa_s} \supnorm{\rho - \tilde{\rho}} + \operatornorm{A_s} L_{-1}^2 \supnorm{D\kappa_s - D\tilde{\kappa}_s}  \nonumber \\
				&\quad + 2 \operatornorm{A_s} L_{-1}^4  \supnorm{\kappa_s} \delta(\varepsilon) \supnorm{\rho - \tilde{\rho}} + \operatornorm{A_s} L_{-1}^3 \supnorm{\kappa_s} \supnorm{D\rho - D \tilde{\rho}} \nonumber \\
				&\quad +\operatornorm{A_s}  L_{-1}^3  \delta(\varepsilon) \supnorm{\kappa_s - \tilde{\kappa}_s} \nonumber \\
				&\le \operatornorm{A_s} L_{-1}^2 \left( 1 + L_{-1} L_s \right) \onenorm{\mathcal{M} - \tilde{\mathcal{M}}} \label{sComponentEqC22}\\
				&\quad + 2 \operatornorm{A_s} L_{-1}^3 \left( 1 + L_{-1} L_s \right) \delta(\varepsilon) \onenorm{\mathcal{M} - \tilde{\mathcal{M}}}. \label{sComponentEqC23}
\end{align}
The second term in the right hand side of \cref{sComponentEqC21} involves $\kappa$, which is bounded by $1 + L_c$, its derivative, which is bounded by
\begin{align*}
\supnorm{D\kappa} = \max \{ D^2 k_c , D \kappa_u , D\kappa_s \} \le \max\{ \varepsilon, \delta(\varepsilon) \} \isfed \gamma(\varepsilon),
\end{align*} and the derivative of $G_s$. The derivative of $G_s$ is bounded by $\supnorm{D^2g_s}\supnorm{DK{}{}}\supnorm{DT}$, which in turn is bounded by  $L_{-1}(1 + L_c ) \varepsilon$. We obtain
\begin{align}
\supnorm{D[G_s ( \kappa \circ T)Q_\rho] &- D[G_s(\tilde{\kappa} \circ T)Q_{\tilde{\rho}}]} \nonumber \\				
				&\le \supnorm{DG_s} L_{-1}^2 \supnorm{\kappa} \supnorm{\rho - \tilde{\rho}} + \supnorm{DG_s} L_{-1}\supnorm{\kappa - \tilde{\kappa}} \nonumber \\
				&\quad + \supnorm{G_s} L_{-1}^3  \supnorm{D\kappa} \supnorm{\rho - \tilde{\rho}} + \supnorm{G_s} L_{-1}^2 \supnorm{D\kappa - D\tilde{\kappa}}  \nonumber \\
				&\quad + 2 \supnorm{G_s} L_{-1}^4  \supnorm{\kappa} \delta(\varepsilon) \supnorm{\rho - \tilde{\rho}} + \supnorm{G_s} L_{-1}^3 \supnorm{\kappa} \supnorm{D\rho - D \tilde{\rho}} \nonumber \\
				&\quad +\supnorm{G_s}  L_{-1}^3  \delta(\varepsilon) \supnorm{\kappa - \tilde{\kappa}} \nonumber \\
				&\le L_{-1}^2( L_g + L_{-1}L_g(1+L_c)) \onenorm{\mathcal{M} - \tilde{\mathcal{M}}} \label{sComponentEqC24}  \\
				&\quad + \left(  L_{-1}^3 L_g \gamma(\varepsilon) +  L_{-1}^3( L_g + 2 L_{-1}L_g(1+L_c)) \delta (\varepsilon)  \right) \onenorm{\mathcal{M} - \tilde{\mathcal{M}}} \label{sComponentEqC25}  \\
				&\quad +  \left( L_{-1}^3(1+L_c)^2 \varepsilon + L_{-1}^2(1 + L_c) \varepsilon \right) \onenorm{\mathcal{M} - \tilde{\mathcal{M}}}. \nonumber 
\end{align}	
We see that \cref{sComponentEqC22,sComponentEqC24} together are $\theta_{2,3}\onenorm{\mathcal{M}- \tilde{\mathcal{M}}}$. Likewise, we can estimate \cref{sComponentEqC23,sComponentEqC25} together by $2 \theta_{3,3} \gamma(\varepsilon)\onenorm{\mathcal{M}- \tilde{\mathcal{M}}}$ as $\delta(\varepsilon) \le \gamma(\varepsilon)$.Then inequality \cref{sComponentEqC21} becomes
\begin{align}
\supnorm{D[\Theta^{[2]}_3(\mathcal{M})] - D[\Theta^{[2]}_3(\tilde{\mathcal{M}})]} &\le \left( \theta_{2,3} + C_{2,3}(\varepsilon) \right) \onenorm{\mathcal{M} - \tilde{\mathcal{M}}}, \label{ C2ContractionEquation6}
\end{align}
where we define $C_{2,3}(\varepsilon) \isdef 2 \theta_{3,3} \gamma(\varepsilon) + L_{-1}^2 ( 1 + L_c) \varepsilon  + L_{-1}^3(1 + L_c)^2  \varepsilon$.

\underline{\textit{Lipschitz constant:}} Inequalities \cref{ C2ContractionEquation4, C2ContractionEquation5, C2ContractionEquation6} imply
\begin{align}
\supnorm{D[\Theta^{[2]}(\mathcal{M})] - D[\Theta^{[2]}(\tilde{\mathcal{M}})]} &= \max_{i=1,2,3} \left\{ \supnorm{D[\Theta^{[2]}_i(\mathcal{M})] - D[\Theta^{[2]}_i(\tilde{\mathcal{M}})]} \right\} \nonumber \\
										&\le \max_{i=1,2,3} \left\{( \theta_{2,i} + C_{2,i}(\varepsilon)) \onenorm{\mathcal{M} - \tilde{\mathcal{M}}} \right\} \nonumber \\
										&\le \theta_2(\varepsilon) \onenorm{\mathcal{M} - \tilde{\mathcal{M}}}. \label{C2ContractionEq2}
\end{align}
Here we define $\theta_2(\varepsilon) \isdef \max\left\{ \theta_{2,1}, \theta_{2,2}, \theta_{2,3} \right\} + \max \{ C_{2,1}(\varepsilon), C_{2,2}(\varepsilon), C_{2,3}(\varepsilon)\}$. \vspace{1 \baselineskip}

\textbf{Step \cref{C2Continuity3}} From \cref{SmallnessConditions} it follows that 
\begin{align*}
\max\left\{ \theta_{2,1}, \theta_{2,2}, \theta_{2,3} \right\} < 1.
\end{align*}
As $\delta(\varepsilon) \downarrow 0$ when $\varepsilon \downarrow 0$, see \cref{UniformSecondDerivativeBound}, we have
\begin{align*}
\lim_{\varepsilon \to 0} \max\{C_{2,1}(\varepsilon), C_{2,2}(\varepsilon),C_{2,3}(\varepsilon)\} = 0.
\end{align*}
For the limit of $\varepsilon$ to $0$ of $\theta_2(\varepsilon)$ we find
\begin{align*}
\lim_{\varepsilon \to 0} \theta_2(\varepsilon) = \max\left\{ \theta_{2,1}, \theta_{2,2}, \theta_{2,3} \right\} < 1.
\end{align*}
Therefore, we can find an $\varepsilon_0 > 0 $ such that $\theta_2(\varepsilon) < 1$ for all $\varepsilon < \varepsilon_0$. Then estimates  \cref{C2ContractionEq1,C2ContractionEq2} give
\begin{align*}
\onenorm{\Theta^{[2]}(\mathcal{M}) - \Theta^{[2]}(\tilde{\mathcal{M}})} &= \max\left\{ \supnorm{\Theta^{[2]}(\mathcal{M}) - \Theta^{[2]}(\tilde{\mathcal{M}})}, \supnorm{D[\Theta^{[2]}(\mathcal{M})] - D[\Theta^{[2]}(\tilde{\mathcal{M}})]} \right\} \\
											&\le \max \left\{ \theta_1(0) \supnorm{\mathcal{M} - \tilde{\mathcal{M}}}, \theta_2(\varepsilon) \onenorm{\mathcal{M} - \tilde{\mathcal{M}}} \right\} \\
											&\le \lambda_2 \onenorm{\mathcal{M} - \tilde{\mathcal{M}}}.
\end{align*}
We define the contraction constant $\lambda_2 \isdef \max\{\theta_1(0) , \theta_2(\varepsilon) \}$, which is smaller than $1$ by the above discussion. Thus we see that $\Theta^{[2]}: \Gamma_2(\delta(\varepsilon)) \to \Gamma_2(\delta(\varepsilon))$ is a contraction with respect to the $C^1$ norm for all $\varepsilon < \varepsilon_0$.
\end{proof}

\begin{corollary}\label{C2Existence}
Let $\varepsilon >0$ such that $F{}{} :X \to X$ satisfies the conditions of \cref{MainTheorem,C1Continuity,C2Continuity}. Then the image of $K{}{}$ is a $C^2$ center manifold for $F{}{}$.
\end{corollary}
\begin{proof}
By assumption, $\varepsilon > 0$ is such that both $\Theta^{}$ and $\Theta^{[2]}$ are contractions. Let $\Lambda = \tmatrix{r}{k_u}{k_s} \in \Gamma_0$ be a fixed point of $\Theta^{}$. Then \cref{C1Existence} implies that $K{}{} = \iota + \tmatrix{k_c}{k_u}{k_s}$ parameterizes a $C^1$ center manifold for $F{}{}$. Thus if we prove that $\Lambda$ is $C^2$, or, equivalent, that $D\Lambda$ is $C^1$, we are done. 

We will show that $D\Lambda$ is in fact the fixed point of $\Theta^{[2]}$. Let
\begin{align*}
\Gamma_{3/2} = \left\{ \mathcal{M} = \begin{pmatrix} \rho \\ \kappa_u \\ \kappa_s \end{pmatrix} \in C^0_b \left(X_c, \mathcal{L}(X_c , X)  \right) \ \middle| \   \begin{matrix*}[l] 
\mathcal{M}(0) = 0, \\
\supnorm{ \rho } \le L_{r} \\ 
\supnorm{\kappa_u}\le \min\{ L_{u} \}  \\ 
\supnorm{  \kappa_s } \le \min\{ L_{s}\} 
\end{matrix*}  \right\},
\end{align*}
then $D\Lambda \in \Gamma_{3/2}$ and $\Gamma_2(\delta(\varepsilon)) \subset \Gamma_{3/2}$. Furthermore, from step \cref{C2Continuity1} of the proof of the previous proposition, it follows that $\Theta^{[2]} : \Gamma_{3/2} \to \Gamma_{3/2}$ is a contraction with respect to the $C^0$ norm. Therefore, $D\Lambda$ is the unique fixed point of $\Theta^{[2]}$ in $\Gamma_{3/2}$. However, $\Theta^{[2]} : \Gamma_{2}(\delta(\varepsilon)) \to \Gamma_{2}(\delta(\varepsilon))$ is also a contraction with respect to the $C^1$ norm for $\varepsilon$ sufficiently small. Let $\mathcal{M} \in \Gamma_2(\delta(\varepsilon))$ be the fixed point of $\Theta^{[2]}$. Then $\mathcal{M} \in \Gamma_{3/2}$ is a fixed point of $\Theta^{[2]}$, which means that $D\Lambda = \mathcal{M} \in \Gamma_2(\delta(\varepsilon))$. We conclude that $D\Lambda$ is $C^1$.
\end{proof}

We have now shown in two steps that there exists a $C^2$ center manifold. In particular, we used the existence of a $C^1$ center manifold to obtain a $C^2$ center manifold in the second step. Furthermore, to obtain the $C^1$ center manifold, we explicitly used that our dynamical system $F$ is at least $C^2$. Hence, with the current proof we cannot obtain a center manifold for a $C^1$ dynamical system. 

We  want to remark that if $F{}{}$ is $C^{0 + \text{Lip}}$, we could slightly alter the definition of $\Gamma_0$ and show that $\Theta^{}$ is a contraction with respect to the $C^0$ norm. Then both $K$ and $r$ would have been $C^{0 + \text{Lip}}$ and we would obtain a $C^{0 + \text{Lip}}$ center manifold and dynamical system.

Furthermore, if $F{}{}$ is $C^1$ instead of $C^2$, we could adapt our proof to obtain a $C^1$ center manifold and dynamical system if $X$ is uniformly convex, e.g. $X = \mathbb{R}^m$. In this case, we would also prove that $\Theta^{}$ is a contraction with respect to the $C^0$ norm, which would give us $C^{0 + \text{Lip}}$ functions $K{}{}$ and $r{}{}$. Using the results from \cite{Clarkson36}, we know that both $K{}{}$ and $r$ would be almost everywhere differentiable. We could still define the fixed point operator $\Theta^{[2]}$ and prove that $\Theta^{[2]}$ would be a contraction with respect to the $C^0$ norm. Using similar arguments as in the previous corollary, we could then show that $K{}{}$ and $r$ would be everywhere continuously differentiable. Therefore, we would obtain a $C^1$ center manifold with $C^1$ dynamical system if we start with a $C^1$ dynamical system $F:X \to X$ on a uniformly convex space $X$.

\section{A \texorpdfstring{$C^m$}{Cm} center manifold}\label{CmContinuitySection}

The final step of our proof scheme in \cref{ProofScheme} is inductively showing that the conjugacy is $C^n$. Similar to what we did in the $C^2$ case in the previous section, we will show that if the conjugacy is $C^m$, then the $m^{\text{th}}$ derivative will be $C^1$. We start again by defining a fixed point operator for the $m^{\text{th}}$ derivative of 
\begin{align*}
\Lambda \isdef \begin{pmatrix} r \\ k_u \\ k_s \end{pmatrix}.
\end{align*} 
To make the definition of the fixed point operator more insightful, we will use several lemmas before we define it. We start by stating Fa\`{a} di Bruno's formula for derivatives of compositions, see for instance \cite{Mennucci17}

\begin{lemma}[Fa\`{a} di Bruno's Formula] Let $X$, $Y$ and $Z$ be Banach spaces, and $f_1 :Y \to Z$ and $f_2 : X \to Y$ $C^m$ functions. Then the $m^{\text{th}}$ derivative of $f_1 \circ f_2$ is given by
\begin{align*}
D^m[f_1 \circ f_2](x)= \sum_{i=1}^m \sum_{\pi \in P_m^i}  D^i f_1(f_2(x)) \left(D^{\pi(1)}f_2(x), \dots  , D^{\pi(i)}f_2(x) \right),
\end{align*}
where $P_m^i$ is the set of ordered partitions of length $i$ of the set $\{1, \dots, m\}$.
\end{lemma}

Since we want to define a fixed point operator for the $m^{\text{th}}$ derivative of $\Lambda$ using $\Lambda = \Theta^{}(\Lambda)$, we want to isolate the $m^{\text{th}}$ derivatives from Fa\`{a} di Bruno's formula, and apply Fa\`{a} di Bruno's formula to the composition of three functions.

\begin{lemma}\label{FaaDiBrunoHighOrder} Let $X$, $Y$ and $Z$ be Banach spaces, and $f_1 :Y \to Z$ and $f_2 : X \to Y$ $C^m$ functions. 
\begin{enumerate}[label = {\roman*)}, ref={\cref{FaaDiBrunoHighOrder}\roman*)}]
\item\label{FaaDiBrunoHighOrder1} The $m^{\text{th}}$ derivative of $f_1 \circ f_2$ is given by
\begin{align*}
D^m[f_1 \circ f_2](x) &= Df_1(f_2(x)) D^mf_2(x) + D^mf_1(f_2(x)) \left( Df_2(x) \right)^{\otimes m}  + \mathcal{P}_m(f_1,f_2)(x)
\end{align*}
where we use the shorthand notation $\left( Df_2(x) \right)^{\otimes m} \isdef \underbrace{\left( Df_2(x) , \dots , Df_2(x) \right)}_{m \text{ times}}$ and $\mathcal{P}_m(f_1,f_2)(x) \isdef \sum_{i=2}^{m-1} \sum_{\pi \in P_m^i}  D^i f_1(f_2(x)) \left(D^{\pi(1)}f_2(x), \cdots ,D^{\pi(i)}f_2(x) \right)$.
\item\label{FaaDiBrunoHighOrder2} Let $f_3 : X \to X$ be another $C^m$ function, then we find
\begin{align*}
D^m[f_1 \circ f_2 \circ f_3](x) &= Df_1(z) \left( Df_2(y) \right) \left( D^mf_3(x) \right) +  Df_1(z)\left( D^mf_2(y) \right) \left( Df_3(x) \right)^{\otimes m} \\
					&\quad + D^m f_1(z) \left( Df_2(y) \right)^{\otimes m} \left( Df_3(x) \right)^{\otimes m} + \mathcal{P}_m(f_1,f_2)(y)\left(Df_3(x) \right)^{\otimes m} \\
					&\quad  + \mathcal{P}_m(f_1 \circ f_2, f_3)(x),
\end{align*}
where $z = f_2(f_3(x))$ and $y=f_3(x)$.
\end{enumerate}
\end{lemma}
\begin{proof} \textit{i)} The equality follows from Fa\`{a} di Bruno's formula and the fact that there is one ordered partition of $m$ of length 1 and one of length $m$.

\textit{ii)} The equality follows from applying Fa\`{a} di Bruno's formula twice. We first apply it to the composition of $f_1 \circ f_2$ and $f_3$, which gives
\begin{align*}
D^m[f_1 \circ f_2 \circ f_3](x) &= D[f_1 \circ f_2](f_3(x)) D^m f_3(x) + D^m[f_1 \circ f_2 ](f_3(x)) \left( Df_3(x) \right)^{\otimes m} \\
					&\quad  + \mathcal{P}_m(f_1 \circ f_2, f_3)(x).
\end{align*}
Then the first derivative of $f_1 \circ f_2$ is given by $Df_1(f_2)Df_2$, and we apply Fa\`{a} di Bruno's formula a second time for the $m^{\text{th}}$ derivative of $f_1 \circ f_2$ to find the desired equality.
\end{proof}

The third component of $\Theta^{}(\Lambda)$ contains the function $(A_c + r)^{-1}$. We want to express the $m^{\text{th}}$ derivative in terms of the $m^{\text{th}}$ derivative of $A_c + r$.

\begin{lemma}\label{HigherDerivativeInverse}Let $R = A_c + r :X_c \to X_c$ be an invertible $C^m$ function with inverse $T$. Then
\begin{align*}
D^m T = - DT D^mr(T) \left( DT \right)^{\otimes m} - DT \mathcal{P}_m(R,T) && \text{ for } m \ge 2.
\end{align*}
\end{lemma}
\begin{proof}
We apply Fa\`{a} di Bruno's formula to $R \circ T$, and notice that its $m^{\text{th}}$ derivative is $0$ as $R \circ T = \operatorname{Id}$. Thus we find
\begin{align*}
DR(T) D^mT = - D^mr(T) \left(DT \right)^{\otimes m} - \mathcal{P}_m(R,T).
\end{align*}
The asserted identity follows by multiplying both sides by the inverse of $DR(T)$, which is $DT$.
\end{proof}
\begin{remark}\label{RemarkC2Difficult}
This lemma is another reason why we start the inductive step after we have proven $C^2$ smoothness, because the first derivative of $\operatorname{Id}$ does not vanish.
\end{remark}

\subsection{Another fixed point operator}

We can now take the $m^{\text{th}}$ derivative of the fixed point identity $\Lambda = \Theta^{}(\Lambda)$, where $\Theta^{}(\Lambda)$ is defined in \cref{FixedPointOperatorEquation}. We will do this component-wise. For the first component we find, by \cref{FaaDiBrunoHighOrder1},
\begin{align*}
D^m r &= A_c D^mk_c + Dg_c(K{}{})D^mK{}{}+D^mg_c(K{}{})\left(DK{}{} \right)^{\otimes m} + \mathcal{P}_m(g_c,K{}{})  \\ 
	&\quad-  Dk_c(R)D^mr - D^mk_c(R)\left(DR \right)^{\otimes m} - \mathcal{P}_m(k_c,R) \\
	&=f_{m,1} + Dg_c(K{}{})D^mK{}{} +  Dk_c(R)D^mr,
\end{align*}
where
\begin{align*}
f_{m,1} \isdef A_c D^mk_c +D^mg_c(K{}{})\left(DK{}{} \right)^{\otimes m} + \mathcal{P}_m(g_c,K{}{})  - D^mk_c(R)\left(DR \right)^{\otimes m} - \mathcal{P}_m(k_c,R).
\end{align*}
For the second component, we find with \cref{FaaDiBrunoHighOrder1}
\begin{align*}
D^mk_u &= A_u^{-1} \left( Dk_u(R)D^mr+D^mk_u(R)\left(DR \right)^{\otimes m} + \mathcal{P}_m(k_u,R) \right) \\
	& \quad - A_u^{-1} \left( Dg_u(K{}{})D^mK{}{}+D^mg_u(K{}{})\left(DK{}{} \right)^{\otimes m} + \mathcal{P}_m(g_u,K{}{}) \right)  \\
	&=f_{m,2} + A_u^{-1} \left( Dk_u(R)D^mr +  D^mk_u(R) (DR)^{\otimes m} - Dg_u(K{}{})D^mK{}{} \right),
\end{align*}
where
\begin{align*}
f_{m,2} \isdef  A_u^{-1} \left( \mathcal{P}_m(k_u,R)  -D^mg_u(K{}{})\left(DK{}{} \right)^{\otimes m} - \mathcal{P}_m(g_u,K{}{}) \right).
\end{align*}
Finally, for the third component we find from \cref{FaaDiBrunoHighOrder2,HigherDerivativeInverse}, with $T= (A_c + r)^{-1}$,
\begin{align*}
D^mk_s &= A_s \left( Dk_s(T)D^mT+D^mk_s(T)\left(DT \right)^{\otimes m} + \mathcal{P}_m(k_s,T) \right) \\
		&\quad + Dg_s(K{}{} \circ T) \left( DK{}{}(T) \right) \left( D^mT \right) +  Dg_s(K{}{} \circ T)\left( D^mK{}{}(T) \right) \left( DT \right)^{\otimes m} \\
					&\quad + D^m g_s(K{}{} \circ T) \left( DK{}{}(T) \right)^{\otimes m} \left( DT \right)^{\otimes m} + \mathcal{P}_m(g_s,K{}{})(T)\left(DT \right)^{\otimes m} \\
					&\quad  + \mathcal{P}_m(g_s \circ K{}{}, T) \\
		&=  A_s \Big( -Dk_s(T)DTD^mr(T)\left(DT\right)^{\otimes} - Dk_s(T)DT\mathcal{P}_m(R,T) \\
		&\quad +D^mk_s(T)\left(DT \right)^{\otimes m} + \mathcal{P}_m(k_s,T) \Big) \\
		&\quad - Dg_s(K{}{} \circ T) \left( DK{}{}(T) \right) DTD^mr(T)\left(DT\right)^{\otimes} - Dg_s(K(T))DT\mathcal{P}_m(R,T)   \\
		&\quad +  Dg_s(K{}{} \circ T)\left( D^mK{}{}(T) \right) \left( DT \right)^{\otimes m} + D^m g_s(K{}{} \circ T) \left( DK{}{}(T) \right)^{\otimes m} \left( DT \right)^{\otimes m}  \\
		&\quad + \mathcal{P}_m(g_s,K{}{})(T)\left(DT \right)^{\otimes m} + \mathcal{P}_m(g_s \circ K{}{}, T) \\
		&=f_{m,3} - A_s Dk_s (T) h_m(D^mr)  + A_s D^mk_s(T) \left( D T \right)^{\otimes m}  \\
		&\quad - Dg_s(K{}{} \circ T) DK{}{} (T) h_m(D^mr) + Dg_s(K{}{} \circ T) D^mK{}{}(T) \left( DT \right)^{\otimes m} ,
\end{align*}
where we define
\begin{align*}
f_{m,3} &\isdef  A_s \Big( - Dk_s(T)DT\mathcal{P}_m(R,T) + \mathcal{P}_m(k_s,T) \Big) - Dg_s(K(T))DT\mathcal{P}_m(R,T)    \\
		&\quad + D^m g_s(K{}{} \circ T) \left( DK{}{}(T) \right)^{\otimes m} \left( DT \right)^{\otimes m}  + \mathcal{P}_m(g_s,K{}{})(T)\left(DT \right)^{\otimes m} + \mathcal{P}_m(g_s \circ K{}{}, T)
\end{align*}
and
\begin{align*}
h_m(\rho)(x) &\isdef DT(x) \rho(T(x))(DT(x))^{\otimes m}.
\end{align*}

Hence, we introduce the fixed point operator, where we  use $\kappa = \tmatrix{D^mk_c}{\kappa_u}{\kappa_s}$, 
\begin{align}
\Theta^{[m+1]} &: \Gamma_{m+1} \isdef C^1_b(X_c, \mathcal{L}^m(X_c,X)) \to C^1(X_c, \mathcal{L}^m(X_c,X)),\nonumber\\
& \begin{pmatrix} \rho \\ \kappa_u \\ \kappa_s \end{pmatrix} \mapsto
\begin{pmatrix*}[l]
f_{m,1} +Dg_c(K{}{})\kappa - Dk_c(R) \rho  \vspace{3 pt} \\
f_{m,2} + A_u^{-1} \left( \kappa_u \circ R \right) (DR)^{\otimes m}  + A_u^{-1} Dk_u(R) \rho  -A_u^{-1} Dg_u(K{}{})\kappa  \vspace{3 pt} \\
f_{m,3}  - A_s Dk_s (T) h_m(\rho) + A_s \left( \kappa_s \circ T \right) \left( D T \right)^{\otimes m}  \\
		\qquad -  Dg_s(K{}{} \circ T)DK{}{} (T) h_m(\rho) + Dg_s(K{}{} \circ T) \left( \kappa \circ T \right) \left( DT \right)^{\otimes m} 
\end{pmatrix*}. \label{FixedPointOperatorCmEquation}
\end{align}
We note that if $\Lambda \in C^m_b(X_c, X)$, then the functions $f_{m,1}$, $f_{m,2}$ and $f_{m,3}$ are bounded, since we assume that $g \in C^m_b(X,X)$. Hence, if $\Lambda \in C^m_b(X_c, X)$, then we have that $\Theta^{[m+1]}: \Gamma_{m+1} \to \Gamma_{m+1}$.

\begin{proposition}\label{FixedPointOperatorCm} Let $\Lambda = \tmatrix{r}{k_u}{k_s} \in \Gamma_0$ be a $C^m$ fixed point of $\Theta^{}$ for $2 \le m <n$. If $\Lambda \in C^m_b(X_c, X)$, then $D^m\Lambda$ is a fixed point of $\Theta^{[m+1]} : \Gamma_{m+1} \to \Gamma_{m+1}$.
\end{proposition}
\begin{proof}
This follows from the above discussion.
\end{proof}

\subsection{Another contraction}

We note that we can apply the estimates of \cref{DifferenceC2Norm} to terms such as $Dg_c(K{}{}) \kappa - Dg_c(K{}{}) \tilde{\kappa}$. However, we do not have estimates for terms such as $(\kappa_u \circ R) (DR)^{\otimes m} - (\tilde{\kappa}_u \circ R) (DR)^{\otimes m}$ with $\kappa_u,\tilde{\kappa}_u \in C^1_b(X_c , \mathcal{L}^m(X_c, X_c))$. Therefore, we start with some  preliminary results about the differences of products.

\begin{lemma}\label{DifferenceCmNorm} Let $X$, $Y$ and $Z$ be Banach spaces, $m \ge 2$ and $h \in C^1_b(X,Y)$.
\begin{enumerate}[label = {\roman*)}, ref={\cref{DifferenceCmNorm}\roman*)}]
\item\label{DifferenceCmNorm1} Let $f,g \in C^0_b(Y, \mathcal{L}^m(Y,Z))$. For $\tilde{h} \in C^0_b(X, \mathcal{L}(X,Y))$ we have the $C^0$-estimate
\begin{align*}
\supnorm{(f \circ h) \tilde{h}^{\otimes m} - (g \circ h)\tilde{h}^{\otimes m} } &\le \supnorm{\tilde{h}}^m \supnorm{f - g }.
\end{align*}
\item\label{DifferenceCmNorm2} Let $f,g \in C^1_b(X, \mathcal{L}^m(Y,Z))$. For $\tilde{h} \in C^1_b(X, \mathcal{L}(X,Y))$ we have the $C^1$-estimate
\begin{align*}
\supnorm{D [ (f \circ h) \tilde{h}^{\otimes m}] - D[  (g \circ h) \tilde{h}^{\otimes m} ] } &\le   \supnorm{Dh} \supnorm{\tilde{h}}^m \supnorm{Df - Dg} \\
						&\quad +m\supnorm{\tilde{h}}^{m-1} \supnorm{D\tilde{h}} \supnorm{f - g}.
\end{align*}
\end{enumerate}
\end{lemma}
\begin{proof} \textit{i)} Let $x \in X$ and recall that the norm of the $k$-linear operator $(f - g)(h(x))$ is given by
\begin{align*}
\operatornorm{(f - g)(h(x)} = \sup_{\substack{\norm{x_i} \le 1 \\ 1 \le i \le m}} \norm{(f - g)(h(x))(x_1, \dots, x_m)},
\end{align*}
from which the desired $C^0$-estimate follows.

\textit{iii)} For the $C^1$-estimate we use the product rule and triangle inequality to find
\begin{align*}
\supnorm{D [ (f \circ h) \tilde{h}^{\otimes m}] - D[  (g \circ h) \tilde{h}^{\otimes m} ]}	& = \supnorm{Df(h)\left(Dh,\tilde{h}^{\otimes m}\right) - Dg(h)\left(Dh,\tilde{h}^{\otimes m} \right)} \\
						&\quad + \supnorm{ (f \circ h) D[\tilde{h}^{\otimes m}] - (g \circ h) D[\tilde{h}^{\otimes m}]}.
\end{align*}
We note that $D[\tilde{h}^{\otimes m}] = \sum_{i=0}^{m-1} \left( \tilde{h}^{\otimes i}, D\tilde{h}, \tilde{h}^{\otimes m-1-i} \right)$, hence we find
\begin{align*}
\supnorm{D [ (f \circ h) \tilde{h}^{\otimes m}] - D[  (g \circ h) \tilde{h}^{\otimes m} ]} &\le \supnorm{Dh} \supnorm{D\tilde{h}}^m \supnorm{Df - Dg} \\
						&\quad + \sum_{i=1}^{m-1} \supnorm{\tilde{h}}^i\supnorm{D\tilde{h}} \supnorm{\tilde{h}}^{m-1-i}  \supnorm{ f -g},
\end{align*}
from which the desired estimate follows.
\end{proof}

\begin{proposition}\label{CmContinuity}  Let $m \ge 2$ and assume that $L_g$ and $L_c$ are small in the sense of \cref{SmallnessConditions} for $n = m$.  There exists an $\varepsilon_0 >0 $ such for all $\varepsilon < \varepsilon_0$ it holds that  if $\supnorm{D^2 g}, \supnorm{D^2k_c} \le  \varepsilon$ and the fixed point $\Lambda \in \Gamma_1(\delta(\varepsilon))$ of $\Theta^{}$ lies in $C^m_b(X_c,X)$, then $\Theta^{[m+1]} : \Gamma_{m+1} \to \Gamma_{m+1}$ is a contraction with respect to the $C^1$ norm.
\end{proposition}
\begin{proof} Let $\varepsilon > 0$ and $\supnorm{D^2 g}, \supnorm{D^2k_c} \le \varepsilon$. Let $\mathcal{M} = \tmatrix{\rho}{\kappa_u}{\kappa_s},  \tilde{\mathcal{M}}= \tmatrix{\tilde{\rho}}{\tilde{\kappa}_u}{ \tilde{\kappa}_s} \in \Gamma_{\tilde{m}}$, where we denote $\tilde{m}=m+1$.

To show that $\Theta^{[\tilde{m}]}$ is a $C^1$ contraction, we use the same steps as we used in the proofs of \cref{C1Continuity,C2Continuity}.

\begin{enumerate}[label=\Alph*)]
\item\label{CmContinuity1} We prove that $\Theta^{[\tilde{m}]}$ is a contraction with respect to the $C^0$ norm, independent of $\varepsilon$,

\item\label{CmContinuity2} We show the existence of a constant $\theta_{\tilde{m}}(\varepsilon)$ such that 
\begin{align*}
\supnorm{D[\Theta^{[\tilde{m}]}(\mathcal{M})] - D[\Theta^{[\tilde{m}]}(\tilde{\mathcal{M}})]} \le \theta_{\tilde{m}}(\varepsilon) \onenorm{\mathcal{M} - \tilde{\mathcal{M}}},
\end{align*}

\item\label{CmContinuity3} We show that $\varepsilon > 0$ can be chosen so that $\theta_{\tilde{m}}(\varepsilon) <1$, thus showing that $\Theta^{[\tilde{m}]}$ is a contraction with respect to the $C^1$ norm. 
\end{enumerate}

\textbf{Step \cref{CmContinuity1}} We want to find $\theta_m<1$ such that 
\begin{align*}
\supnorm{\Theta^{[\tilde{m}]}(\mathcal{M}) - \Theta^{[\tilde{m}]}(\tilde{\mathcal{M}})} \le \theta_m \supnorm{\mathcal{M} - \tilde{\mathcal{M}}}.
\end{align*}
As we did in the proofs of \cref{C1Continuity,C2Continuity}, we will find the contraction constant component-wise, i.e. we will show that
\begin{align*}
\supnorm{ \Theta^{[\tilde{m}]}_i(\mathcal{M}) - \Theta^{[\tilde{m}]}_i(\tilde{\mathcal{M}})} &\le \theta_{m,i} \supnorm{\mathcal{M} - \tilde{\mathcal{M}}}
\end{align*}
for $\theta_{m,i}$ given explicitly in equation \cref{ContractionConstants1,ContractionConstants2,ContractionConstants3} and $i=1,2,3$. Furthermore, we note that we will directly apply \cref{DifferenceC2Norm2,DifferenceCmNorm2} to obtain the estimates \cref{ CmContractionEquation1,, CmContractionEquation2,, CmContractionEquation3}

\underline{\textit{$\rho$-component:}} We recall that $\supnorm{Dg_c} \le L_g$ and $\supnorm{Dk_c} \le L_c$, which follows from assumption \cref{MainTheoremKnownBounds } of \cref{MainTheorem}. We estimate
\begin{align}
\supnorm{\Theta^{[\tilde{m}]}_1(\mathcal{M}) - \Theta^{[\tilde{m}]}_1(\tilde{\mathcal{M}})} &\le \left( L_g + L_c  \right) \supnorm{\mathcal{M} - \tilde{\mathcal{M}}} = \theta_{m,1} \supnorm{\mathcal{M} - \tilde{\mathcal{M}}}. \label{ CmContractionEquation1}
\end{align}

\underline{\textit{$\kappa_u$-component:}} 
Similarly, we have the estimates $\supnorm{Dg_u} \le L_g$, $\supnorm{(DR)} \le \operatornorm{A_c} + L_r $ and $\supnorm{k_u} \le L_u$, hence we find
\begin{align}
\supnorm{\Theta^{[\tilde{m}]}_2(\mathcal{M}) - \Theta^{[\tilde{m}]}_2(\tilde{\mathcal{M}})}	&\le \operatornorm{A_u^{-1}} \left((\operatornorm{A_c} + L_r)^m + L_u + L_g \right) \supnorm{\mathcal{M} - \tilde{\mathcal{M}}} \nonumber \\
									&= \theta_{m,2} \supnorm{\mathcal{M} - \tilde{\mathcal{M}}}. \label{ CmContractionEquation2}
\end{align} 

\underline{\textit{$\kappa_s$-component:}} Finally, we have the bounds $\supnorm{DT} \le L_{-1}$,  $\supnorm{Dk_s} \le L_s$, $\supnorm{Dg_s} \le L_g$, $\supnorm{DK{}{}} \le 1 + L_c$ and 
\begin{align}
\supnorm{h_m(\rho) - h_m(\tilde{\rho})} &= \supnorm{DT \left( \rho \circ T - \tilde{\rho} \circ T \right) (DT)^{\otimes m}}  
							\le L_{-1}^{m+1} \supnorm{\mathcal{M} - \tilde{\mathcal{M}}}. \label{HInverseDifference}
\end{align}
We infer that
\begin{align}
\supnorm{\Theta^{[\tilde{m}]}_3(\mathcal{M}) &- \Theta^{[\tilde{m}]}_3(\tilde{\mathcal{M}})} \nonumber \\
									&\le \left( \operatornorm{A_s} \left( L_s L_{-1}^{m+1}  +  L_{-1}^{m}  \right) + L_g(1+ L_c)L_{-1}^{m+1} + L_g L_{-1}^m \right) \supnorm{\mathcal{M} - \tilde{\mathcal{M}}} \nonumber \\
									&= \theta_{m,3} \supnorm{\mathcal{M} - \tilde{\mathcal{M}}}. \label{ CmContractionEquation3}
\end{align}

\underline{\textit{Contraction constant:}} We can now estimate $\supnorm{\Theta^{[\tilde{m}]}(\mathcal{M}) - \Theta^{[\tilde{m}]}(\tilde{\mathcal{M}})}$ with inequalities  \cref{ CmContractionEquation1, CmContractionEquation2, CmContractionEquation3}. We have
\begin{align}
\supnorm{\Theta^{[\tilde{m}]}(\mathcal{M}) - \Theta^{[\tilde{m}]}(\tilde{\mathcal{M}})} &= \max_{i=1,2,3} \left\{ \supnorm{\Theta^{[\tilde{m}]}_i(\mathcal{M}) - \Theta^{[\tilde{m}]}_i(\tilde{\mathcal{M}})} \right\} \nonumber \\
										&\le \max_{i=1,2,3} \left\{ \theta_{m,i} \supnorm{\mathcal{M} - \tilde{\mathcal{M}}} \right\} \nonumber \\
										&= \theta_m \supnorm{\mathcal{M} - \tilde{\mathcal{M}}}. \label{CmContractionEq1}
\end{align}
Here we define $\theta_m \isdef \max\left\{ \theta_{m,1}, \theta_{m,2}, \theta_{m,3} \right\}$.  We have assumed that \cref{SmallnessConditions} holds for $n = m$. Therefore, we have $\theta_{m,i}<1$ and thus $\theta_m <1$. We conclude that $\Theta^{[\tilde{m}]}$ is a contraction with respect to the $C^0$ norm. \vspace{1 \baselineskip}

\textbf{Step \cref{CmContinuity2}} Analogous to step \cref{CmContinuity1}, we want to prove the component-wise inequality
\begin{align*}
\supnorm{D[\Theta^{[\tilde{m}]}_i(\mathcal{M})] - D[\Theta^{[\tilde{m}]}_i(\tilde{\mathcal{M}})]} &\le \left(\theta_{\tilde{m},i} + C_{\tilde{m},i}(\varepsilon) \right) \onenorm{\mathcal{M} - \tilde{\mathcal{M}}},
\end{align*} 
with $\theta_{\tilde{m},i}$ defined in \cref{ContractionConstants1,ContractionConstants2,ContractionConstants3} and $C_{\tilde{m},i}$ will be defined during the proof. We note that $\tmatrix{r}{k_u}{k_s} \in \Gamma_1(\delta(\varepsilon))$, hence $\supnorm{D^2r},\supnorm{D^2k_u},\supnorm{D^2k_s} \le \delta = \delta(\varepsilon)$.

\underline{\textit{$\rho$-component:}} Using the same estimates as we did for the $\rho$-component in step \cref{C2Continuity2} of the proof of \cref{C2Continuity}, we find
\begin{align}
\supnorm{D[\Theta^{[\tilde{m}]}_1(\mathcal{M})] - D[\Theta^{[\tilde{m}]}_1(\tilde{\mathcal{M}})]}	&\le \left( L_g + (1 + L_c) \varepsilon + L_c + ( \operatornorm{A_c}+ L_r) \varepsilon \right) \onenorm{\mathcal{M} - \tilde{\mathcal{M}}} \nonumber \\
							&\le \left( \theta_{\tilde{m},1} + C_{\tilde{m},1}(\varepsilon)\right) \onenorm{\mathcal{M} - \tilde{\mathcal{M}}}, \label{ CmContractionEquation4}
\end{align}
where we define $C_{\tilde{m},1}(\varepsilon) \isdef \left( 1 + L_c + \operatornorm{A_c} +  L_r \right)\varepsilon$.

\underline{\textit{$\kappa_u$-component:}} We have
\begin{align}
\supnorm{D[\Theta^{[\tilde{m}]}_2(\mathcal{M})] &- D[\Theta^{[\tilde{m}]}_2(\tilde{\mathcal{M}})]} \nonumber \\
									 &\qquad \quad \le  \operatornorm{A_u^{-1}} \Big( \supnorm{D[(\kappa_u \circ R)(DR)^{\otimes m}]-D[(\tilde{\kappa}_u \circ R)(DR)^{\otimes m}]} \nonumber \\
									&\qquad \quad \quad +   \supnorm{D[Dk_u(R) \rho] - D[Dk_u(R) \tilde{\rho}]} \nonumber \\
									&\qquad \quad \quad +  \supnorm{ D[ Dg_u(K{}{}) \kappa] - D[Dg_u(K{}{}) \tilde{\kappa}]}  \Big). \label{uComponentEqCm1}
\end{align}			
By applying \cref{DifferenceC2Norm2,DifferenceCmNorm2} we find the estimates
\begin{align*}
\supnorm{D[( \kappa_u \circ R) (DR)^{\otimes m}] - D[( \tilde{\kappa}_u \circ R) &(DR)^{\otimes m}]}   \\
							&\le (\operatornorm{A_c} + L_r)^{m+1} \supnorm{D\kappa_u - D \tilde{\kappa}_u} \\
							&\quad + m (\operatornorm{A_c} + L_r)^{m-1} \delta(\varepsilon)\supnorm{\kappa_u - \tilde{\kappa}_u}, \\
\supnorm{D[Dk_u(R) \rho] - D[Dk_u(R) \tilde{\rho}]}   &\le \left( (\operatornorm{A_c} + L_r) \delta(\varepsilon )+ L_u  \right)\onenorm{\mathcal{M} - \tilde{\mathcal{M}}}, \\
\supnorm{ D[ Dg_u(K{}{}) \kappa] - D[Dg_u(K{}{}) \tilde{\kappa}]}  &\le  \left( (1 + L_c) \varepsilon + L_g \right) \onenorm{\mathcal{M} -\tilde{\mathcal{M}}}.
\end{align*}
By combining these with \cref{uComponentEqCm1} we estimate
\begin{align}
\supnorm{D[\Theta^{[\tilde{m}]}_2(\mathcal{M})] - D[\Theta^{[\tilde{m}]}_2(\tilde{\mathcal{M}})]} &\le \operatornorm{A_u^{-1}} \Big(    ( \operatornorm{A_c} + L_r)^{\tilde{m}} + L_g + L_u \nonumber \\
							&\quad + \left( m( \operatornorm{A_c} + L_r)^{m-1}   +  \operatornorm{A_c} +L_r \right)\delta(\varepsilon)  \nonumber \\
							&\quad + (1 + L_c)\varepsilon \Big) \onenorm{\mathcal{M} - \tilde{\mathcal{M}}} \nonumber \\
							&\le \left( \theta_{\tilde{m},2} + C_{\tilde{m},2}(\varepsilon)\right) \onenorm{\mathcal{M} - \tilde{\mathcal{M}}}, \label{ CmContractionEquation5}
\end{align}
where we define 
\begin{align*}C_{\tilde{m},2}(\varepsilon) \isdef \operatornorm{A_u^{-1}} ((1+ L_c)\varepsilon + (\operatornorm{A_c} + L_r + m(\operatornorm{A_c} + L_r)^{m-1} ))\delta(\varepsilon).
\end{align*}

\underline{\textit{$\kappa_s$-component:}} We have
\begin{align}
\supnorm{D[\Theta^{[\tilde{m}]}_3(\mathcal{M})] &- D[\Theta^{[\tilde{m}]}_3(\tilde{\mathcal{M}})]} \nonumber \\
									&\le \operatornorm{A_s} \supnorm{D[Dk_s(T)h_m(\rho)] - D[Dk_s(T) h_m(\tilde{\rho})]} \nonumber \\
									&\quad +\operatornorm{A_s} \supnorm{D[(\kappa_s \circ T)(DT)^{\otimes m}] -D[(\tilde{\kappa}_s \circ T)(DT)^{\otimes m}]} \nonumber \\
									&\quad + \supnorm{D[Dg_s(K{}{} \circ T)DK{}{}(T)h_m(\rho)] - D[Dg_s(K{}{} \circ T)DK{}{}(T)h_m(\tilde{\rho})]} \nonumber \\
									&\quad + \supnorm{D[Dg_s(K{}{} \circ T)( \kappa \circ T) (DT)^{\otimes m}] - D[Dg_s(K{}{} \circ T)(\tilde{\kappa} \circ T) (DT)^{\otimes m}]}. \label{ sComponentEqCm1}
\end{align}
Before we apply \cref{DifferenceC2Norm2}, we start by deriving an upper bound for $Dh_m(\rho) - Dh_m(\tilde{\rho})$.
\begin{align*}
Dh_m(\rho) &=D^2T \left( \operatorname{Id}, (\rho \circ T)(DT)^{\otimes m} \right) + DT \left( D\rho(T) \right)\left(DT, (DT)^{\otimes m} \right)  \\
			&\quad + DT (\rho \circ T) \sum_{i=0}^{m-1} \left( (DT)^{\otimes i}, D^2T, (DT)^{\otimes m-1-i} \right).
\end{align*}
From \cref{HigherDerivativeInverse} we see that $D^2T$ is bounded by $\supnorm{DT D^2r(T) \left(DT\right)^{\otimes 2}}$ since $\mathcal{P}_2 = 0$. Using the estimates $\supnorm{DT} \le L_{-1}$ and $\supnorm{D^2T} \le L_{-1}^3 \delta(\varepsilon)$, see \cref{EstimateDT} for the latter bound, we find
\begin{align}
\supnorm{Dh_m(\rho) - Dh_m(\tilde{\rho})} \le \left( L_{-1}^{m+3} \delta(\varepsilon) + L_{-1}^{m+2} + m L_{-1}^{m+3}\delta(\varepsilon) \right) \onenorm{\rho - \tilde{\rho}}. \label{DHInverseDifference}
\end{align}
We can now apply \cref{DifferenceC2Norm2} and estimate \cref{HInverseDifference} for the bound on $h_m(\rho) - h_m(\tilde{\rho})$ to get
\begin{align}
\supnorm{D[Dk_s&(T)h_m(\rho)] - D[Dk_s(T) h_m(\tilde{\rho})]} \nonumber \\
				& \le  \supnorm{D^2k_s} \supnorm{DT} \supnorm{h_m(\rho) - h_m(\tilde{\rho})} + \supnorm{Dk_s}\supnorm{Dh_m(\rho) - Dh_m(\tilde{\rho})} \nonumber \\
				& \le \left( L_{-1}^{m+2} \delta(\varepsilon) + (m+1)L_{-1}^{m+3}L_s \delta(\varepsilon)  + L_{-1}^{m+2} L_s \right)\onenorm{\mathcal{M} - \tilde{\mathcal{M}}}. \label{ sComponentCmEq1} 
\end{align}
By applying \cref{DifferenceCmNorm2} we find
\begin{align}
\supnorm{D[(\kappa_s \circ T)&(DT)^{\otimes m}] -D[(\tilde{\kappa}_s \circ T)(DT)^{\otimes m}]} \nonumber \\
				 &\quad \le \supnorm{DT} \supnorm{(DT)^{\otimes m}}  \supnorm{D\kappa_s - D \tilde{\kappa}_s} + \supnorm{D (DT)^{\otimes m}} \supnorm{\kappa_s - \tilde{\kappa}_s} \nonumber \\
				 &\quad \le \left(  L_{-1}^{m+1} + m L_{-1}^{m+2} \delta(\varepsilon) \right) \onenorm{\mathcal{M} - \tilde{\mathcal{M}}}.  \label{ sComponentCmEq2}
\end{align}
For the third term of the right hand side of \cref{ sComponentEqCm1}, we take $x \mapsto Dg_s(K{}{}(x))DK{}{}(x)$ for the functions $f_1$ and $g_1$ in \cref{DifferenceC2Norm2} and use $T$ as $h$. Recall that $\gamma(\varepsilon) = \max\{ \varepsilon,\delta(\varepsilon)\}$, so that
\begin{align*}
\supnorm{D[ Dg_s(K{}{})DK{}{}]} \le \supnorm{D^2g_s(DK{}{},DK{}{})} + \supnorm{Dg_s D^2K{}{}} \le (1+L_c)^2 \varepsilon + L_g \gamma(\varepsilon).
\end{align*}
Hence we find the estimate, using \cref{DifferenceC2Norm2} and \cref{DHInverseDifference,HInverseDifference},
\begin{align}
\supnorm{D[Dg_s(K{}{} \circ T)&DK{}{}(T)h_m(\rho)] - D[Dg_s(K{}{} \circ T)DK{}{}(T)h_m(\tilde{\rho})]} \nonumber \\
			 &\le \supnorm{D[Dg_s(K{}{})DK{}{}]} \supnorm{DT} \supnorm{h_m(\rho) - h_m(\tilde{\rho})} \nonumber \\
			&\quad + \supnorm{Dg_s(K{}{})DK{}{}} \supnorm{Dh_m(\rho) - Dh_m(\tilde{\rho})} \nonumber \\
			&\le \left( L_{-1}^{m+2} L_g \gamma(\varepsilon) + (m+1) L_{-1}^{m+3} L_g ( 1 + L_c) \delta(\varepsilon) \right) \onenorm{\mathcal{M} - \tilde{\mathcal{M}}} \label{ sComponentCmEq3} \\
			&\quad +  L_{-1}^{m+2} L_g(1+L_c)  \onenorm{\mathcal{M} - \tilde{\mathcal{M}}} + L_{-1}^{m+2} (1+L_c)^2\varepsilon \onenorm{\mathcal{M} - \tilde{\mathcal{M}}}. \label{ sComponentCmEq4}
\end{align}
For the final term, we will apply \cref{DifferenceCmNorm}, which gives
\begin{align}
\supnorm{D[Dg_s(K{}{} &\circ T)\kappa(T) (DT)^{\otimes m}] - D[Dg_s(K{}{} \circ T)\tilde{\kappa}(T) (DT)^{\otimes m}]} \nonumber \\
		&\le \supnorm{D^2g_s} \supnorm{D[K{}{} \circ T]} \supnorm{(\kappa \circ T)(DT)^{\otimes m} - (\tilde{\kappa} \circ T)(DT)^{\otimes m}} \nonumber \\
		&\quad + \supnorm{Dg_s} \supnorm{D[(\kappa \circ T)(DT)^{\otimes m}] - D[(\tilde{\kappa} \circ T)(DT)^{\otimes m}]} \nonumber \\
		&\le \supnorm{D^2g_s} \supnorm{DK{}{}}\supnorm{DT} \supnorm{\kappa - \tilde{\kappa}} \supnorm{(DT)^{\otimes m}} \nonumber \\
		&\quad + L_g \left( \supnorm{DT} \supnorm{(DT)^{\otimes m}} + \supnorm{D[(DT)^{\otimes m}]} \right)\onenorm{\kappa - \tilde{\kappa}} \nonumber \\
		&\le  L_{-1}^{m+1} (1 + L_c) \varepsilon \onenorm{\mathcal{M} - \tilde{\mathcal{M}}} + \left( L_g L_{-1}^{m+1} + m L_g L_{-1}^{m+2} \delta(\varepsilon) \right) \onenorm{\mathcal{M} - \tilde{\mathcal{M}}}. \label{ sComponentCmEq5}
\end{align}
We first recall that \cref{ sComponentCmEq1, sComponentCmEq2} should be multiplied by $\operatornorm{A_s}$. We will collect the terms without $\varepsilon$ in \cref{ sComponentCmEq1, sComponentCmEq2, sComponentCmEq4}, which sum up to $\theta_{\tilde{m},3}$. Then we collect the terms containing $\delta(\varepsilon)$ and $\gamma(\varepsilon)$ in \cref{ sComponentCmEq1,, sComponentCmEq2,, sComponentCmEq3,, sComponentCmEq5}. We estimate $\delta(\varepsilon)$ by $\gamma(\varepsilon)$, so that those terms together are bounded by $\tilde{m}\theta_{\tilde{m}+1,3} \gamma(\varepsilon)$. Together with terms containing $\varepsilon$, we get
\begin{align*}
C_{\tilde{m},3}(\varepsilon) \isdef \tilde{m}\theta_{\tilde{m}+1,3} \gamma(\varepsilon) +L_{-1}^{\tilde{m}}(1+L_c)\varepsilon + L_{-1}^{\tilde{m}+1}(1+L_c)^2\varepsilon.
\end{align*}
Finally, we conclude that \cref{ sComponentEqCm1} reduces to
\begin{align}
\supnorm{D[\Theta^{[\tilde{m}]}_3(\mathcal{M})] - D[\Theta^{[\tilde{m}]}_3(\tilde{\mathcal{M}})]} &\le \left( \theta_{\tilde{m},3} + C_{\tilde{m},3}(\varepsilon) \right) \onenorm{\mathcal{M} - \tilde{\mathcal{M}}}. \label{ CmContractionEquation6}
\end{align}

\underline{\textit{Lipschitz constant:}} Inequalities \cref{ CmContractionEquation4, CmContractionEquation5, CmContractionEquation6} give
\begin{align}
\supnorm{D[\Theta^{[\tilde{m}]}(\mathcal{M})] - D[\Theta^{[\tilde{m}]}(\tilde{\mathcal{M}})]} &= \max_{i=1,2,3} \left\{ \supnorm{D[\Theta^{[\tilde{m}]}_i(\mathcal{M})] - D[\Theta^{[\tilde{m}]}_i(\tilde{\mathcal{M}})]} \right\} \nonumber \\
										&\le \max_{i=1,2,3} \left\{( \theta_{\tilde{m},i} + C_{\tilde{m},i}(\varepsilon)) \onenorm{\mathcal{M} - \tilde{\mathcal{M}}} \right\} \nonumber \\
										&\le \theta_{\tilde{m}}(\varepsilon) \onenorm{\mathcal{M} - \tilde{\mathcal{M}}}. \label{CmContractionEq2}
\end{align}
Here we define $\theta_{\tilde{m}}(\varepsilon) \isdef \max\left\{ \theta_{\tilde{m},1}, \theta_{\tilde{m},2}, \theta_{\tilde{m},3} \right\} + \max \{ C_{\tilde{m},1}(\varepsilon), C_{\tilde{m},2}(\varepsilon), C_{\tilde{m},3}(\varepsilon)\}$. \vspace{1 \baselineskip}

\textbf{Step \cref{CmContinuity3}} From \cref{SmallnessConditions} it follows that 
\begin{align*}
\max\left\{ \theta_{\tilde{m},1}, \theta_{\tilde{m},2}, \theta_{\tilde{m},3} \right\} < 1.
\end{align*}
As $\delta(\varepsilon) \downarrow 0$ when $\varepsilon \downarrow 0$, see \cref{UniformSecondDerivativeBound}, we have
\begin{align*}
\lim_{\varepsilon \to 0} \max \{ C_{\tilde{m},1}(\varepsilon), C_{\tilde{m},2}(\varepsilon), C_{\tilde{m},3}(\varepsilon)\} = 0.
\end{align*}
Therefore, we can find an $\varepsilon_0 > 0 $ such that $\theta_{\tilde{m}}(\varepsilon) < 1$ for all $\varepsilon < \varepsilon_0$. Then estimates  \cref{CmContractionEq1,CmContractionEq2} give
\begin{align*}
\onenorm{\Theta^{[\tilde{m}]}(\mathcal{M}) &- \Theta^{[\tilde{m}]}(\tilde{\mathcal{M}})} \\
											&= \max\left\{ \supnorm{\Theta^{[\tilde{m}]}(\mathcal{M}) - \Theta^{[\tilde{m}]}(\tilde{\mathcal{M}})}, \supnorm{D[\Theta^{[\tilde{m}]}(\mathcal{M})] - D[\Theta^{[\tilde{m}]}(\tilde{\mathcal{M}})]} \right\} \\
											&\le \max \left\{ \theta_{m} \supnorm{\mathcal{M} - \tilde{\mathcal{M}}}, \theta_{\tilde{m}}(\varepsilon) \onenorm{\mathcal{M} - \tilde{\mathcal{M}}} \right\} \\
											&\le \lambda_{\tilde{m}} \onenorm{\mathcal{M} - \tilde{\mathcal{M}}}.
\end{align*}
Here we define the contraction constant $\lambda_{\tilde{m}} \isdef \max\{\theta_{m} , \theta_{\tilde{m}}(\varepsilon) \}$ which is smaller than $1$ by our previous discussion. We conclude that $\Theta^{[\tilde{m}]}: \Gamma_{\tilde{m}} \to \Gamma_{\tilde{m}}$ is a contraction with respect to the $C^1$ norm for all $\varepsilon < \varepsilon_0$.
\end{proof}

\begin{corollary}\label{CnExistence}
Let $\varepsilon >0$ be such that $F{}{} :X \to X$ satisfies the conditions of \cref{MainTheorem,C1Continuity,C2Continuity,CmContinuity}.  Then the image of $K{}{}$ is a $C^n$ center manifold for $F{}{}$.
\end{corollary}
\begin{proof}
The proof follows by induction on the smoothness of the conjugacy and the conjugate dynamics, with similar arguments as we had in the proof of \cref{C2Continuity}.
\end{proof}

\section{Existence and uniqueness of the center manifold}\label{UniquenessSection}

\subsection{Proof of \texorpdfstring{\cref{MainTheorem}}{the main theorem}}\label{ProofOfMainTheorem}

From \cref{CnExistence} it follows that there exists an $\varepsilon >0$ such that  if $\supnorm{D^2g}$ and $\supnorm{D^2k_c}$ are smaller than $\varepsilon$, then there exists a $C^n$ center manifold for $F{}{}: X \to X$. However, in \cref{MainTheorem}, the only assumption on the second derivative of $k_c$ and $g$ is that they are both bounded. We will use a simple scaling argument to show that we can bound $D^2k_c$ and $D^2g$ by $\varepsilon$ without affecting the assumed bound on $Dk_c$ and $Dg$ from \cref{MainTheorem}.

\begin{lemma}\label{SecondDerivativeBounded} Let $\varepsilon > 0$ and define $h^\varepsilon(x) \isdef \varepsilon^{-1} h(\varepsilon x)$ for $h \in C^2_b(X,Y)$.
\begin{enumerate}[label = {\roman*)}, ref={\cref{SecondDerivativeBounded}\roman*)}]
\item\label{SecondDerivativeBounded1} We have $h^\varepsilon (0) = 0$ if $h(0) = 0$, $Dh^\varepsilon (0) = Dh(0)$ and $\supnorm{Dh^\varepsilon} = \supnorm{Dh}$.
\item\label{SecondDerivativeBounded2} We have $\supnorm{D^2h^\varepsilon} = \varepsilon \supnorm{D^2h}$.
\item\label{SecondDerivativeBounded3} For $h_1 \in C_b^2(Y,Z)$ and $h_2 \in C_b^2(X,Y)$ we have $h_1^\varepsilon \circ h_2^\varepsilon = (h_1 \circ h_2)^\varepsilon$.
\end{enumerate}
\end{lemma}
\begin{proof} 
These results follow from straightforward computations.
\end{proof}

We can now prove \cref{MainTheorem}.

\begin{proof}[\underline{Proof of \cref{MainTheorem}}] Let $k_c :X_c \to X$ be chosen such that 
\begin{align*}
k_c \in \left\{ h \in C^{n}_b(X_c,X_c) \mid h(0) = 0 , \ Dh(0) = 0 \text{ and } \supnorm{Dh} \le L_c \right\}.
\end{align*}
Then it follows from \cref{SecondDerivativeBounded1} that for all $\varepsilon >0$ we have
\begin{align*}
k_c^{\varepsilon} \in \left\{ h \in C^{n}_b(X_c,X_c) \mid h(0) = 0 , \ Dh(0) = 0 \text{ and } \supnorm{Dh} \le L_c \right\}.
\end{align*}
Likewise, if $F{}{} = A + g$ satisfies the conditions of \cref{MainTheorem}, we have for all $\varepsilon > 0$ that $F{}{}^{\varepsilon} = A + g^{\varepsilon}$ with
\begin{align*}
g^{\varepsilon} \in \left\{ h \in C^{n}_b(X,X) \mid h(0) = 0 , \ Dh(0) = 0 \text{ and } \supnorm{Dh} \le L_g \right\}.
\end{align*}
Hence $F{}{}^{\varepsilon}$ also satisfies the conditions of \cref{MainTheorem}.

By \cref{SecondDerivativeBounded2} we can apply \cref{CnExistence} to $F{}{}^{\varepsilon}$ and $k_c^{\varepsilon}$ for $\varepsilon$ sufficiently small. We fix $\varepsilon > 0$ sufficiently small. \cref{CnExistence} then provides a $K{}{} :X_c \to X$ which conjugates $f^{\varepsilon} :X \to X$ with $A_c + r : X_c \to X_c$. Then we find from \cref{SecondDerivativeBounded3} that
\begin{align*}
 (F{}{} \circ K{}{}^{1/\varepsilon})^\varepsilon  =  F{}{}^\varepsilon \circ K{}{} = K{}{} \circ (A_c + r) =  (K{}{}^{1/\varepsilon} \circ (A_c + r^{1/\varepsilon}))^\varepsilon,
\end{align*}
and, again by \cref{SecondDerivativeBounded3},
\begin{align*}
F{}{} \circ K{}{}^{1/\varepsilon} = \left( (F{}{} \circ K{}{}^{1/\varepsilon})^\varepsilon \right)^{1/\varepsilon}  =  \left( \left(K{}{}^{1/\varepsilon} \circ (A_c + r^{1/\varepsilon}) \right)^\varepsilon \right)^{1/\varepsilon} = K{}{}^{1/\varepsilon} \circ (A_c + r^{1/\varepsilon}).
\end{align*}
Thus we see that $K{}{}^{1/\varepsilon} :X_c \to X$ conjugates $F{}{}: X \to X$ with $A_c + r^{1/\varepsilon} : X_c \to X_c$. Furthermore, by combining \cref{SecondDerivativeBounded1} with \cref{C1Existence}, it follows that $r^{1/\varepsilon}$ satisfies property \cref{PropertiesR} and $k_u^{1/\varepsilon}$ and $k_s^{1/\varepsilon}$ satisfy property \cref{PropertiesK}.
\end{proof}

\subsection{Uniqueness of the center manifold}

Since we found our conjugacy $K{}{}$ and dynamical system $A_c + r$ with a contraction argument, we have uniqueness of $\tmatrix{r}{k_u}{k_s}$ in $\Gamma_0$. However, we can improve the uniqueness of $k_u$ and $k_s$ to arbitrary bounded functions, and the uniqueness of $r$ to arbitrary bounded  functions such that $A_c + r$ is invertible.

\begin{lemma}\label{UniqueCenterManifold}Let $F{}{}:X \to X$ and $k_c : X_c \to X_c$ satisfy the conditions of \cref{MainTheorem}. Then 
\begin{align}
F{}{} \circ \begin{pmatrix}\operatorname{Id} + k_c \\ k_u \\ k_s \end{pmatrix} = \begin{pmatrix} \operatorname{Id} + k_c \\ k_u \\ k_s \end{pmatrix} \circ (A_c + r) \label{UniqueCenterManifoldEq}
\end{align}
 has a unique solution for $k_u \in C^0_b(X_c, X_u)$, $k_s \in C^0_b(X_c,X_s)$ and $r \in C^0_b(X_c, X_c)$  with the property that $A_c + r$ is a homeomorphism.
\end{lemma}
\begin{proof} 
Let $\Lambda  = \tmatrix{r}{k_u}{k_s}  \in C_b^n(X_c ,X)$ be obtained from \cref{MainTheorem} and let $\mathcal{M} =  \tmatrix{\tilde{r}}{\tilde{k}_u}{\tilde{k}_s} \in C^0_b(X_c,X)$ be such that \cref{UniqueCenterManifoldEq} holds. Then we know from \cref{FixedPointOperatorProposition} that $\mathcal{M} = \Theta^{}(\mathcal{M})$ and $\Lambda = \Theta^{}(\Lambda)$. We can mimic step \cref{C1Continuity1} of the proof of \cref{C1Continuity}, since \cref{DifferenceC0Norm1,DifferenceComponentThree1} can both be applied to the components of $\Theta^{}(\mathcal{M}) - \Theta^{}(\Lambda)$, even though  $\mathcal{M}$ is merely $C^0$, see \cref{UniquenessInverseR}. Hence we obtain
\begin{align*}
\supnorm{\mathcal{M} - \Lambda} = \supnorm{\Theta^{}(\mathcal{M}) - \Theta^{}(\Lambda)} \le \theta_0 \supnorm{\mathcal{M} - \Lambda}.
\end{align*}
Since $\theta_0 < 1$, we find $\mathcal{M} = \Lambda$.
\end{proof}

We conclude from \cref{UniqueCenterManifold} that, given $k_c : X_c \to X_c$, the conjugacy between the center dynamical system $R : X_c \to X_c$ and the original dynamical system $F{}{}:X \to X$ is unique in the space of continuous functions with bounded unstable and stable components. Additionally, we want to show that the center manifold is unique independent of our choice of $k_c : X_c \to X_c$. That is, we want to prove that the image of the conjugacy does not depend on a given $k_c : X_c \to X_c$.

\begin{proposition}\label{LuckManifoldIndifferent}Let $F{}{}:X \to X$ and $k_c, \tilde{k}_c : X_c \to X_c$ satisfy the conditions  of \cref{MainTheorem}. Then the image of $K{}{} = \iota + \tmatrix{ k_c}{k_u}{k_s}$ and $\tilde{K{}{}} = \iota +  \tmatrix{ \tilde{k}_c}{\tilde{k}_u}{\tilde{k}_s}$ are the same, for $K{}{},\tilde{K{}{}}$ the (unique) conjugacy obtained from \cref{MainTheorem}.
\end{proposition}

\begin{proof}
From \cref{GlobalDiffeomorphism1} it follows that $\operatorname{Id} +  k_c$ is invertible, and from \cref{GlobalDiffeomorphism2} we see that its inverse is given by $\operatorname{Id} + \psi$ for some bounded function $\psi : X_c \to X_c$. In particular, we can write
\begin{align*}
\operatorname{Id}  + \tilde{k}_c = (\operatorname{Id}  + k_c) \circ ( \operatorname{Id} + \psi ) \circ ( \operatorname{Id} + \tilde{k}_c ) = (\operatorname{Id}   + k_c ) \circ ( \operatorname{Id} + \varphi ),
\end{align*}
where $\varphi =  \psi \circ ( \operatorname{Id} + \tilde{k}_c ) + \tilde{k}_c$ is a bounded function. Likewise, $\operatorname{Id} + \tilde{k}_c$ is invertible, and thus $\operatorname{Id} + \varphi = ( \operatorname{Id}  + k_c)^{-1} \circ ( \operatorname{Id}  + \tilde{k}_c)$ is invertible, as it is the composition of two invertible functions. By \cref{GlobalDiffeomorphism1} its inverse is given by $\operatorname{Id} + \phi$ for some bounded function $\phi$. We infer that
\begin{align*}
F{}{} \circ (K{}{} \circ (\operatorname{Id} + \varphi)) &= K{}{} \circ (A_c + r ) \circ (\operatorname{Id} + \varphi) \\
						&= (K{}{} \circ ( \operatorname{Id} + \varphi) ) \circ ( ( \operatorname{Id} + \phi)  \circ (A_c + r) \circ ( \operatorname{Id} + \varphi) ) \\
						&= \begin{pmatrix} \operatorname{Id} +  \tilde{k}_c \\ k_u \circ ( \operatorname{Id} + \varphi) \\ k_s \circ ( \operatorname{Id} + \varphi) \end{pmatrix} \circ (A_c + \Phi),
\end{align*}
where we define $\Phi = r \circ ( \operatorname{Id} + \varphi) + \phi \circ (A_c + r) \circ ( \operatorname{Id} + \varphi)$. By property \cref{PropertiesR} of \cref{MainTheorem}, we have that $A_c + r$ is invertible, and thus $A_c + \Phi$ is the composition of three invertible functions, hence invertible itself. Therefore, we use \cref{UniqueCenterManifold} to conclude $\tilde{k}_u = k_u \circ (\operatorname{Id} + \varphi)$ and $\tilde{k}_s = k_s \circ (\operatorname{Id} + \varphi)$. Since $\operatorname{Id} + \tilde{k}_c =  (\operatorname{Id} + k_c) \circ  (\operatorname{Id} + \varphi)$, we see that $\tilde{K{}{}} = K{}{} \circ  (\operatorname{Id} + \varphi)$. As $\operatorname{Id} + \varphi$ is invertible, we conclude that
\begin{align*}
\operatorname{Im}(K{}{}) = K{}{}(X_c) = K{}{} ( (\operatorname{Id} + \varphi)(X_c)) = \tilde{K{}{}}(X_c) = \operatorname{Im}(\tilde{K{}{}}). &\qedhere
\end{align*}
\end{proof}

\subsection{Proof of \texorpdfstring{\cref{AlmostSolution}}{main corollary}}\label{ProofofMainCorollary}

Finally, we want to show that if we have found approximations of the center manifold and the center dynamics, i.e., we have found an approximate conjugacy $K{}{}_0$ and approximate dynamical system $R_0$ such that
\begin{align*}
\norm{ F{}{} \circ K{}{}_0 - K{}{}_0 \circ R_0}_m \le \varepsilon,
\end{align*}
then the center manifold and dynamical system lie close to these approximations. 

\begin{proof}[\underline{Proof of \cref{AlmostSolution}}] Let $m < n$ and let $k_0 = \tmatrix{k_{u,0}}{k_{s,0}}{} : X_c \to X_u \oplus X_s$ and $r_0 : X_c \to X_c$. We use $\mathcal{M}$ to denote $\tmatrix{r_0}{k_{u,0}}{k_{s,0}}$ and $\mathcal{F} =  F \circ \tmatrix{k_c}{k_{u,0}}{k_{s,0}} - \tmatrix{k_c}{k_{u,0}}{k_{s,0}} \circ (A_c + r_0)$. We assume that $k_0$ and $r_0$ satisfy the conditions of \cref{AlmostSolution}, that is, there exists constants $M > 0$ and $\varepsilon > 0$ such that
\begin{align*}
 k_0 &\in \left\{ h \in C_b^{m+1}(X_c,X_u \oplus X_s ) \ \middle| \ h(0) = 0,\   Dh(0) = 0  \text{ and } \norm{h}_{m+1} \le M \right\}, \\
 r_0 &\in \left\{ h \in C_b^{m+1}(X_c,X_c ) \ \middle| \ h(0) = 0, \  Dh(0) = 0, \supnorm{Dh} \le L_r \text{ and } \norm{h}_{m+1} \le M \right\}, \\
\mathcal{F} &\in \left\{ h \in C_b^m(X_c, X) \middle| \ \norm{h}_m \le \varepsilon \right\}.
\end{align*}
Our proof consists of the following steps:
\begin{enumerate}
\item We prove that if $\mathcal{F}$ is small in $C^m$, then $\mathcal{M} - \Theta^{}(\mathcal{M})$ is small in $C^m$.
\item We prove that if $\mathcal{M} - \Theta^{}(\mathcal{M})$ is small in $C^0$, then $\mathcal{M} - \Lambda$ is small in $C^0$, where $\Lambda$ is the fixed point of $\Theta^{}$. 
\item Using induction, we prove that if $\mathcal{M} - \Theta^{}(\mathcal{M})$ is small in $C^m$, then $D^m\mathcal{M} - D^m\Lambda$ is small in $C^0$.
\end{enumerate}
For the first step, we want an explicit estimate for $\mathcal{M} - \Theta^{}(\mathcal{M})$. By definition of $\Theta^{}$, see \cref{FixedPointOperator}, we have 
\begin{align*}
\mathcal{M} - \Theta^{}(\mathcal{M}) = \begin{pmatrix} - \mathcal{F}_1 \\ A_u^{-1} \mathcal{F}_2 \\ - \mathcal{F}_3 \circ (A_c + r_0)^{-1}\end{pmatrix}.
\end{align*}
We can clearly estimate the $C^m$ norm of the first two components by $\norm{\mathcal{F}}_m \le \varepsilon$ and $\operatornorm{A_u^{-1}} \norm{\mathcal{F}}_m \le \operatornorm{A_u^{-1}} \varepsilon$ respectively. For the third component, we use Fa\'{a} di Bruno's formula and \cref{HigherDerivativeInverse} to obtain an estimate. We find
\begin{align*}
\norm{\mathcal{F}_3 \circ (A_c + r_0)^{-1}}_m &\le \mathcal{C}\left(D(A_c + r_0)^{-1}, D^i(A_c + r_0) \text{ for } i \le m \right)\norm{\mathcal{F}_3}_m \\
					&\le \mathcal{C}_1(M,m) \varepsilon,
\end{align*}
where we used that the derivatives of $r_0$ are bounded by $M$ and $D(A_c + r_0)^{-1}$ is bounded by $L_{-1}$ as $Dr_0$ is bounded by $L_r$. Hence we obtain
\begin{align}
\norm{\mathcal{M} - \Theta^{} (\mathcal{M})}_m \le \max\{ 1 , \operatornorm{A_u^{-1}}, \mathcal{C}_1(M,m) \} \varepsilon = \mathcal{C}_2(M,m) \varepsilon. \label{NewEq2}
\end{align}
Here $\mathcal{C}_2(M,m) \isdef \max\{ 1 , \operatornorm{A_u^{-1}}, \mathcal{C}_1(M,m) \}$. Hence we have shown that if $\mathcal{F}$ is small, then $\mathcal{M}$ is an almost fixed point of $\Theta^{}$.

To prove the second step, we use that $\Theta^{}$ is a contraction in the $C^0$ norm with contraction constant $\theta_0$, see the proof of \cref{UniqueCenterManifold}, thus we have
\begin{align*}
\supnorm{\Lambda - \mathcal{M}} &\le \supnorm{\Theta^{} ( \Lambda) - \Theta^{}(\mathcal{M})} + \supnorm{ \mathcal{M} - \Theta^{}(\mathcal{M})} \\
			&\le \theta_0 \supnorm{\Lambda - \mathcal{M}} + \supnorm{ \mathcal{M} - \Theta^{}(\mathcal{M})}.
\end{align*}
By rewriting, we obtain the estimate
\begin{align}
\supnorm{\Lambda - \mathcal{M}} \le \frac{1}{1- \theta_0} \supnorm{ \mathcal{M} - \Theta^{}(\mathcal{M})}. \label{NewEq1}
\end{align}
Combining \cref{NewEq1,NewEq2} proves \cref{AlmostSolution} for $m=0$, that is
\begin{align*}
\norm{\Lambda - \mathcal{M}}_0 \le \frac{\mathcal{C}_2(M,0) }{1- \theta_0} \varepsilon \isfed \mathcal{C}(M,0) \varepsilon.
\end{align*}

To prove step 3, we will use induction. So let us assume that $\norm{\Lambda - \mathcal{M}}_{m-1} \le \mathcal{C}(M,m-1) \varepsilon$. By construction of the contraction $\Theta^{[m+1]}$, its fixed point is $D^m \Lambda$. Thus similarly to \cref{NewEq1}, we have
\begin{align}
\supnorm{ D^m \Lambda - D^m \mathcal{M}} \le \frac{1}{1 - \theta_m} \supnorm{ D^m \mathcal{M} - \Theta^{[m+1]}(D^m \mathcal{M})}. \label{NewEq4}
\end{align}
From \cref{NewEq2}, we know that $\supnorm{D^m \left( \mathcal{M} - \Theta^{}(\mathcal{M}) \right)} \le \mathcal{C}_1(M,m)\varepsilon$. Hence we estimate \cref{NewEq4} with the triangle inequality and the estimate on $D^m \left( \mathcal{M} - \Theta^{}(\mathcal{M}) \right)$ to obtain
\begin{align}
\supnorm{ D^m \Lambda - D^m \mathcal{M}} \le \frac{\mathcal{C}_2(M,m) \varepsilon}{1 - \theta_m} + \frac{ \supnorm{D^m \Theta^{}(\mathcal{M}) - \Theta^{[m+1]}(D^m \mathcal{M})} }{1- \theta_m}. \label{NewEq3}
\end{align}
Therefore, it remains to find a bound on $D^m \Theta^{}(\mathcal{M}) - \Theta^{[m+1]}(D^m \mathcal{M})$. Recall from \cref{FixedPointOperatorEquation} the definition of $\Theta^{}$. By Fa\'{a} di Bruno's formula, there exists a function $\mathcal{G}$ such that for all $\mathcal{T} :X_c \to X$
\begin{align}
D^m \Theta^{}(\mathcal{T}) = \mathcal{G}(\mathcal{T}, D\mathcal{T}, \dots, D^m \mathcal{T}). \label{NewEq5}
\end{align}
By construction, $\mathcal{G}(\mathcal{T}, \dots, D^m \mathcal{T})(x)$ is a linear combination of products of various derivatives of $\mathcal{T}$, evaluated at either $x$ or $\mathcal{T}(x)$.
In particular, we have for $\mathcal{M} : X_c \to X$
\begin{align}
D^m \Theta^{}(\mathcal{M}) &= \mathcal{G}(\mathcal{M}, D \mathcal{M}, \dots , D^{m-1} \mathcal{M}, D^m \mathcal{M} ).  \label{NewEq6}
\intertext{On the other hand, by  definition of $\Theta^{[m+1]}$, see \cref{FixedPointOperatorCmEquation}, we also have that}
\Theta^{[m+1]}(D^m \mathcal{M}) &= \mathcal{G}(\Lambda, D \Lambda, \dots , D^{m-1} \Lambda, D^m \mathcal{M} ).  \label{NewEq7}
\end{align}
Subtracting \cref{NewEq7} from \cref{NewEq6} we thus obtain 
\begin{align}
D^m \Theta^{}(\mathcal{M}) -  \Theta^{[m+1]}(D^m \mathcal{M}) &= \int_0^1 D \mathcal{G}(\mathcal{N}(s), D^m \mathcal{M}) \text{d}s \begin{pmatrix} \Lambda - \mathcal{M} \\ \vdots \\ D^{m-1} \left( \Lambda - \mathcal{M} \right) \\ 0 \end{pmatrix}, \label{NewEq10}
\end{align}
where $\mathcal{N}(s) =(1-s) \left(\mathcal{M} , \dots, D^{m-1} \mathcal{M} \right) + s \left(\Lambda, \dots,  D^{m-1} \Lambda \right)$. 

To calculate the partial derivative of $\mathcal{G}$ in the direction of its first input, we use the chain rule and obtain an expression depending on $\mathcal{G}$, $\mathcal{N}(s)$, $D \mathcal{N}(s)$, $D^m \mathcal{M}$ and $D^{m+1} \mathcal{M}$. The other partial derivatives of $\mathcal{G}$ are partial derivatives of polynomials, hence only depend on $\mathcal{G}$, $\mathcal{N}(s)$ and $D^m \mathcal{M}$. In particular, $D\mathcal{G}$ is continuous and evaluated on the compact set $\left\{ (\mathcal{N}(s), D^m \mathcal{M}) \mid s \in [0,1] \right\}$ in \cref{NewEq10} -- note that $\Lambda, \mathcal{M} \in C_b^{m+1}(X_c, X)$. Thus $\operatornorm{D\mathcal{G}(\mathcal{N}(s), D^m\mathcal{M})} \le \mathcal{C}_3(M,m)$ for all $s \in [0,1]$. Therefore, we can estimate \cref{NewEq10} by
\begin{align*}
\supnorm{D^m \Theta^{}(\mathcal{M}) -  \Theta^{[m+1]}(D^m \mathcal{M})} &\le \mathcal{C}_3(M,m) \norm{\Lambda - \mathcal{M}}_{m-1} \\
			&\le \mathcal{C}_3(M,m) \mathcal{C}(M,m-1) \varepsilon.
\end{align*}
Using \cref{NewEq3} and our induction hypothesis, we thus conclude that 
\begin{align*}
\norm{\Lambda - \mathcal{M}}_m &= \max\left\{ \norm{\Lambda - \mathcal{M}}_{m-1}, \supnorm{D^m\Lambda - D^m \mathcal{M}} \right\} \\
					&\le \max\left\{ \mathcal{C}(M,m-1) , \frac{\mathcal{C}_2(M,m) +  \mathcal{C}_3(M,m) \mathcal{C}(M,m-1) }{1 - \theta_m} \right\} \varepsilon\\
					&\isfed \mathcal{C}(M,m) \varepsilon. \qedhere
\end{align*}
\end{proof}

\bibliographystyle{abbrv}
\bibliography{Bib} 

\end{document}